\documentclass[12pt]{amsart}

\usepackage[left=3cm,marginpar=2.5cm,tmargin=2cm,bmargin=2.5cm]{geometry}
\usepackage{thmtools}
\newtheorem{theorem}{Theorem}[section]

\newtheorem{lemma}[theorem]{Lemma}
\newtheorem{proposition}[theorem]{Proposition}

\usepackage{amsfonts, amsthm, amssymb, amsmath}
\usepackage{natbib}
\usepackage{url}
\usepackage{mathtools}
\usepackage{tikz-cd}

\theoremstyle{definition}
\newtheorem{defi}[theorem]{Definition}
\newtheorem{definition}[theorem]{Definition}
\theoremstyle{remark}
\newtheorem{remark}[theorem]{Remark}

\def \Sym {\operatorname{Sym}}
\def \Sing {\operatorname{Sing}}
\def \Gal {\operatorname{Gal}}
\def \Bun {\operatorname{Bun}}
\def \Ext {\operatorname{Ext}}

\def \sgn {\operatorname{sgn}}
\def \Ind {\operatorname{Ind}}

\def \rank {\operatorname{rank}}
\def \tr {\operatorname{tr}}
\def \div {\operatorname{div}}

\def \Frob {\operatorname{Frob}}
\def \ad {\operatorname{ad}}

\def \bfg {{\mathbf g}}
\def \bfh {{\mathbf h}}

\begin{document}

\begin{abstract} The sup-norm problem in analytic number theory asks for the largest value taken by a given automorphic form. We observe that the function-field version of this problem can be reduced to the geometric problem of finding the largest dimension of the $i$th stalk cohomology group of a given Hecke eigensheaf at any point. This problem, in turn, can be reduced to the intersection-theoretic problem of bounding the ``polar multiplicities" of the characteristic cycle of the Hecke eigensheaf, which in known cases is the nilpotent cone of the moduli space of Higgs bundles. We solve this problem for newforms on $GL_2 (\mathbb A_{\mathbb F_q(T)})$ of squarefree level, leading to bounds on the sup-norm that are stronger than what is known in the analogous problem for newforms on $GL_2(\mathbb A_{\mathbb Q})$  (i.e. classical holomorphic and Maa{\ss} modular forms.) \end{abstract}

\title[A geometric approach to sup-norms]{A geometric approach to the sup-norm problem for automorphic forms: the case of newforms on $GL_2(\mathbb F_q(T))$ with squarefree level}
\author{Will Sawin}
\address{Department of Mathematics \\ Columbia University \\ New York, NY 10027}
\email{sawin@math.columbia.edu}

\maketitle

\tableofcontents

\section{Introduction}

The sup-norm problem in the analytic number theory of automorphic forms studies the sup-norms of automorphic forms when viewed as functions on locally symmetric spaces or on adelic groups. In this paper, we prove upper bounds on the sup-norms of certain automorphic forms over function fields:

\begin{theorem}\label{sup-norm-intro} Let $ F = \mathbb F_q(T)$, let $N$ be a squarefree effective divisor on $\mathbb P^1$, and let $f: GL_2(\mathbb A_F) \to \mathbb C$ be a cuspidal newform of level $N$ with unitary central character. Assume that for each place $v$ in the support of $N$, the restriction of the central character of $f$ to $\mathbb F_q^\times \subset F_v^\times$ is trivial. Then 
\[ ||f||_{\infty} = O \left(  \left(\frac{ 2 \sqrt{q} +2}{ \sqrt{ 2 \sqrt{q}+ 1}}\right)^{ \deg N} \right)\] 
if $f$ is  Whittaker normalized and
\[ ||f||_{\infty} =  O \left(  \left(\frac{ 2  (1+ q^{-1/2} ) }{\sqrt{ 2 \sqrt{q}+ 1}}\right)^{ \deg N} \log(\deg N)^{3/2} \right) \]
if $f$ is $L^2$-normalized. \end{theorem}

We leave the definitions of most of these terms until Subsection \ref{ss-fourier}, while assuring readers familiar with classical modular forms or automorphic forms over number fields that the definitions are very similar to the corresponding definitions in those settings. (In the adelic language, they can be written almost identically).

The condition on the central character is satisfied automatically if $N$ is prime or the central character is trivial.

Theorem \ref{sup-norm-intro} can be compared to known bounds for the sup-norm of classical or Maa{\ss} modular forms of squarefree level $N$, which all take the form of a power of $N$ times a lower-order term. This comparison is not just an analogy - we expect that analogues of all the classical analytic results we discuss below can be proven by the same method, obtaining the same exponents, in the function field setting. To make sense of this, one should note that powers of $N$ classically correspond to powers of the norm $|N| = q^{\deg N}$ in the function field setting. Because Theorem \ref{sup-norm-intro} is an exponential in $\deg N$, we can express it as a power of $|N|$, with the exponent depending on $q$, times a lower-order term. In particular, as $q$ goes to $\infty$, the exponent will approach $-\frac{1}{4}$. (This is in the $L^2$ normalization, which we will use for all our comparisons.) 

We first compare to the local bound, which in the classical setting is $N^\epsilon$. In the function field setting, the local bound is particularly simple, because the double coset space $GL_2(F) \backslash GL_2(\mathbb A_F)/ \Gamma_1(N)$ is a discrete space, where each point has positive measure, and so the value of a function with $L^2$-norm $1$ at a point is at most the inverse-square-root of the measure. One can check that cusp forms vanish at any point with measure less than $1/(q+1)$, so the local bound is $\sqrt{q+1}$ - essentially, constant. Theorem \ref{sup-norm-intro} will improve on this, obtaining a bound decreasing with $\deg N$, as soon as $q> 9$.

In the classical setting, there are two main methods, both purely analytic, to improve on the local bound for the sup-norm problem. These are the methods based on the analysis of the Whittaker expansion and the amplification method. We discuss the Whittaker method first as it is a closer analogue of our proof, which also uses the Whittaker expansion. (Despite this, the general geometric framework we discuss below, of which our proof is a part, suggests that a similar method could be used in settings where the Whittaker expansion doesn't exist, as long as the appropriate Hecke eigensheaves can be constructed.) 

The strongest results obtained using the Whittaker expansion are in situations somewhat different from squarefree level $N$. For instance, \citet{TemplierLowerBounds} proved a lower bound of $N^{-1/4-\epsilon}$ for the sup-norm when $N$ is a square and the conductor of the central character is $N$. In the same case, \citet{Comtat} proved an upper bound of $N^{\epsilon-1/4}$. It is interesting that the exponent $1/4$, which is also the large $q$ limit of our bounds, appears often away from the squarefree case.  It is possible that our method could be generalized to hold regardless of the level, which would provide a better explanation of this, as $1/4$ is the best exponent possible for a completely general result without contradicting \citep{TemplierLowerBounds}.

One can also compare with the results of \citet{Xia}, for holomorphic modular forms of fixed level and varying weight $k$, which were a matching lower bound of $k^{1/4-\epsilon}$ and upper bound of $k^{1/4+\epsilon}$, which because the local bound in this setting is $k^{1/2}$, represents a similar exponent improvement. However, this is not as close an analogue of Theorem \ref{sup-norm-intro} because it varies the weight and not the level. \citet*{MinimalType} suggested that the analogy between the weight and level aspects could be made closer by using the ``minimal type vectors" they define, instead of newforms, on the level side, so a better comparison might be obtained after studying these vectors in the function field setting.

Using the amplification method, \citet{HarcosTemplier} proved bounds of the form $N^{\epsilon -1/6}$ for modular forms of squarefree level $N$, building on a series of works by multiple authors with successively improving exponents. Theorem \ref{sup-norm-intro} gives a better exponent as long as 
\[ \frac{ 2 (1+ 1/\sqrt{q}) }{ \sqrt{ 2 \sqrt{q}+ 1}}< q^{-1/6} \] which occurs for $q>134$.

In the classical setting, with squarefree level, this has not been improved. For other levels, stronger bounds are known with the amplification method as well. Combining both amplification and estimates for Whittaker coefficients, \citet{SahaPowerful2} obtained a bound of $N^{\epsilon-1/4}$ in the limit where $N$ becomes more powerful while the conductor of the central character divides $\sqrt{N}$.  \citet{HuSaha} obtained a bound of $N^{\epsilon - 7/24}$ by the amplification method for forms on division algebras, with level a high power of a small prime, where the conductor of the central character is not too large. It is likely possible to obtain a similar statement for modular forms, with the same level condition, by the same method. This exponent beats $N^{-1/4}$, and doing this well by our method would require further geometric ideas.

Before we explain the proof, we observe that the proof can be viewed as proceeding via a geometric problem that may be of independent interest. For this reason, we introduce the geometric problem, and a natural approach to it, first. We then explain how to modify it to produce our method for the sup-norm problem.

From a purely analytic perspective, the most interesting feature of the proof might be the way that the Theorem \ref{sup-norm-intro} follows from a series of bounds (Lemmas \ref{cusp-newform-bound}, \ref{Atkin-Lehner-bound}, and \ref{local-squarefree-bound}) for the value of the form $f$ at a point, that depend in an intricate way on the geometry of the point (specifically, its distance to cusps and ``virtual cusps"). In the amplification method, by contrast, the bound at a given point depends on some lattice point counts. We do not know to what extent these are related to virtual cusps. Because of the nature of our proof, how this local bound varies from point to point has some meaning for the geometry of the ``modular curve", being directly related to the characteristic cycle of this space, which is independent of the choice of modular form.  This geometric perspective suggests that it may be fruitful for some purposes to reformulate the sup-norm problem as studying $\sup_f |f(x)|$ for points $x$  instead of $\sup_x |f(x)|$ for eigenforms $f$.

\subsection{Geometric Langlands and the general sup-norm problem}

Let $C$ be a smooth proper geometrically connected curve over a field $k$, $G$ an algebraic group over $k$, and $\Bun_G$ the moduli space of $G$-bundles on $C$. Let $\mathcal F$ be a Hecke eigensheaf on $\Bun_G$, i.e. an irreducible perverse sheaf satisfying the conditions \cite[5.4.2]{BeilinsonDrinfeld} studied in the geometric Langlands program. We can ask the following interrelated set of questions about $\mathcal F$:

\begin{enumerate}

\item For a point $x \in \Bun_G(k)$, how large is the stalk dimension $\dim \mathcal H^i(\mathcal F)_x$ for each integer $i$? In particular, for which $i$ does the stalk cohomology vanish?

\item Assume $\mathcal F$ is a pure perverse sheaf. (Either using the weights of Frobenius over a finite field or an abstract weight filtration as in the theory of mixed Hodge modules).  For a point $x \in \Bun_G(k)$, how large is the sum over $i$ of the dimension of the weight $w$ graded piece of $H^i(\mathcal F)_x$ for each integer $w$? In particular, for which $w$ does the weight $w$ part vanish?

\item Assume $k$ is a finite field $\mathbb F_q$. How large is $\sum_i (-1)^i \operatorname{tr} (\operatorname{Frob}_q, \mathcal H^i(\mathcal F)_x)$?

\end{enumerate}

Bounds for question (1) imply corresponding bounds for question (2) because the $i$th stalk cohomology of a perverse sheaf pure of weight $w$ is mixed of weight $\leq w+i$. Bounds for question (2) imply bounds for question (3) by the definition of weights of Frobenius.

The general sup-norm problem in analytic number theory asks for the maximum value, or the maximum value on some region, taken by a (cuspidal) Hecke eigenform. For $\mathcal F$ a Hecke eigensheaf, $x \mapsto \sum_i (-1)^i \operatorname{tr} (\operatorname{Frob}_q, \mathcal H^i(\mathcal F)_x)$ is a Hecke eigenform, so question (3) is a special case of the function field version of the classical sup-norm problem. It is equivalent to the full sup-norm problem in cases where we know every cuspidal Hecke eigenform comes from a Hecke eigensheaf, as in the case of $GL_n$ by combining the main results of \citep{Lafforgue} and \citep*{FrenkelGaitsgoryVillonen}.  

Question (1) also may have interest as a purely geometric problem.

Massey proved bounds for the dimensions of the stalk cohomology groups of a perverse sheaf in terms of the ``polar multiplicities" of its characteristic cycle. In \citep{mypaper} these were generalized from characteristic zero to characteristic $p$. Because $\mathcal F$ is perverse, we can use these generalizations to give bounds for the three questions above if we can calculate the characteristic of $\mathcal F$. This characteristic cycle lies in the cotangent bundle of $\Bun_G$, which is also the moduli space of Higgs bundles.

 The characteristic cycle of the Hecke eigensheaves was first studied by Laumon. Laumon predicted that, in the $GL_n$ case, the characteristic cycle should be contained in the nilpotent cone of the moduli space of Higgs bundles \cite[Conjecture 6.3.1]{Laumon}. (He stated this only over characteristic zero, but now that the characteristic cycle is known to exist in characteristic $p$ \citep{saito1}, we can extend the conjecture there as well.) This was by analogy to Lusztig's theory of character sheaves. He verified this in the case of $GL_2$ using Drinfeld's explicit construction \cite[Proposition 5.5.1]{Laumon}.
 
 Beilinson and Drinfeld constructed Hecke eigensheaves, in the form of $D$-modules, associated to special local systems known as opers. They calculated the characteristic cycle of their Hecke eigensheaves especially, and found that it was equal to the nilpotent cone. More precisely, it is equal as a cycle to the zero fiber of the Hitchin fibration \cite[Proposition 5.1.2(ii)]{BeilinsonDrinfeld}.\footnote{I learned the information in this paragraph and the previous one from user:t3suji on mathoverflow}.
 
 It is reasonable to expect that this formula for the characteristic cycle holds for eigensheaves on $\Bun_G$ in arbitrary characteristic, at least when the centralizer of the associated local system is the center of the Langlands dual group. Whenever this is proved, we can bound the stalk cohomology by the polar multiplicities of the zero fiber of the Hitchin fibration. In particular we have the following vanishing result:
 
\begin{proposition}[Proposition \ref{a-c-v}]\label{automorphic-cohomology-vanishing} Let $\mathcal F$ be a perverse sheaf on $\Bun_G$ whose characteristic cycle is contained in the nilpotent cone of the moduli space of Higgs bundles. Then for a $G$-bundle $\alpha$ on $C$, $\mathcal H^i(\mathcal F)_{\alpha} $ vanishes for \[i > \dim \{ v \in H^0 ( C, \operatorname{ad} (\alpha) \otimes K_C) | v \textrm { nilpotent} \} - (g-1) \dim G. \] \end{proposition}

Proposition \ref{automorphic-cohomology-vanishing} applies to the category of perverse sheaves with characteristic cycle contained in the nilpotent cone of the moduli space of Higgs bundles, which is also studied in the Betti geometric Langlands program.
 
We now explain how the proof of Theorem \ref{sup-norm-intro} differs from this setup.

It is not clear how to calculate the polar multiplicities exactly in this level of generality, which would be required to get numerical bounds for the sup-norm problem. The main difficulty in computing the polar multiplicities for a general $C$ and $G$ is the potentially complicated geometric structure of the moduli of nilpotent Higgs fields on a given vector bundle. To make this as simple as possible, we have chosen to work with $C = \mathbb P^1$ and $G= GL_2$. Because the canonical bundle is negative, all Higgs fields must preserve the Harder-Narasimhan filtration, and all such fields are nilpotent, so here the moduli of nilpotent fields for a given vector bundle is simply a vector space. This allows us to calculate the polar multiplicities of the nilptotent cone in this setting.
 
Because there are no cusp forms of level $1$ on $\mathbb P^1$, we have chosen to work with nontrivial level structure. Working with newforms, there is an appropriate analogue of the nilpotent cone, which is not much more complicated - the polar multiplicities are just sums of the polar multiplicities from the unramified case.

Finally, we do not work with the Hecke eigensheaf, but rather with the Whittaker model of it. Recall that the key difficulty in the construction of Hecke eigensheaves studied by \citet{Drinfeld}, \citet{Laumon}, and \citet*{FrenkelGaitsgoryVillonen} is the descent of an explicit perverse sheaf from some covering $\Bun_n'$ of $\Bun_n$ to $\Bun_n$. With regards to applications to the sup-norm problem, it is no loss to work on $\Bun_n'$, because the numerical function we are computing is the same in each case. This carries two advantages. First, we can avoid the descent step, and therefore work in greater generality than \cite{Drinfeld}. Second, the polar multiplicities are often smaller on this covering than the base, and so we get better bounds this way.  

Because we are working on a covering, $\Bun_G$, the moduli of Higgs bundles, and the nilpotent cone almost never appear explicitly in our proof, replaced by the moduli space of extensions of two fixed line bundles, its cotangent bundle, and a certain explicit cycle in that cotangent bundle. Despite this, the fundamental idea is the same.

We calculate the polar multiplicities precisely, obtaining a bound for each point in $\Bun_2'$ over our chosen point of $\Bun_2$. To get the best possible bound we choose the optimal point of $\Bun_2'$. Roughly speaking, our bound consists of contributions from different cusps, that grow larger as we get closer to the cusps, but we do not have to count the cusp we are performing a Whittaker expansion around. To get a good bound, we need to perform a Whittaker expansion around whichever cusp the point is closest to. Because we have only done our geometric calculations for the expansion at the standard cusp, we transform an arbitrary cusp into the standard cusp using Atkin-Lehner operators, which requires $N$ to be squarefree and leads to our condition on the central character.  

 \subsection{Notation for automorphic forms on $GL_2$}\label{ss-fourier}
 
 Let us now explain the notation, and basic theory, needed to understand the statement of Theorem \ref{sup-norm-intro}.

Let $C$ be a smooth projective geometrically irreducible curve over $\mathbb F_q$ and let  $F= \mathbb F_q(C)$ its field of rational functions.

We can identify the set of places of $F$ with the set of closed points $|C|$ of $C$. For $v$ a place of $F$, let $F_v$ be the completion of $F$ at $v$, $\mathcal O_{F_v}$ the ring of integers of $F_v$, $\pi_v$ a uniformizer of $\mathcal O_{F_v}$, and $\kappa_v = \mathcal O_{F_v}/\pi_v$ the residue field at $v$. Let $\mathbb A_F = \prod'_v F_v$ be the adeles of $F$.

A divisor on $C$ is a finite $\mathbb Z$-linear combination of closed points of $C$. The degree of the divisor is the corresponding $\mathbb Z$-linear combinations of the degrees of its points. For an adele $a$, we define the divisor $\div a = \sum_v  v(a)[v]$, and $\deg a = \deg \div a$. Using this convention, adeles contained in $\prod_v \mathcal O_{F_v}$ have divisors that are effective and degrees that are nonnegative.

 Fix $N$ an effective divisor on $C$. Write $N = \sum_{v \in |N|} c_v [v]$ for some set $|N|$ of places $v$ of $C$ and some positive integer multiplicities $c_v$.  We say $N$ is squarefree if the multiplicities $c_v$ are all at most $1$. We always take $c_v=0$ if $v$ is not in the support of $N$.
 
 We will always use $v$ to refer to places of $C$, or closed points, and $x$ to refer to points of $C(\overline{\mathbb F}_q)$. So we will also write $N= \sum_{x \in |N|} c_x [x]$, where $c_x$ is the multiplicity of the $\overline{\mathbb F}_q$-point $x$ (which equals the multiplicity of the closed point $v$ that $x$ lies over.)

\begin{defi} An \emph{automorphic form of level $N$} on $GL_2(\mathbb A_F)$ is a function $f: GL_2(\mathbb A_F) \to \mathbb C$ which is left invariant under $GL_2(F)$ and right invariant under \[\Gamma_1(N) =  \prod_{v \in |C| }\left\{  \begin{pmatrix} a & b \\ c & d \end{pmatrix}  \in GL_2 ( \mathcal O_{F_v} ) \mid  c \equiv 0\mod \pi_v^{c_v} , d \equiv 1 \mod \pi_v^{c_v} \right \}.\]

We say that $f$ is \emph{cuspidal} if $\int_{z \in \mathbb A_F/F} f\left( \begin{pmatrix} 1& z \\ 0 & 1 \end{pmatrix} {\mathbf g} \right) dz=0$ for all $ {\mathbf g} \in GL_2(\mathbb A_F)$, for the Haar measure on $\mathbb A_F/F$ that assigns the ring of integers (say) mass one.

We say $f$ is a \emph{Hecke eigenform} if it is an eigenfunction of the two standard Hecke operators at each place $v$ of $C$ not in the support of $N$.

We say $f$ is a \emph{newform} if it is a Hecke eigenform and the same set of Hecke eigenvalues is not shared by any Hecke eigenform of level $N'<N$.  \end{defi}

For any eigenform, there is a unique character $\eta: (\mathbb A_F^\times / F^\times) \to \mathbb C^\times$ such that $f ( t^{-1} {\mathbf g} ) = \eta(t) f({\mathbf g})$ for any scalar $t \in \mathbb A_F^\times$. We call $\eta$ the central character of $f$.

Fix a meromorphic $1$-form $\omega_0$ on $C$ and a character $\psi_0: \mathbb F_q \to \mathbb C^\times$. Define a character $\psi: \mathbb A_F / F \to \mathbb C^\times$ by $\psi(z) = \psi_0( \langle z, \omega_0\rangle)$ where $\langle, \rangle$ is the residue pairing. 

For $\mathcal F$ a middle-extension $\overline{\mathbb Q}_\ell$-sheaf on $C$, let $r_{\mathcal F}(D)$ be the unique function from effective divisors $D$ to $\overline{\mathbb Q}_\ell$ satisfying $r_{\mathcal F}(D_1 + D_2) = r_{\mathcal F}(D_1) r_{\mathcal F}(D_2)$ if $D_1$ and $D_2$ are relatively prime and, for $v$ a closed point of $C$, \[\sum_{n=0}^\infty r_{\mathcal F} (n[v])  u^n = \frac{1}{ \det ( 1 - u  \operatorname{Frob}_{ |\kappa_v|}, \mathcal F_v)}\] where $\mathcal F_v$ is the stalk of $\mathcal F$ at some geometric point lying over $v$. 

The following Whittaker expansion is essentially due to \citet[(4)]{Drinfeld}.

\begin{lemma}[Lemma \ref{Drinfeld-formula}]\label{Drinfeld-formula-intro}For any newform $f$ of level $N$ whose central character has finite order, there exists $\mathcal F$ an irreducible middle extension sheaf of rank two on $C$, pure of weight $0$, of conductor $N$, and $C_f\in \mathbb C$ such that

\[ f \left( \begin{pmatrix} a & bz \\ 0 & b \end{pmatrix} \right) =C_f  q^{- \frac{ \deg (\omega_0 a/b) }{2} } \eta(b)^{-1} \sum_{\substack{ w \in F^{\times} \\ \operatorname{div} (w \omega_0 a/b ) \geq 0}}  \psi(wz ) r_{\mathcal F} (\operatorname{div}(  w \omega_0 a/b ))  \] for all $a,b,z \in \mathbb A_F$.

$\mathcal F$ and $C_f$ are unique with this property.\end{lemma}

 In the future, we will use the notation $\mathcal F$ and $C_f$ for the unique $\mathcal F$ and $C_f$ of Lemma \ref{Drinfeld-formula}.

\begin{defi} We say that $f$ is \emph{Whittaker normalized} if $\eta$ has finite order and $C_f=1$.

Let $\mu$ be the unique invariant measure on $PGL_2(F) \backslash PGL_2(\mathbb A_F)$ that assigns measure $1$ to $P \Gamma_1(N)$. We say that $f$ is \emph{$L^2$-normalized} if $\eta$ is unitary and
\[ \int_{PGL_2(F)  \backslash PGL_2(\mathbb A_F)  } |f(\bfg) |^2 d \mu(g)  =1.\]

Here we note that $|f(\bfg)|$ is a well-defined function on $PGL_2(F) \backslash PGL_2(\mathbb A_F)$ because $f$ has unitary central character. \end{defi}

\subsection{Plan of paper}

We now describe the strategy of proof in more detail. This is based on the formula of Lemma \ref{Drinfeld-formula-intro}, so the first few sections of the paper are devoted to proving a bound for the right side of the formula (and thereby deducing a bound for the left side).

The first step is constructing a sheaf on the space of effective divisors whose trace function is $r_{\mathcal F}$, and calculating its characteristic cycle. In fact, in Section \ref{symmetric-powers}, we calculate the characteristic cycle of a much more general class of sheaves, which includes sheaves whose trace functions are many other functions of interest in function field number theory, and thus should have further applications - see Theorem \ref{characteristic-cycle-sym}.  

The set of  $\{w \in F\mid \operatorname{div} (w \omega_0 a/b ) \geq 0\}$ is the global sections of the line bundle $\mathcal O_C ( \operatorname{div} (\omega_0 a/b))$. In the special case $C = \mathbb P^1$, this line bundle must be $\mathcal O(n)$ for some $n$, and so the set of $w$ is $H^0 ( \mathbb P^1, \mathcal O(n))$. Then $\psi(wz)$ is the additive character $\psi_0$ composed with a linear form $w \mapsto \langle wz, \omega_0 \rangle$ on $H^0 ( \mathbb P^1, \mathcal O(n))$, and the sum on the right side of Lemma \ref{Drinfeld-formula-intro} depends on $z$ only via this linear form. The next step, in Section \ref{geometric-setup}, is defining a pure perverse sheaf on the projective space parameterizing linear forms on $H^0 ( \mathbb P^1, \mathcal O(n))$ whose trace function is this sum.

In Section \ref{calculating-characteristic-cycle} we calculate the characteristic cycle of the perverse sheaf defined in Section \ref{geometric-setup}. This relies on the characteristic cycle computations of Section \ref{symmetric-powers}. Using these, we calculate the polar multiplicities of this perverse sheaf, in Section \ref{calculating-polar-multiplicities}. This culminates in Lemma \ref{first-newform-bound}, which gives a bound for the quantity considered in Lemma \ref{Drinfeld-formula-intro}.

It is always possible to put an element of $GL_2(\mathbb A_F)$ into upper-triangular form by multiplying on the left by an element of $GL_2(F)$ and on the right by an element of $\Gamma_1(N)$. So this, combined with the invariance properties of $f$, gives a bound for the value of $f$ at any point. In fact, there can be many ways to do this, which give different bounds after applying Lemma \ref{first-newform-bound}. This raises the question of which way  to do this gives the best bound. The analogous problem in the classical setting is which cusp should we take the Fourier expansion of a modular form around to give the best bound for its value at a particular point. Section \ref{heights-virtual-cusps} gives the tools necessary to answer this question. In this section, we define a height function that measures how close an element of  $GL_2(\mathbb A_F)$ is to a particular cusp. In Lemma \ref{cusp-newform-bound} we express the bound of Lemma \ref{first-newform-bound} in terms of a sum over these cusps of a contribution that grows larger as the point gets closer to that cusp, but without a contribution from the cusp we take the Fourier expansion around. In particular, this suggests it should be possible to take the Fourier expansion over the closest cusp.

However, we have only set things up to allow for Fourier expansion over conjugates under $GL_2(F)$ of the standard cusp. In Section \ref{Atkin-Lehner}, we use Atkin-Lehner operators to transform other cusps into the standard cusp. This culminates Lemma \ref{Atkin-Lehner-bound}, which gives a bound for the value of $f$ at any point by a somewhat complicated expression. To prove the main theorem, it then suffices to control the largest possible value of the expression. Using the properties of the height function established in Section \ref{heights-virtual-cusps}, this is a purely combinatorial problem, which we solve in the remainder of Section \ref{Atkin-Lehner}.

\subsection{Acknowledgments} I would like to thank Simon Marshall, Farell Brumley, Paul Nelson, and Ahbishek Saha for helpful discussions on the sup-norm problem, Takeshi Saito for helpful discussions on the characteristic cycle, and the three helpful referees for many helpful comments on an earlier version of this paper.

This research was conducted during the period the author was supported by Dr. Max R\"{o}ssler, the Walter Haefner Foundation and the ETH Zurich Foundation, and, later, during the period the author served as a Clay Research Fellow.

\section{Characteristic cycles of natural sheaves on the symmetric power of a curve}\label{symmetric-powers}

Let us review the definitions of the characteristic cycle and singular support, and some related definitions, from  \cite{Beilinson} and \cite{saito1}. Afterwards, we will define a large class of sheaves, including in particular the sheaves we will study in the remainder of the paper, and compute their characteristic cycles. Throughout this section, we will work over a fixed perfect field $k$.

\begin{definition}\label{C-transversal-1} \cite[Definition 3.5(1)]{saito1} Let $X$ be a smooth scheme over $k$ and let $C \subseteq T^* X$ be a closed conical subset of the cotangent bundle. Let $f: X\to Y$ be a morphism of smooth schemes over $k$.

We say that $f: X \to Y$ is \emph{$C$-transversal} if the inverse image $df^{-1}(C)$ by the canonical morphism $X \times_Y T^* Y \to T^* X$ is a subset of the zero-section $X \subseteq X \times_Y T^* Y$.
\end{definition}

%
%
%

\begin{definition}\cite[Definition 3.1]{saito1} Let $X$ be a smooth scheme over $k$ and let $C \subseteq T^* X$ be a closed conical subset of the cotangent bundle. Let $h: W \to X$ be a morphism of smooth schemes over $k$.

Let $h^* C$ be the pullback of $C$ from $T^* X$ to $W \times_X T^* X$ and let $K$ be the inverse image of the $0$-section $W \subseteq T^* W$ by the canonical morphism $dh: W \times_X T^* X \to T^* W$. 

We say that $h: W\to X$ is \emph{$C$-transversal} if the intersection $h^* C \cap K$ is a subset of the zero-section $W \subseteq W \times_X T^* X$.

If $h: W \to X$ is $C$-transversal, we define a closed conical subset $h^\circ C \subseteq T^* W$ as the image of $h^* C$ under $dh$ (it is closed by \cite[Lemma 3.1]{saito1}). \end{definition}

\begin{definition}\cite[Definition 3.5(2)]{saito1} We say that a pair of morphisms $h: W \to X$ and $f:W\to Y$ of smooth schemes over $k$ is \emph{$C$-transversal}, for $C \subseteq T^* X$ a closed conical subset of the cotangent bundle, if $h$ is $C$-transversal and $f$ is $h^\circ C$-transversal. \end{definition}

\begin{definition}\cite[1.3]{Beilinson} For $K \in D^b_c(X, \mathbb F_\ell)$, let the \emph{singular support $SS(K)$ of $K$} be the smallest closed conical subset $C \in T^* X$ such that for every $C$-transversal pair $h: W \to X$ and $f: W\to Y$, the morphism $f: W\to Y$ is locally acyclic relative to $h^* K$. 

\end{definition}

\begin{definition} \cite[Definition 5.3(1)]{saito1} Let $X$ be a smooth scheme of dimension $n$ over $k$  and let $C\subseteq T^* X$ be a closed conical subset of the cotangent bundle. Let $Y$ be a smooth curve over $k$ and $f: X\to Y$ a morphism over $k$. 

We say a closed point $x \in X$ is at most an \emph{isolated $C$-characteristic point} of $f$ if $f$ is $C$-transversal when restricted to some open neighborhood of $x$ in $X$, minus $x$.  We say that $x\in X$ is an \emph{isolated $C$-characteristic point} of $f$ if this holds, but $f$ is not $C$-transversal when restricted to any open neighborhood of $X$. \end{definition}

\begin{definition} For $V$ a representation of the Galois group of a local field over $\mathbb F_\ell$ (or a continuous $\ell$-adic representation), we define $\operatorname{dimtot} V$ to be the dimension of $V$ plus the Swan conductor of $V$. For a complex $W$ of such representations, we define $\operatorname{dimtot}W $ to be the alternating sum $\sum_i (-1)^i \operatorname{dimtot} \mathcal H^i(W)$ of the total dimensions of its cohomology objects. \end{definition}

\begin{definition}\cite[Definition 5.10]{saito1} Let $X$ be a smooth scheme of dimension $n$ over $k$ and $K$ an object of $D^b_c(X, \mathbb F_\ell)$. Let the \emph{characteristic cycle of $K$}, $CC(K)$, be the unique $\mathbb Z$-linear combination of irreducible components of $SS(K)$ such that for every \'{e}tale morphism $j: W \to X$, every morphism $f: W\to Y$ to a smooth curve and every at most isolated $h^\circ SS(\mathcal F)$-characteristic point $u \in W$ of $f$, we have
\[ - \operatorname{dimtot} \left( R \Phi_f(j^* K) \right)_u =  (j^* CC(K), (df)^*\omega )_{T^*W,u} \] where $\omega$ is a meromorphic one-form on $Y$ with no zero or pole at $f(u)$.

\end{definition}

Here the notation $(,)_{T*W ,u}$ denotes the intersection number in $T^* W$ at the point $u$.

The existence and uniqueness is \cite[Theorem 5.9]{saito1}, except for the fact that the coefficients lie in $\mathbb Z$ and not $\mathbb Z[1/p]$, which is \cite[Theorem 5.18]{saito1} and is due to Beilinson, based on a suggestion by Deligne.

%
%
%
%
%
%

\vspace{10pt}

We now describe the setting in which we will  construct our sheaves, and explain the notation we will use to express their characteristic cycle. Let $C$ be a smooth curve and $n$ a natural number. Let $C^{(n)}$ be the $n$th symmetric power of $C$. 

We can view points in $C^{(n)}$ as ideal sheaves $\mathcal I$ whose quotient $\mathcal O_C/\mathcal I$ has length $n$. The tangent space to $C^{(n)}$ at an ideal sheaf $\mathcal I$ can be viewed as $H^0(C, \mathcal I^\vee / \mathcal O_C)$. Here at a point $(x_1,\dots,x_n)$ of $C^n$, the derivative of the natural map $sym: C^n \to C^{(n)}$ is given by $\sum_{i=1}^{n} \frac{d x_i}{x_i}$ where $d x_i$ is calculated by some local coordinate at $x_i$.

Let $K$ be a perverse sheaf on $C$. Let $\Sing$ be the singular locus of $K$. Let $\operatorname{rank}$ be the generic rank of $K$ and for $x \in \Sing$, let $c_{x}$ be the multiplicity of the contangent space at $x$ in the characteristic cycle of $K$ (i.e. the generalized logarithmic Artin conductor of $K$ at $x$).

Let $\rho$ be a representation of $S_{n} $.

\begin{defi}

Let $(e_x)_{x \in \Sing} $ be a tuple of natural numbers indexed by $\Sing$ and let $( w_k)_{k \in \mathbb N^+} $ be a tuple of natural numbers indexed by positive natural numbers such that $\sum_{x \in \Sing} e_x + \sum_k  k w_k =n$. We will refer to the tuples as $(e_x)$ and $(w_k)$, for short, and their individual elements as $e_x$ and $w_k$. 

Let $ev_{(e_x), (w_k)} $ be the map from $\prod_{k=1}^{\infty} C^{(w_k)}$ to $C^{(n)}$ that sends a tuple $(D_k)_{k \in \mathbb N^+}$ of effective divisors, with $\deg D_k =w_k$, to $ \sum_{x \in \Sing} e_x [x] +  \sum_{k=1}^{\infty}  k [D_k]$.

Define a closed subset $A_{(e_x), (w_k)} $ of $C^{(n)}$ as the image of the map $ev_{(e_x),(w_k)}$.

Define the vector bundle $W_{ (e_x), (w_k)} $ on $\prod_{k=1}^{\infty} C^{(w_k)}$ as the kernel of the natural surjection of vector bundles
\[ H^0\left(  \mathcal O_C \left( \sum_x e_x [x] + \sum_k k D_k\right) /  \mathcal O_C\right) ^\vee \to H^0\left(  \mathcal O_C \left(  \sum_k  D_k\right) /  \mathcal O_C\right)^\vee \]
where $(D_k)_{k=1}^{\infty} $ is a point of $\prod_{k=1}^{\infty} C^{(w_k)}$.  We have a map from $W_{ (e_x), (w_k)} $ to $T^* C^{(n)}$ arising from the map $ev_{(e_x), (w_k)} : \prod_{k=1}^{\infty} C^{(w_k)} \to C^{ (n)}$  and the injection of vector bundles
\[ W \to H^0\left(  \mathcal O_C \left( \sum_x e_x [x] + \sum_k k D_k\right) /  \mathcal O_C\right) ^\vee  = H^0 (C, \mathcal I^\vee /\mathcal O_C)^\vee \]  where $\mathcal I =  O_C \left( -\sum_x e_x [x] -  \sum_k k D_k\right)$ corresponds to $ev_{(e_x),(w_k) }( (D_k)_{k=1}^{\infty}) $. 

Define $B_{(e_x), (w_k)}$ as the pushforward of the class of $W_{ (e_x), (w_k)} $ to $T^* C^{(n)}$.  \end{defi}



It is not hard to see that $B_{(e_x),(w_k)}$ is a closed conical cycle on $T^* C^{(n)}$ of dimension $n$.

Note that the map $ev_{(e_x), (w_k)}$ used in this definition is generically injective, so that the cycle $B_{(e_x),(w_k)}$ is an irreducible closed conical subset of $T^* C^{(n)}$ with multiplicity one.

\begin{remark} The characteristic cycle in characteristic $p$ is not necessarily the conormal bundle of its support. For instance, one can see that $B_{(e_x),(w_k)}$ is a conormal bundle if and only if $w_k=0$ for all $k$ a multiple of $p$, and we will see in Theorem \ref{characteristic-cycle-sym} that $B_{(e_x), (w_k)}$ will appear as an irreducible component of the characteristic cycle of some natural sheaves. \end{remark}

Let \[M_{ K, \rho} ( (e_x), (w_k)) = 
\dim \left (  \bigotimes_{x \in \Sing }  ( \mathbb C^{c_{x}})^{\otimes e_{x} }  \otimes  (\mathbb C^{\operatorname{rank}} ) ^{ \otimes \sum_{k = 1 }^{\infty}  k w_{k }  }    \otimes \rho \right)^{ \prod_{x \in \Sing}  S_{e_{x} } \times  \prod_{k =1}^{\infty}  S_{k  }^{w_{k } } } .   \]

Here $\prod_{x \in \Sing}  S_{e_{x} } \times  \prod_{k =1}^{\infty}  S_{k  }^{w_{k } }  $ embeds into $S_{n} $ by acting as the group preserving the partition of $\{1,\dots,n\}$ into one part of size $e_{x} $ for each $x\in \Sing$ and $w_k$ parts of size $k $ for each $k$ in $\mathbb N^+$, and acts on $  \bigotimes_{x \in \Sing }  ( \mathbb C^{c_{x}})^{\otimes e_{x} }  \otimes  (\mathbb C^{\operatorname{rank}} ) ^{ \otimes \sum_{k = 1 }^{\infty}  k w_{k }}$ where $S_{e_x}$ permutes the factors of the form $\mathbb C^{c_x}$ and each copy of $S_k$ permutes $k$ factors of the form $\mathbb C^{\operatorname{rank}}$.

\begin{theorem}\label{characteristic-cycle-sym} We have \[ CC \left( ( sym_* K^{\boxtimes n} \otimes \rho)^{S_{n }} \right) = \sum_{\substack{ (e_{x}) :  \Sing \to \mathbb N \\ (w_{k} ): \mathbb N^+ \to \mathbb N \\  \sum_{ x\in \Sing} e_{x} + \sum_{k =1}^{\infty} k w_{k} = n   }}M_{ K, \rho} ( (e_x), (w_k))    [ B_{(e_x),(w_k)} ]     .\]
\end{theorem}

To prove Theorem \ref{characteristic-cycle-sym} will take several steps. We will first verify that the singular support of $sym_* K^{\boxtimes n}$ is contained in $\bigcup B_{ (e_x), (w_k)}$, so it suffices to check that the multiplicity of each irreducible component of the singular support in the characteristic cycle is as stated. We will next set up an inductive system where knowing this multiplicity identity for lesser $n$ lets us deduce it for most irreducible components for the original $n$. Then we will use the index formula to verify that if the identity holds for all but one irreducible component, then it holds for all irreducible components. Finally we will use a series of examples, as well as the \'{e}tale-local nature of the characteristic cycle, to deduce the identity for all possible irreducible components.

\begin{lemma}\label{singular-support} We have

\[ SS \left( ( sym_* K^{\boxtimes n} \otimes \rho)^{S_{n }} \right)\subseteq \bigcup_ {\substack{ e_{x} :  \Sing \to \mathbb N \\ w_{k} : \mathbb N^+ \to \mathbb N \\  \sum_{ x\in \Sing} e_{x} + \sum_{k =1}^{\infty} k w_{k} = n  } } B_{ (e_x), (w_k)}   .\]

\end{lemma}

\begin{proof} Because tensoring with $\rho$ and taking $S_n$-invariants both preserve local acyclicity along any map, they can only shrink the singular support, and so it suffices to prove

\[ SS \left(  sym_* K^{\boxtimes n}  \right)\subseteq \bigcup_ {\substack{ (e_{x}) :  \Sing \to \mathbb N \\ (w_{k}) : \mathbb N^+ \to \mathbb N \\  \sum_{ x\in \Sing} e_{x} + \sum_{k =1}^{\infty} k w_{k} = n   } }B_{(e_x),(w_k)}   .\]

We apply \cite[Lemma 2.2(ii)]{Beilinson} to the map $sym$. This says that \[ SS \left(  sym_* K^{\boxtimes n}  \right)\subseteq  sym_\circ  SS(K^{\boxtimes n})\] where $sym_\circ$ of a cycle means the image under $(T^* C^{(n)} \times_{C^{(n)}} C^n) \to T^* C^{(n)}$ of the inverse image under $d(sym): (T^* C^{(n)} \times_{C^{(n)}} C^n)  \to T^* C^n$ of the cycle. By \cite[Theorem 2.2(3)]{saito2}, the singular support of $K^{\boxtimes n}$ inside $T^* (C^n) = (T^* C)^n$ is the $n$-fold product of the union of the zero section with the inverse image of $Sing$ in $T^* C$.  Fix $x_1,\dots, x_n$ in $C$. Then a linear form on the tangent space of $C^n$ at $(x_1,\dots, x_n)$ lies in $SS(K^{\boxtimes n})$ if and only if, for all $i$ from $1$ to $n$, if  $x_i \notin \Sing$, the linear form, restricted to the tangent space to $C$ at $x_i$, vanishes. Letting $\mathcal I = \mathcal O_C ( -x_1 - \dots - x_n)$, we deduce that a linear form on $H^0(\mathcal I^\vee / \mathcal O)$ lies in $ sym_\circ  SS(K^{\boxtimes n})$ if and only if, for all $i$ from $1$ to $n$ such that $x_i \notin \Sing$, the linear form vanishes on functions with a pole of order $1$ at $x_i$ and no poles elsewhere. 

Now let $e_x$ for $x \in \Sing$ be the number of $i$ from $1$ to $n$ such that $x_i = x$, and let $w_k$ for $k$ in $\mathbb N^+$ be the number of $y \in C \setminus \Sing$ such that $|\{ i \in \{1,\dots n \} | x_i =y \} | = k$ .

Then we have an equality of divisors  \[x_1 + \dots + x_n =   \sum_{x \in \Sing} e_x [x] +  \sum_{k=1}^{\infty}  k [D_k]\] where $D_k$ is the divisor of degree $w_k$ consisting of the sum of all $y \in C \setminus \Sing$ such that $|\{ i \in \{1,\dots n \} | x_i =y \} | = k$. It follows that the divisor $x_1 + \dots + x_n$ lies in $A_{(e_x),(w_k)}$. A linear form on $H^0(\mathcal I^\vee / \mathcal O)$ lies in $ sym_\circ  SS(K^{\boxtimes n})$ if and only if it vanishes on the space of functions with poles of order at most $1$ at the support of  $D_k$ and no poles elsewhere, meaning it vanishes on $H^0 ( C, \mathcal O_C ( \sum_k D_k) /\mathcal O_C)$ because all the points in the support of all the $D_k$ are distinct. Equivalently, a linear form on $H^0(\mathcal I^\vee / \mathcal O)$ lies in $ sym_\circ  SS(K^{\boxtimes n})$ if it lies in the kernel $W_{(e_x),(w_k)}$ of the natural map $H^0(\mathcal I^\vee / \mathcal O)^\vee$ to $H^0 ( C, \mathcal O_C ( \sum_k D_k) /\mathcal O_C)^\vee$. 

Hence a linear form over $x_1 + \dots + x_n$ lying in $sym_\circ  SS(K^{\boxtimes n})$ must also lie in $B_{(e_x),(w_k)}$. 

Thus \[sym_\circ SS(K^{\boxtimes n})\subseteq \bigcup_ {\substack{( e_{x}) :  \Sing \to \mathbb N \\ (w_{k}) : \mathbb N^+ \to \mathbb N \\  \sum_{ x\in \Sing} e_{x} + \sum_{k =1}^{\infty} k w_{k} = n   } }B_{(e_x),(w_k)}   \] and we are done. \end{proof}

\begin{lemma}\label{splitting-1} Let $(e_{x,1})$ and $(e_{x,2})$ be two tuples of natural numbers indexed by $\Sing$, such that $e_{x,1}e_{x,2}=0$ for all $x$. Let $(w_{k,1})$ and $(w_{k,2})$ be two tuples of natural numbers indexed by $\mathbb N^+$. Let $n_i = \sum_{x \in \Sing} e_{x,i} + \sum_{k=1}^\infty k w_{k,i}$, and let $n=n_1+n_2$. 

Then:
\begin{enumerate}

\item Near a general point of $A_{(e_{x,1}), (w_{k,1})} \times A_{(e_{x,2}), (w_{k,2})}$, the natural map $s: C^{(n_1)} \times C^{(n_2)}  \to C^{(n)}$ is \'{e}tale. 

\item Over the locus where $s$ is \'{e}tale, the multiplicity of $B_{(e_{x,1}), (w_{k,1})} \times B_{(e_{x,2}), (w_{k,2})}$ in $ s^* B_{(e_x),(w_k)}$ is $1$ if $ e_x= e_{x,1} +e_{x,2} $ and $w_k = w_{k,1} +w_{k,2} $ and zero otherwise.

\end{enumerate}

\end{lemma}

\begin{proof} To prove part (1), we observe that $s$ is \'{e}tale at a point $(y_1,y_2) \in C^{(n_1)} \times C^{(n_2)}$ when  the $n_1$ points of $C$ that make up $y_1$ are distinct from the the $n_2$ points of $C$ that make up $y_2$. At a generic point of  $A_{(e_{x,1}), (w_{k,1})} \times A_{(e_{x,2}), (w_{k,2})}$, all the points in the $D_k$s are distinct from each other and from $\Sing$, and by the assumption $e_{x,1}e_{x,2}=0$, all the points in $\Sing$ that make up $y_1$ are distinct from all the points of $\Sing$ that make up $y_2$. So $s$ is indeed \'{e}tale at this point.

To prove part (2), first observe that the image of a general point of $B_{(e_{x,1}), (w_{k,1})} \times B_{(e_{x,2}),(w_{k,2})} $ under $s$ is a general point of $B_{( e_{x,1}+e_{x,2}), (w_{k,1} + w_{k,2})}$. Because $B_{(e_x),(w_k)}$ are distinct for distinct $(e_x), (w_k)$, and equidimensional, the pullback of any other $B_{(e_x),(w_k)}$ includes $B_{(e_{x,1}), (w_{k,1})} \times B_{(e_{x,2}),(w_{k,2})}$  with multiplicity one. Furthermore, because $s$ is \'{e}tale here, it preserves multiplicity of cycles, so the pullback of $B_{ (e_{x,1}+e_{x,2}), (w_{k,1} + w_{k,2})}$ must include $B_{(e_{x,1}), (w_{k,1})} \times B_{(e_{x,2}),(w_{k,2})} $  with multiplicity one. \end{proof}

\begin{lemma}\label{splitting-induction} Assume Theorem \ref{characteristic-cycle-sym} holds for $n'< n$. Let $(e_x), (w_k)$ satisfy \[ \sum_{x \in \Sing} e_x + \sum_{k=1}^{\infty} k w_k = n. \] 

Then the multiplicity of $B_{(e_x),(w_k)}$ inside \[ CC \left( ( sym_* K^{\boxtimes n} \otimes \rho)^{S_{n }} \right) \] is equal to \[M_{K, \rho}((e_x), (w_k)) \] unless $e_x=n$ for some $x$ or $w_n=1$. \end{lemma}

\begin{proof} We have \[ \sum_{x \in \Sing} e_x + \sum_{k=1}^{\infty} k w_k = n.\] Unless $e_x=n$ for some $n$ or $w_n=1$, we can find $(e_{x,1}), (e_{x,2}), (w_{k,1}),(w_{k,2})$ with \[e_{x,1}+e_{x,2}=e_x\] \[ w_{k,1} + w_{k,2} = w_{k} \]  \[ e_{x,1}e_{x,2}=0\] \[  0<  \sum_{x \in \Sing} e_{x,i} +  \sum_{k=1}^{\infty} k w_{k,i}<  n.\]

Let $n_i= \sum_{x \in \Sing} e_{x,i} +  \sum_{k=1}^{\infty} k w_{k,i}.$

We calculate the multiplicity of $B_{(e_{x,1}), (w_{k,1})} \times B_{(e_{x,2}), (w_{k,2})}$ in $CC \left( s^* ( sym_* K^{\boxtimes n} \otimes \rho)^{S_{n }} \right) $ in two ways.

By \citep[Lemma 5.11(2)]{saito1}, we have \begin{equation}\label{5-11-2} CC \left( s^* ( sym_* K^{\boxtimes n} \otimes \rho)^{S_{n }} \right) = s^* CC \left( ( sym_* K^{\boxtimes n} \otimes \rho)^{S_{n }} \right)\end{equation} after restricting both sides to the open locus where $s$ is \'{e}tale. (This uses the fact that the characteristic cycle is compatible with restriction to open subsets, which is a special case of \citep[Lemma 5.11(2)]{saito1}.) By Lemma \ref{splitting-1}(1), $s$ is \'{e}tale at the generic point of $A_{(e_{x,1}), (w_{k,1})} \times A_{(e_{x,2}), (w_{k,2})}$. It follows that the multiplicity of $B_{(e_{x,1}), (w_{k,1})} \times B_{(e_{x,2}), (w_{k,2})}$ in both sides of \eqref{5-11-2} are equal. By Lemma \ref{splitting-1}(2), it follows that the multiplicity of $B_{(e_{x,1}), (w_{k,1})} \times B_{(e_{x,2}), (w_{k,2})}$ in $CC \left( s^* ( sym_* K^{\boxtimes n} \otimes \rho)^{S_{n }} \right) $ is equal to the multiplicity of $B_{(e_x),(w_k)}$ inside $ CC \left( ( sym_* K^{\boxtimes n} \otimes \rho)^{S_{n }} \right)$.

Second note that \[ s^* ( sym_* K^{\boxtimes n} \otimes \rho)^{S_{n }}  = ( sym_{n_1*} K^{\boxtimes n_1} \boxtimes sym_{n_2 *} K^{\boxtimes n_2}  \otimes \rho)^{S_{n_1} \times S_{n_2}}.\] If we write \[\rho=\bigoplus_{a} \rho_{a,1} \otimes \rho_{a,2}\] where $\rho_{a,i}$ are irreducible representations of $S_{n_i}$, then 
\[ ( sym_{n_1*} K^{\boxtimes n_1} \boxtimes sym_{n_2 *} K^{\boxtimes n_2}  \otimes \rho)^{S_{n_1} \times S_{n_2}}= \bigoplus_{a}  ( sym_{n_1*} K^{\boxtimes n_1}\otimes \rho_{a,1} ) \boxtimes ( sym_{n_1*} K^{\boxtimes n_1}\otimes \rho_{a,2} ) .\]

The characteristic cycle of this complex is \cite[Theorem 2.2.2]{saito2}
\[ \sum_a CC( sym_{n_1*} K^{\boxtimes n_1}\otimes \rho_{a,1} ) \boxtimes CC( sym_{n_1*} K^{\boxtimes n_1}\otimes \rho_{a,2} ).\]
By our assumption about Theorem \ref{characteristic-cycle-sym}, the multiplicity of $B_{(e_{x,1}), (w_{k,1})} \times B_{(e_{x,2}), (w_{k,2})}$  inside this characteristic cycle is \[ \sum_a M_{K, \rho_{a,1}} ((e_{x,1}),(w_{k,1}))   M_{K, \rho_{a,2}} ((e_{x,2}),(w_{k,2})) .\]

So it remains to check that
\[ M_{K, \rho} ((e_{x}),(w_k)) = \sum_a M_{K, \rho_{a,1}} ((e_{x,1}),(w_{k,1}))   M_{K, \rho_{a,2}} ((e_{x,2}),(w_{k,2})) .\]
To do this, we recall the definition 
\[M_{ K, \rho} ( (e_x), (w_k)) = 
\dim \left (  \bigotimes_{x \in \Sing }  ( \mathbb C^{c_{x}})^{\otimes e_{x} }  \otimes  (\mathbb C^{\operatorname{rank}} ) ^{ \otimes \sum_{k = 1 }^{\infty}  k w_{k }  }    \otimes \rho \right)^{ \prod_{x \in \Sing}  S_{e_{x} } \times  \prod_{k =1}^{\infty}  S_{k  }^{w_{k } } } .   \] 

Next note that everything in this formula splits, i.e.
\[ \prod_{x \in \Sing} S_{e_x} \times \prod_{k=1}^{\infty} S_k^{w_k} =  \left( \prod_{x \in \Sing} S_{e_{x,1} } \times \prod_{k=1}^{\infty} S_k^{w_{k,1}}\right) \times \left( \prod_{x \in \Sing} S_{e_{x,2} } \times \prod_{k=1}^{\infty} S_k^{w_{k,2}}\right)\]
and
\[ (\mathbb C^{c_{x}})^{\otimes e_{x} }  \otimes  ((\mathbb C^{\operatorname{rank}} ) ^{ \otimes \sum_{k = 1 }^{\infty}  k w_{k }  }    = \left( (\mathbb C^{c_{x}})^{\otimes e_{x,1 } }  \otimes  (\mathbb C^{\operatorname{rank}} ) ^{ \otimes \sum_{k = 1 }^{\infty}  k w_{k,1  }  }    \right) \otimes \left( \mathbb C^{c_{x}})^{\otimes e_{x,2 } }  \otimes  (\mathbb C^{\operatorname{rank}} ) ^{ \otimes \sum_{k = 1 }^{\infty}  k w_{k,2  }  }    \right).  \]
compatibly with the actions of these groups on these vector spaces. Furthermore, note that the splitting of groups is compatible with the embedding of $\prod_{x \in \Sing} S_{e_x} \times \prod_{k=1}^{\infty} S_k^{w_k} $ in $S_n$ and $\prod_{x \in \Sing} S_{e_{x,i}} \times \prod_{k=1}^{\infty} S_k^{w_{k,i}} $ in $S_{n_i}$ in the sense that it commutes with the natural embedding $S_{n_1} \times S_{n_2} \subset  S_n$. 

Because these splittings are compatible with this embedding,

\[M_{ K, \rho} ( (e_x), (w_k)) = 
\dim \left (  \bigotimes_{x \in \Sing }  ( \mathbb C^{c_{x}})^{\otimes e_{x} }  \otimes  (\mathbb C^{\operatorname{rank}} ) ^{ \otimes \sum_{k = 1 }^{\infty}  k w_{k }  }    \otimes \rho \right)^{ \prod_{x \in \Sing}  S_{e_{x} } \times  \prod_{k =1}^{\infty}  S_{k  }^{w_{k } } }\]

\[= \dim \left ( \bigotimes_{i \in \{1,2\}} \left( (\mathbb C^{c_{x}})^{\otimes e_{x,i } }  \otimes  (\mathbb C^{\operatorname{rank}} ) ^{ \otimes \sum_{k = 1 }^{\infty}  k w_{k,i  }  }    \right)  \otimes \rho \right)^{  \prod_{i \in \{1,2\} } \left( \prod_{x \in \Sing} S_{e_{x,i} } \times \prod_{k=1}^{\infty} S_k^{w_{k,i}}\right) }  \]

\[= \dim \bigoplus_a  \left ( \bigotimes_{i \in \{1,2\}} \left( (\mathbb C^{c_{x}})^{\otimes e_{x,i } }  \otimes  (\mathbb C^{\operatorname{rank}} ) ^{ \otimes \sum_{k = 1 }^{\infty}  k w_{k,i  }  }    \right)  \otimes \rho_{a,1} \otimes \rho_{a,2} \right)^{  \prod_{i \in \{1,2\} } \left( \prod_{x \in \Sing} S_{e_{x,i} } \times \prod_{k=1}^{\infty} S_k^{w_{k,i}}\right) }  \]

\[= \dim \bigoplus_a  \left ( \bigotimes_{i \in \{1,2\}} \left( (\mathbb C^{c_{x}})^{\otimes e_{x,i } }  \otimes  (\mathbb C^{\operatorname{rank}} ) ^{ \otimes \sum_{k = 1 }^{\infty}  k w_{k,i  }  }   \otimes \rho_{a,i}  \right)  \right)^{  \prod_{i \in \{1,2\} } \left( \prod_{x \in \Sing} S_{e_{x,i} } \times \prod_{k=1}^{\infty} S_k^{w_{k,i}}\right) }  \]

\[= \dim \bigoplus_a  \bigotimes_{i \in \{1,2\}} \left( (\mathbb C^{c_{x}})^{\otimes e_{x,i } }  \otimes  (\mathbb C^{\operatorname{rank}} ) ^{ \otimes \sum_{k = 1 }^{\infty}  k w_{k,i  }  }   \otimes \rho_{a,i}  \right) ^{  \prod_{x \in \Sing} S_{e_{x,i} } \times \prod_{k=1}^{\infty} S_k^{w_{k,i}} }  \]

\[ =\sum_a \prod_{i\in \{1,2\}} \dim \left( (\mathbb C^{c_{x}})^{\otimes e_{x,i } }  \otimes  (\mathbb C^{\operatorname{rank}} ) ^{ \otimes \sum_{k = 1 }^{\infty}  k w_{k,i  }  }   \otimes \rho_{a,i}  \right) ^{  \prod_{x \in \Sing} S_{e_{x,i} } \times \prod_{k=1}^{\infty} S_k^{w_{k,i}} }\]

\[ = \sum_a M_{K, \rho_{a,1}} ((e_{x,1}),(w_{k,1}))   M_{K, \rho_{a,2}} ((e_{x,2}),(w_{k,2}))\]

as desired.

\end{proof}

\begin{lemma}\label{basic-index-count} We have
\[ ( B_{(e_x),(w_k)}, C^{(n)})_{ T^* C^{(n)}}  = \frac{ (2g-2)! } { (2g-2 - \sum_{k=1}^{\infty} w_k)! \prod_{k=1}^{\infty} w_k! }  = \frac{\prod_{i=0}^{-1+ \sum_{k=1}^{\infty} w_k} (2g-2-i) }{ \prod_{k=1}^{\infty}w_k! } .\]

Here the second formula is to be used to ensure that the expression is well-defined in the case that $2g-2 - \sum_{k=1}^{\infty} w_k$ is negative. \end{lemma} 

\begin{proof} 
By definition, $B_{(e_x),(w_k)}$ is the pushforward of the class of $W_{ (e_x), (w_k)} $ to $ T^* C^{(n)}$. Hence by the projection formula \[ ( B_{(e_x),(w_k)}, C^{(n)})_{ T^* C^{(n)}} = (W_{ (e_x), (w_k)} ,  \prod_{k=1}^{\infty} C^{(w_k)})_{ H^0(  \mathcal O_C ( \sum_x e_x [x] + \sum_k k D_k) /  \mathcal O_C) ^\vee}. \]

By \cite[Example 6.3.5]{Fulton}, the class of $W_{ (e_x), (w_k)} $ is simply the top Chern class of $H^0(  \mathcal O_C (  \sum_k  D_k) /  \mathcal O_C)^\vee$. So it suffices to show that degree of the top Chern class of the vector bundle $H^0(  \mathcal O_C (  \sum_k  D_k) /  \mathcal O_C)^\vee$ on $\prod_{k=1}^\infty C^{(w_k)}$ is $ \frac{ (2g-2)! } { (2g-2 - \sum_{k=1}^{\infty} w_k)! \prod_{k=1}^{\infty} w_k! } .$

To do this, note that this vector bundle is the pull back of $H^0 (\mathcal O_C(D) /\mathcal O_C)^\vee$ from $C^{ (\sum_{k=1}^{\infty} w_k )} $ under the map $\prod_{k=1}^\infty C^{(w_k)} \to C^{ (\sum_{k=1}^{\infty} w_k)}$ that sends $(D_k)_{k=1}^{\infty}$ to $\sum_{k=1}^{\infty} D_k$, where $D$ denotes the universal divisor over $C^{ (\sum_{k=1}^{\infty} w_k)}$. Because this is the pullback of a map of degree $\frac{ (\sum_{k=1}^{\infty} w_k)!}{ \prod_{k=1}^{\infty} w_k!}$, it suffices to show that the top Chern class of this vector bundle on $C^{ (\sum_{k=1}^{\infty} w_k )} $ is \[ \binom{ 2g-2 }{  \sum_{k=1}^{\infty} w_k} .\]

This vector bundle is simply the cotangent bundle of $C^{ (\sum_{k=1}^{\infty} w_k )} $, so its top Chern class is $(-1)^{\sum_{k=1}^{\infty} w_k}$ times the topological Euler characteristic, which is indeed  $\binom{2g-2 }{  \sum_{k=1}^{\infty} w_k}$. We can view this binomial coefficient as a polynomial in $g$, in which case it will still correctly calculate the Euler characteristic outside the range where the binomial coefficient is usually defined. Thus our formula will hold in general as well. \end{proof}

\begin{lemma}\label{advanced-index} We have

\[ \left( CC \left( ( sym_* K^{\boxtimes n} \otimes \rho)^{S_{n }} \right), C^{(n)}\right)  =\left(  \sum_{\substack{( e_{x}) :  \Sing \to \mathbb N \\ (w_{k} ): \mathbb N^+ \to \mathbb N \\  \sum_{ x\in \Sing} e_{x} + \sum_{k =1}^{\infty} k w_{k} = n   }}M_{ K, \rho} ( (e_x), (w_k))    [ B_{(e_x),(w_k)} ] , C^{(n)}\right)   . \]
\end{lemma}

\begin{proof}
For the left side, by \cite[Theorem 6.13]{saito1} we have
\[ \left(CC \left( ( sym_* K^{\boxtimes n} \otimes \rho)^{S_{n }} \right), C^{(n)}\right)   = \chi ( C^{(n)}, (sym_* K^{\boxtimes n} \otimes \rho)^{S_{n }} ) = \chi \left(  \left( H^* ( C, K)^{\otimes n} \otimes \rho \right)^{S_n} \right).\]

Now $\chi(C,K) = (2g-2) \rank + \sum_{x\in \Sing} c_x$ by the Grothendieck-Ogg-Shafarevich formula.

We have \[   \chi \left(  \left( H^* ( C, K)^{\otimes n} \otimes \rho \right)^{S_n} \right)= \frac{1}{n!} \sum_{\sigma \in S_n} \tr(\sigma, \rho) \tr \left( \sigma, H^* ( C, K)^{\otimes n}\right)   =  \frac{1}{n!} \sum_{\sigma \in S_n} \ \tr(\sigma, \rho)  \chi(C,K)^{\#\textrm{ of orbits of }\sigma} \]  is a polynomial in $\chi(C,K)$, thus a polynomial in $g$. By Lemma \ref{basic-index-count}, $ (B_{(e_x),(w_k)}, C^{(n)})$ is a polynomial in $g$. Thus the identity to prove is an identity between two polynomials in $g$. Hence we may assume that $g>0$. Because $g>0$, and because we may freely twist by a rank one lisse sheaf, we may assume that $H^*(C,K)$ is supported in degree zero, so that \[ \chi \left(  \left( H^* ( C, K)^{\otimes n} \otimes \rho \right)^{S_n} \right) = \dim \left( H^0(C,K)^{\otimes n}  \otimes \rho\right)^{S_n}.\]

Then $H^0(C,K)$ is a vector space of dimension $(2g-2) \rank + \sum_{x \in \Sing}  c_x$. We can partition a basis for $H^0(C,K)$ into $2g-2$ parts of cardinality $\rank$ and one part of cardinality $c_x$ for each $x \in \Sing$. Having done this, we obtain a basis for $H^0( C, K)^{\otimes n}$ by tensor products $v$ of ordered tuples of $n$ basis vectors. For each tensor product of $n$ basis vectors $v$, let $k_i(v)$ be the number of basis vectors in the $i$th part of size $\rank$ and let $e_x(v)$ be the number of basis vectors in the part of size $c_x$.

For each pair of tuples $(k_i), (e_x)$, the tensor products of $n$ vectors with $(k_i(v))=(k_i)$ and $(e_x(v))=(e_x)$ generate a $S_n$-stable subspace of $H^0(C,K)^{\otimes n}$. Because each basis vector lies in exactly one of these generating sets, the sum of all these spaces is exactly $H^0(C,K)^{\otimes n}$. Thus we can write $\dim  \left( H^0 ( C, K)^{\otimes n} \otimes \rho \right)^{S_n}$ as a sum over tuples $((k_i),(e_x))$ of the contribution to the dimension from the corresponding subspace $H^0 ( C, K)^{\otimes n}$.

As a representation of $S_n$, the subspace of $H^0 ( C, K)^{\otimes n} $ corresponding to $((k_i), (e_x))$ is \[ \Ind_{ \prod_{x\in \Sing} S_{e_x} \times \prod_{i=1}^{2g-2} S_{k_i}   }^{S_n} \bigotimes_{x\in \Sing} (\mathbb C^{c_{x}})^{\otimes e_{x} }  \otimes  (\mathbb C^{\operatorname{rank}} ) ^{ \otimes \sum_{i=1 }^{2g-2}  k_i  }.\]

Hence tensoring with $\rho$ and taking $S_n$-invariants, the contribution of this subspace to  $\dim  \left( H^0 ( C, K)^{\otimes n} \otimes \rho \right)^{S_n}$ is
\[ \left( \bigotimes_{x\in \Sing}( \mathbb C^{c_{x}})^{\otimes e_{x} }  \otimes  (\mathbb C^{\operatorname{rank}} ) ^{ \otimes \sum_{i=1 }^{2g-2}  k_i  } \otimes \rho \right)^{ \prod_{x\in \Sing} S_{e_x} \times \prod_{i=1}^{2g-2} S_{k_i}   } \]
\[ = M_{K,\rho}((e_x), (w_k)) \] where $w_k$ is the number of $i$ with $k_i=k$. 

The number of tuples $(k_i)$ that produce a fixed sequence $(w_k)$ is  \[  \frac{\prod_{i=0}^{-1+ \sum_{k=1}^{\infty} w_k} (2g-2-i) }{ \prod_{k=1}^{\infty}w_k! } \] so by Lemma \ref{basic-index-count} the contribution of all these tuples to $\dim  \left( H^0 ( C, K)^{\otimes n} \otimes \rho \right)^{S_n}$  is exactly equal to \[ M_{ K, \rho} ( e_x, w_k)    ( [B_{(e_x),(w_k)} ], C^{(n)})_{T^* C^{(n)}}.\] Summing, we get 

\[  \left(CC \left( ( sym_* K^{\boxtimes n} \otimes \rho)^{S_{n }} \right), C^{(n)}\right)   =\dim \left( H^0(C,K)^{\otimes n}  \otimes \rho\right)^{S_n} \] \[= \sum_{\substack{( e_{x}) :  \Sing \to \mathbb N \\ (w_{k} ): \mathbb N^+ \to \mathbb N \\  \sum_{ x\in \Sing} e_{x} + \sum_{k =1}^{\infty} k w_{k} = n   }}\left(  M_{ K, \rho} ( (e_x), (w_k))    [ B_{(e_x),(w_k)} ] , C^{(n)}\right)    \] \[=\left(  \sum_{\substack{( e_{x}) :  \Sing \to \mathbb N \\ (w_{k} ): \mathbb N^+ \to \mathbb N \\  \sum_{ x\in \Sing} e_{x} + \sum_{k =1}^{\infty} k w_{k} = n   }}M_{ K, \rho} ( (e_x), (w_k))    [ B_{(e_x),(w_k)} ] , C^{(n)}\right)   . \] 

\end{proof} 

\begin{lemma}\label{sym-induction-step} Assume Theorem \ref{characteristic-cycle-sym} holds for $n'< n$. Then for any $e_x, w_k$ with $\sum_{x\in \Sing} e_x + \sum_{k=1}^{\infty} k w_k=n$, the multiplicity of $B_{(e_x),(w_k)}$ inside \[ CC \left( ( sym_* K^{\boxtimes n} \otimes \rho)^{S_{n }} \right) \] is equal to \[M_{K,\rho}((e_x), (w_k)) .\]

\end{lemma}

\begin{proof} Let $\operatorname{mult}_{K,\rho}((e_x),(w_k))$ be the multiplicity of $B_{(e_x),(w_k)}$ inside $CC \left( ( sym_* K^{\boxtimes n} \otimes \rho)^{S_{n }} \right) $. We aim to show that
\[\operatorname{mult}_{K,\rho}((e_x),(w_k))=  M_{K,\rho}((e_x), (w_k)) \]
for all tuples $(e_x), (w_k)$. Our main tools will be Lemma \ref{splitting-induction}, which shows that this holds unless $w_n=1$ or $e_x=1$ for some $n$, and Lemma \ref{advanced-index}, which implies that
\[  \sum_{\substack{( e_{x}) :  \Sing \to \mathbb N \\ (w_{k} ): \mathbb N^+ \to \mathbb N \\  \sum_{ x\in \Sing} e_{x} + \sum_{k =1}^{\infty} k w_{k} = n   }}\operatorname{mult}_{ K, \rho} ( (e_x), (w_k))   \left(  [ B_{(e_x),(w_k)} ] , C^{(n)} \right)    \] \[  =   \sum_{\substack{( e_{x}) :  \Sing \to \mathbb N \\ (w_{k} ): \mathbb N^+ \to \mathbb N \\  \sum_{ x\in \Sing} e_{x} + \sum_{k =1}^{\infty} k w_{k} = n   }}M_{ K, \rho} ( (e_x), (w_k))   \left(  [ B_{(e_x),(w_k)} ] , C^{(n)} \right)     . \] This uses the fact, from Lemma \ref{singular-support}, that $ [ B_{(e_x),(w_k)} ] $ are the only irreducible components that appear in $CC \left( ( sym_* K^{\boxtimes n} \otimes \rho)^{S_{n }} \right) $.

We can cancel from the left side and right side terms where we already know that $\operatorname{mult}_{K,\rho}((e_x),(w_k))=  M_{K,\rho}(e_x, w_k)$. Our goal will be to get to a situation where only one term remains, and the intersection number $\left(  [ B_{(e_x),(w_k)} ] , C^{(n)} \right)$ is nonzero, so we can divide by $\left(  [ B_{(e_x),(w_k)} ] , C^{(n)} \right)$ and get our desired identity for that term. 

We now turn to the proof.

Let us first handle the case that $w_n=1$.  Because $\sum_{x\in \Sing} e_x + \sum_{k=1}^{\infty} k w_k=n$, we must have $e_x=0$ for all $x$ and $w_k=0$ for all $k \neq n$. Thus $B_{(e_x),(w_k)}$ is supported over the diagonal curve $C \subseteq C^{(n)}$. Because the characteristic cycle is preserved by \'{e}tale pullbacks, as is $M_{K, \rho}$, we may work in an \'{e}tale-neighborhood of a general point of that curve, which we take to be $\Sym^n$ of an \'{e}tale neighborhood of a general point of $C$. By working \'{e}tale-locally at a generic point, we may assume that $K$ is lisse, so that $\Sing$ is empty, and further assume that $g \neq 1$. It follows from Lemma \ref{splitting-induction} that the multiplicity of $B_{\emptyset,(w_k')}$ in the characteristic cycle is $M_{K,\rho} (\emptyset,(w_k'))$ for all $w_k'\neq w_k$.  Hence we can cancel all terms from Lemma \ref{advanced-index} except the contribution of $(w_k)$, obtaining 
\[ \operatorname{mult}_{ K, \rho} ( \emptyset, (w_k))   \left(  [ B_{\emptyset,(w_k)} ] , C^{(n)} \right)      =   M_{ K, \rho} ( \emptyset, (w_k))   \left(  [ B_{\emptyset, (w_k)} ] , C^{(n)} \right)     . \] 

By Lemma \ref{basic-index-count} $\left(  [ B_{(e_x),(w_k)} ] , C^{(n)} \right) = 2g-2$, and we assumed $g\neq 1$, so $\left(  [ B_{(e_x),(w_k)} ] , C^{(n)} \right)\neq 0$. Hence we may divide by it, obtaining   \[\operatorname{mult}_{K, \rho}(\emptyset,(w_k))=  M_{K, \rho}(\emptyset, (w_k)) \] and thus finishing the case $w_n=1$.

Next let us handle the case when $e_x=n$, $K$ has tame ramification around $x$, and $K$ is locally around $x$ a middle extension from the open set where it is lisse.  Because $\sum_{x\in \Sing} e_x + \sum_{k=1}^{\infty} k w_k=n$, we must have $e_{x'}=0$ for all $x' \neq x$ and $w_k=0$ for all $k$. Thus $B_{(e_x),(w_k)}$ is supported at the point $n[x]$ in $C^{(n)}$. Because this identity is an \'{e}tale-local question on $C$ in a neighborhood of $x$, we may assume that $C = \mathbb P^1$, $x=0$, and $K$ is a lisse sheaf on $\mathbb G_m$ with tame ramification at $0$ and $\infty$, placed in degree $-1$ and middle extended from $\mathbb G_m$ to $\mathbb P^1$. (This is because of the equivalence of categories between tame lisse sheaves on $\mathbb G_m$ and tame lisse sheaves on the punctured local ring of $\mathbb P^1$ at $0$.) By Lemma \ref{splitting-induction} and the previous case, we have \[\operatorname{mult}_{K,\rho}((e_x'),(w_k'))=  M_{K,\rho}((e_x'), (w_k')) \] for all tuples $(e_x'),(w_k')$ except for those with $w_k'=0$ and $e_0'=n, e_\infty'=0$ or $e_0'=0, e_\infty'=n$. There are two tuples with this property. Call them $((e_{x,1}),(w_{k,1}))$ and $((e_{x,2}),(w_{k,2}))$. Lemma \ref{advanced-index} now gives
\[ \operatorname{mult}_{ K, \rho} ( (e_{x,1}) , (w_{k,1}))   \left(  [ B_{(e_{x,1}) , (w_{k,1})} ] , C^{(n)} \right)  + \operatorname{mult}_{ K, \rho} ( (e_{x,2}) , (w_{k,2}))   \left(  [ B_{(e_{x,2}) , (w_{k,2})} ] , C^{(n)} \right) \] \[=M_{ K, \rho} ( (e_{x,1}) , (w_{k,1}))   \left(  [ B_{(e_{x,1}) , (w_{k,1})} ] , C^{(n)} \right)  + M_{ K, \rho} ( (e_{x,2}) , (w_{k,2}))   \left(  [ B_{(e_{x,2}) , (w_{k,2})} ] , C^{(n)} \right) .\]

Because $(w_{k,i})=0$ for $i=1,2$, $B_{(e_{x,i}),( w_{k,i})}$ is simply the fiber of the cotangent bundle over a point so \[ \left(  [ B_{(e_{x,1}) , (w_{k,1})} ] , C^{(n)} \right)  =  \left(  [ B_{(e_{x,2}) , (w_{k,2})} ] , C^{(n)} \right) =1\]
 and thus 
\[ \operatorname{mult}_{ K, \rho} ( (e_{x,1}) , (w_{k,1}))  + \operatorname{mult}_{ K, \rho} ( (e_{x,2}) , (w_{k,2})) \] \[ =M_{ K, \rho} ( (e_{x,1}) , (w_{k,1}))    + M_{ K, \rho} ( (e_{x,2}) , (w_{k,2})) .\]

Next we can check that \[\operatorname{mult}_{ K, \rho} ( (e_{x,1}) , (w_{k,1})) =  \operatorname{mult}_{ K, \rho} ( (e_{x,2}) , (w_{k,2})).\] By the classification of lisse sheaves on $\mathbb G_m$, $K$ in a neighborhood of $0$ is geometrically isomorphic to the dual of $K$ in a neighborhood of $\infty$, so $ (K^{\boxtimes n} \otimes \rho)^{S_n}$ in a neighborhood of $0$ is isomorphic to the dual of $ (K^{\boxtimes n} \otimes \rho)^{S_n}$ in a neighborhood of $\infty$. Because they are dual, they have the same characteristic cycle \cite[Lemma 4.13.4]{saito1}.

Furthermore, for both these $e_x',w_k'$, \[ M_{K,\rho} (e_{x'}, w_k')=  \left( \left( \mathbb C^{c_0} \right)^{\otimes n} \otimes \rho \right)^{S_n} = \left( \left( \mathbb C^{c_0} \right)^{\otimes n} \otimes \rho \right)^{S_\infty}.\]  so we have \[ M_{ K, \rho} ( (e_{x,1}) , (w_{k,1}))    = M_{ K, \rho} ( (e_{x,2}) , (w_{k,2})) .\]

It now follows by division by $2$ that \[ \operatorname{mult}_{ K, \rho} ( (e_{x,1}) , (w_{k,1}))    =M_{ K, \rho} ( (e_{x,1}) , (w_{k,1}))  .\] This completes the case that $e_x=n$ for some $x$ where $K$ has tame ramification and is a middle extension sheaf.

Let us finally handle the general case when $e_x=n$. Again we may pass to an \'{e}tale neighborhood of $x$. Katz and Gabber showed that any lisse sheaf on the punctured spectrum of the \'{e}tale local ring at $0$ of $\mathbb P^1$ can be extended to a lisse sheaf on $\mathbb G_m$ with tame ramification at $\infty$ \cite[Theorem 1.4.1]{KatzCanonical}. It follows that any perverse sheaf on the spectrum of the \'{e}tale local ring at $0$ can be extended to a sheaf on $\mathbb A^1$, lisse on $\mathbb G_m$, and with tame ramification at $\infty$. We can furthermore middle extend from $\mathbb A^1$ to $\mathbb P^1$. Working locally, we may assume $C$ is $\mathbb P^1$ and $K$ has this form.

We now argue as before. By Lemma \ref{splitting-induction} and the previous cases, the multiplicity  of $B_{ (e_{x}'), (w_k')}$ in the characteristic cycle is equal to $M_{K,\rho} (e_{x}',w_k')$ for all $e_{x}',w_k'$ except the one whose multiplicity we would like to compute. Because the intersection number of $B_{ (e_{x}'), (w_k')}$ with the zero-section is $1$, we can cancel all the other terms in Lemma \ref{advanced-index} and extract our desired identity.

\end{proof}

Theorem \ref{characteristic-cycle-sym} now follows by induction on $n$, with Lemma \ref{sym-induction-step} as the induction step and either $n=0$ or $n=1$, depending on preference, as the base case.

\section{Geometric setup}\label{geometric-setup}

We now specialize to the genus zero case. Let $C = \mathbb P^1$. Then $C^{(n)} = \mathbb P^n$. Specifically, this isomorphism comes from viewing $\mathbb P^n$ as the projectivization of the vector space $H^0( \mathbb P^1, \mathcal O(n))$. Let $\mathbb P^\vee$ be the projectivization of the dual vector space. Let $Y \subseteq \mathbb P^\vee \times C^{(n)}$ be the graph of the universal family of hyperplanes, with projection $p_1$ to $\mathbb P^\vee$ and $p_2$ to $C^{(n)}$.

For a complex $K$ of $\ell$-adic sheaves on $\mathbb P^n$, define the \emph{Radon transform} of $K$ to be $Rp_{1*} p_2^{*} K [n-1]$.

Let $\mathcal F$ be a rank two middle extension sheaf on $C$, pure of weight $0$. For each geometric point $x$ in the singular locus $\Sing$ of $\mathcal F$, let $c_x$ be the Artin conductor of $\mathcal F$ at $x$, and let $N = \sum_x c_x[x]$ be the associated divisor. Let $K_n$ be the Radon transform of $( sym_* \mathcal F^{\boxtimes n})^{S_n} [n]$, or in other words
\[ K_n= R p_{1*} p_2^* ( sym_* \mathcal F^{\boxtimes n})^{S_n} [2n-1] .\]

For $\alpha$ a linear form on $H^0( \mathbb P^1, \mathcal O(n))$, let $P(\alpha)$ be the set of effective divisors of degree $n$ on $\mathbb P^1_{\mathbb F_q}$ such that $\alpha(f)=0$ for all $f \in H^0(\mathbb P^1, \mathcal O(n))$ with $\div(f)=D$. In other words, $P(\alpha) $ is the image under $p_2$ of $p_1^{-1}(\alpha) (\mathbb F_q)$.

\begin{lemma}\label{Radon-trace-function} The trace of $\Frob_q$ on the stalk of $K_n$ at a  $\mathbb F_q$-point of $\mathbb P^\vee$ corresponding to a linear form $\alpha$ on $H^0(\mathbb P^1, \mathcal O(n))$ is equal to
\[ - \sum_{ D \in P(\alpha) }   r_{\mathcal F}(D) .\] \end{lemma}

\begin{proof} By the Lefschetz fixed point formula, the trace of $\Frob_q$ on the stalk of $K_n$ at $\alpha$ is minus the sum of the trace of $\Frob_q$ on the stalk of $p_2^* ( sym_* \mathcal F^{\boxtimes n})^{S_n}$ over $p_1^{-1}(\alpha) (\mathbb F_q)$. Because the $\mathbb F_q$-points of $p_1^{-1}(\alpha)$ are exactly the divisors of $P(\alpha)$, it suffices to prove that the trace of Frobenius on the stalk of $( sym_* \mathcal F^{\boxtimes n})^{S_n}$ at a divisor $D$ is $r_{\mathcal F}(D)$.

If we write $D = \sum_i m_i x_i$ for distinct geometric points $x_i$ with multiplicities $m_i$, \[ \left( sym_* \mathcal F^{\boxtimes n}\right)_D = \bigoplus_{ \substack{ (y_1,\dots, y_n) \in \mathbb P^1(\overline{\mathbb F}_q) \\  | \{ j | y_j = x_i \}  |= m_i \textrm{ for all }i}}  \bigotimes_{j=1}^n \mathcal F_{y_j} .\]

The group $S_n$ acts transitively on the set of such tuples $y_j$, with stabilizer $\prod_i S_{m_i}$, so we can write this sum as an induced representation 
\[ \left( sym_* \mathcal F^{\boxtimes n}\right)_D = \operatorname{Ind}_{\prod_i S_{m_i}}^{S_n} \bigotimes_{i} \mathcal F_{x_i}^{\otimes m_i} \]
and thus
\[ \left( sym_* \mathcal F^{\boxtimes n}\right)_D^{S_n}  = \left( \bigotimes_{i} \mathcal F_{x_i}^{\otimes m_i} \right)^{ \prod_{i} S_{m_i}} = \bigotimes_i  \Sym^{m_i} \mathcal F_{x_i}  \]

Now $\Frob_q$ acts on $\{x_i\}$ with one orbit of size $d$ for each closed point in $D$ of degree $d$. Thus the trace of $\Frob_q$ on $\bigotimes_i  \Sym^{m_i} \mathcal F_{x_i}$ is the product over the set of orbits of the trace of $\Frob_q$ on the tensor product of $\Sym^{m_i} \mathcal F_{x_i}$ for $x_i$ in that orbit. For an orbit of degree $d$, the trace of $\Frob_q$ on the tensor product of the $d$ terms in the orbit is simply the trace of $\Frob_q^d$ on a single term $\Sym^{m_i} \mathcal F_{x_i}$. This orbit corresponds to a closed point $v$ of degree $d$ with multiplicity $m_i$ in $D$, and the local contribution is \[ \tr (\Frob_{|\kappa_v|}, \Sym^{m_i} \mathcal F_v) = r_{\mathcal F} (m_v [v] )\] by the generating function identity \[ \frac{1}{ \det( 1- u \Frob_{|\kappa_v | }, \mathcal F_v)}= \sum_{m=0}^{\infty}  \tr (\Frob_{|\kappa_v|}, \Sym^{m_i} \mathcal F_v) u^m.\] Thus the total trace at $D$ is \[ \prod_v r_{\mathcal F} (m_v [v] ) = r_{\mathcal F} ( \sum_v m_v[v]) = r_{\mathcal F}(D),\] as desired. \end{proof}

\begin{lemma}\label{Radon-perverse-pure}  The complex $K_n$ is perverse and pure of weight $2n-1$. \end{lemma}

\begin{proof} We will first verify that $( sym_* \mathcal F^{\boxtimes n})^{S_n} [n]$ is perverse and pure of weight $n$.

Because $\mathcal F$ is a middle extension sheaf pure of weight $0$, $\mathcal F[1]$ is perverse and pure of weight $1$, because middle extension sheaves are perverse middle extensions up to shift \cite[Example on p. 153]{KiehlWeissauer} and shifting the degree shifts the weights by definition. Thus $\mathcal F^{\boxtimes n}[n]$ is perverse and pure of weight $n$ by $n-1$ applications of \cite[Proposition 4.2.8 and Stability 5.1.14.1 and 5.14.1*]{bbdg}. Then  $ sym_* \mathcal F^{\boxtimes n}[n]$ is perverse and pure of weight $n$ by \cite[Corollary 4.1.3 and Stability 5.1.14(i,i*)]{bbdg} because $sym$ is finite so $sym_* = sym_!$.  Finally, passing to $S_n$-invariants is taking a summand, and the definitions of both purity and perversity are manifestly preserved by summands.

We now apply properties of the Radon transform. 

The Radon transform sends complexes pure of weight $w$ to complexes pure of weight $w+n-1$ because pullback along the smooth morphism $p_2$ and pushforward along the smooth morphism $p_1$ both preserve weights by \cite[Stability 5.1.14(i,i*)]{bbdg}, since $p_{1*} = p_{1!}$ and $p_{2}^! = p_2^* [ 2 (n-1) ] (n-1)$, while shifting the degree by $n-1$ shifts weights by $n-1$. Thus $K_n$ is pure of weight $2n-1$.

To check that $K_n$ is perverse, it suffices to show that its perverse homology sheaves in degrees other than $0$ vanish. Let $d: \mathbb P^n \to \operatorname{Spec} \mathbb F_q$ be the unique map to a point.  We can view the cohomology group  $H^j (C^{(n)}_{\overline{\mathbb F}_q}, ( sym_* \mathcal F^{\boxtimes n})^{S_n} [n]) $ as a sheaf on $ \operatorname{Spec} \mathbb F_q$ (because it admits an action of the Galois group of $\mathbb F_q$).

 By \cite[IV, Lemma 2.2]{KiehlWeissauer}, we have a formula for the perverse homology sheaf if $i \neq 0$. More precisely, we have
 \[ {}^p \mathcal H^i ( K_n) = d^* H^{i+1} (C^{(n)}_{\overline{\mathbb F}_q}, ( sym_* \mathcal F^{\boxtimes n})^{S_n} [n]) ]   [n] \] if $i>0$ and
  \[ {}^p \mathcal H^i ( K_n) = d^* H^{i-1} (C^{(n)}_{\overline{\mathbb F}_q}, ( sym_* \mathcal F^{\boxtimes n})^{S_n} [n]) ]   [n] \] if $i<0$.
  
 Thus, to show $K_n$ is perverse, it suffices to show that $H^{j} (C^{(n)}_{\overline{\mathbb F}_q}, ( sym_* \mathcal F^{\boxtimes n})^{S_n} [n]) =0 $ where $j= i+1$ for $j >0$ and $j= i-1$ for $j<0$. In fact, we will show this for $j \neq 0$, which is a stronger statement.
To do this, note that

\[ H^{j} (C^{(n)}_{\overline{\mathbb F}_q}, ( sym_* \mathcal F^{\boxtimes n})^{S_n} [n])  = H^{j+n} (C^{(n)}_{\overline{\mathbb F}_q}, ( sym_* \mathcal F^{\boxtimes n})^{S_n} ) \] is a summand of 
\[ H^{j+n} (C^{(n)}_{\overline{\mathbb F}_q}, ( sym_* \mathcal F^{\boxtimes n})= H^{j+n} ( C^n _{\overline{\mathbb F}_q}, \mathcal F^{\boxtimes n} ) .\]

Because $\mathcal F$ is geometrically irreducible of rank $2$, it has no global monodromy invariants or coinvariants. Thus $H^i ( C _{\overline{\mathbb F}_q}, \mathcal F )$ vanishes for $i \neq 1$. It follows from the K\"{u}nneth formula that its $n$-fold tensor product $ H^{j+n} ( C^n _{\overline{\mathbb F}_q}, \mathcal F^{\boxtimes n} ) $ vanishes for $j+n \neq n$, and thus for $j \neq 0$. Because that cohomology group vanishes, its summand  $H^{j} (C^{(n)}_{\overline{\mathbb F}_q}, ( sym_* \mathcal F^{\boxtimes n})^{S_n} [n]) $ also vanishes, as desired.
\end{proof}

\section{Calculating the characteristic cycle}\label{calculating-characteristic-cycle}

In this section, we will calculate the characteristic cycle of the complex $K_n$  defined in Section \ref{geometric-setup}. The most difficult part is calculating the multiplicity of the zero section, which we can express as the rank (equivalently, the Euler characteristic) at the generic point. We calculate this by first determining the Euler characteristic of a simpler model sheaf in Lemma \ref{constant-sheaf-generic} and then relating this model sheaf to $K_n$ in Lemma \ref{characteristic-cycle-Whittaker}. The other main ingredient of Lemma \ref{characteristic-cycle-Whittaker} is Lemma \ref{cycle-rank-two}, which is a special case of Theorem \ref{characteristic-cycle-sym}.

It will be helpful to reinterpret some of the concepts of Section \ref{symmetric-powers} in the special case where $C = \mathbb P^1$.

Viewing points of $C^{(n)}$ as nonzero sections of $H^0( C, \mathcal O(n))$ up to scaling, we can view $A_{(e_x),(w_k)}$ as the set of sections of the form $\left( \prod_{x \in \Sing} l_x^{e_x} \right) \prod_{k=1}^{\infty} f_k^k$ where $l_x$ is a fixed section of $\mathcal O(1)$ vanishing at $x$ and $f_k$ is an arbitrary section of $\mathcal O(w_k)$.

We can view the tangent space at a point of $C^{(n)}$ corresponding to a section $g$ as the space of sections of $\mathcal O(n)$ modulo $g$. The isomorphism to our earlier description of the tangent space as $H^0 ( \mathcal I^\vee/\mathcal O)$ is given by dividing a degree $n$ polynomial by $g$, producing a section of the dual of the ideal sheaf generated by $g$, modulo $\mathcal O$.

Over the point  $\left( \prod_{x \in \Sing} l_x^{e_x} \right) \prod_{k=1}^{\infty} f_k^k$, the closed set $B_{(e_x),(w_k)}$ is the set of linear forms on the tangent space that vanish on all elements whose divisor of poles is at most the divisor of $\prod_k f_k$, which is the set of linear forms vanishing on the image of the multiplication-by-$\left( \prod_{x \in \Sing} l_x^{e_x} \right) \prod_{k=1}^{\infty} f_k^{k-1}$ map from $H^0( \mathbb P^1, \mathcal O(\sum_k w_k)) $ to $ H^0( \mathbb P^1, \mathcal O(n)) .$

Recall that for $\mathbb P^\vee$ the space of nonzero linear forms on $H^0(C, \mathcal O(n))$ up to scaling, i.e. the projective dual space to $C^{(n)}$, we defined $Y \subseteq \mathbb P^\vee \times C^{(n)}$ the graph of the universal family of hyperplanes (i.e. the locus of pairs of a section and a linear form vanishing on that section) and $p_1: Y \to \mathbb P^\vee , p_2: Y \to  C^{(n)}$ the projection maps.

For a middle extension sheaf $\mathcal F$, the conductor $N$ of $\mathcal F$ is defined as  $ \sum_{ x\in \Sing} c_x [x]$. In particular, the support $|N|$ of $N$ is $\Sing$. 

\begin{lemma}\label{constant-sheaf-generic} Let $\mathcal  F'  =\mathbb Q_\ell^2 \oplus \bigoplus_{x \in \Sing} \delta_x^{c_x}[-1]$. The Euler characteristic of the stalk at the generic point of 
\[  R p_{1*} p_2^* ( sym_*\mathcal F^{'\boxtimes n})^{S_n} [2n-1]   \]
is the coefficient of $u^n$ in the generating series $-\frac{2u (1-u)^{\sum_{x\in \Sing} c_x} } {(1-u)^4 (1+u)}$. \end{lemma}

\begin{proof}
For $a,b$ natural numbers with $a+b + \sum_{x \in |N| } e_x = n$, let $ f_{a,b,(e_x)}$ be the map $C^{(a)} \times C^{(b)} \to C^{(n)}$ given by adding the two divisors together and then adding $\sum_x e_x [x]$.

First we check that \begin{equation}\label{pushforward-constant-sheaf}( sym_* \mathcal F^{'\boxtimes n})^{S_n}[n]  =  \bigoplus_ {\substack{( e_x) :  \Sing \to \mathbb N \\  e_x \leq c_x \\ a, b \in \mathbb N \\a+b+ \sum_{ x \in |N|} e_x   =n } }   \left(  f_{a,b,(e_x) *} \mathbb Q_\ell [a+b] \right)^{\prod_{x \in \Sing } \binom{c_x }{ e_x} }  .\end{equation}
To do this, write $\mathcal F' = \bigoplus_{i=1}^{ 2+ \sum_{x \in |N| } c_x}  \mathcal F_i$ where $\mathcal F_1 = \mathcal F_2= \mathbb Q_\ell$ and the remaining summands are skyscraper sheaves supported at the points of $x$. Then $\mathcal F^{' \boxtimes n}$ is a sum over $n$-tuples $t_i$ of numbers from $1$ to $ 2+ \sum_{x \in |N| } c_x$ of $\boxtimes_{i=1}^n \mathcal F_{t_i}$. Thus $sym_* \mathcal F^{'\boxtimes n}$ is the sum over these $n$-tuples $t_i$ of $( sym_* \boxtimes_{i=1}^n \mathcal F_{t_i})$.  Then $S_n$ acts by permuting the $n$-tuples, so the $S_n$-invariants of $sym_* \mathcal F^{'\boxtimes n}$ can be viewed as a sum over unordered $n$-tuples of the $S_n$-invariants of the sum of $sym_* \boxtimes_{i=1}^n \mathcal F_{t_i}$ for all orderings $t_i$ of that unordered tuple. Viewing this as an induced representation, the $S_n$-invariants of this sum will equal the invariants of $sym_* \boxtimes_{i=1}^n \mathcal F_{t_i}$ under the stabilizer in $S_n$ of this tuple $t_i$.

If any number greater than two occurs at least twice among the $t_i$, a transposition swapping two occurances will act as $-1$ on $sym_* \boxtimes_{i=1}^n \mathcal F_{t_i}$, because the tensor product of two skyscraper sheaves is a single skyscraper sheaf so the action is by the Koszul sign for a tensor product of complexes, which is $-1$ because the skyscraper sheaves are in degree $1$. Hence there are no invariants under the stabilizer unless each $t_i$ greater than $2$ occurs at most once. For each unordered tuple, let $a$ be the number of $i$ with $t_i=1$, $b$ be the number of $i$ with $t_i=2$, and $e_x$ be the number of $i$ with $\mathcal F_{t_i} = \delta_x[-1]$. Then the number of unordered tuples attaining $(a,b,(e_x))$ is $\prod_{x \in \Sing } \binom{c_x }{ e_x}$, and each tuple produces the $S^a \times S^b$ invariants of the pushforward from $C^a \times C^b$ of $\mathbb Q_\ell$, which is the pushforward from $C^{(a)} \times C^{(b)}$ along $ f_{a,b,(e_x)}$ of $\mathbb Q_\ell$.

Having verified Equation \eqref{pushforward-constant-sheaf}, we next observe that the Euler characteristic of the Radon transform of  $f_{a,b,(e_x) *} \mathbb Q_\ell [a+b] $ is $(-1)^{n-1+a+b}$ times the Euler characteristic of the inverse image under $f_{a,b,(e_x)}$ of a general hyperplane. The inverse image of a general hyperplane is a $(1,1)$-hypersurface in $C^{(a)} \times C^{(b)} = \mathbb P^a \times \mathbb P^b$. Viewing this as a $\mathbb P^{\max(a,b)}$-bundle on $\mathbb P^{\min(a,b)}$, we see that if the hypersurface does not contain any fiber, then it is a $\mathbb P^{\max(a,b)-1}$ bundle on $\mathbb P^{\min(a,b)}$ and hence has Euler characteristic $(\min(a,b)+1) (\max(a,b)) =ab+ \max(a,b)$. 

To check that the inverse image of a general hyperplane does not contain any fiber, we must check that for a generic linear form on polynomials of degree $n$, there is no polynomial $f$ of degree $\min(a,b)$ such that the linear form vanishes on all multiples of $f \prod_{x} l_x^{e_x}$ by polynomials of degree $\max(a,b)$. The space of such linear forms has dimension $\min(a,b) + \sum_x e_x$ and the choices of polynomials, up to scaling, are $\min(a,b)$-dimensional, so dimension of the space of linear forms is $2 \min(a,b) + \sum_x e_x \leq a+b  + \sum_x e_x =n$, which is less than the $n+1$-dimensional space of all linear forms, so indeed a generic linear form does not vanish in this way, and the general Euler characteristic is $ab+ \max(a,b)$, so the contribution to the total Euler characteristic is $(-1)^{n-1+a+b} (ab+ \max(a,b))$. Hence the total Euler characteristic is
\[ \sum_{\substack{ (e_x) :  \Sing \to \mathbb N \\  e_x \leq c_x \\ a, b \in \mathbb N \\ a+b +\sum_{x \in |N|} e_x   =n } } (-1)^{n-1+a+b} \left (\prod_{x \in \Sing } \binom{c_x }{ e_x} \right)   (ab+ \max(a,b)) \]

Using the generating series \[ F(u,v)= \sum_{a,b\in \mathbb N}  (ab+ \max(a,b)) u^{a} v^{b},\] this is the coefficient of $u^n$ in $ - (1-u )^{\sum_x c_x} F(u,u)$. Because $ab+\max(a,b)=0$ if $\min(a,b)=-1$, we have
\[ (1-uv) F(u,v) =  \sum_{a,b\in \mathbb N}  (ab+ \max(a,b) - (a-1) (b-1) - \max(a-1,b-1) ) u^{a} v^{b}= \sum_{a,b \in \mathbb N} (a+b-1+1) u^a v^b\] \[ =\sum_{a,b\in \mathbb N } a u^a v^b + \sum_{a,b,\in \mathbb N} b u^a v^b =  \frac{u}{ (1-u)^2 (1-v)} +\frac{v}{(1-v)^2(1-u)} \] so \[F(u,u) = \frac{ 2u}{ (1-u)^3 (1-u^2)} = \frac{ 2u}{ (1-u)^4 (1+u)} .\] Plugging in,  we see that the Euler characteristic is the coefficient of $u^n$ in $-\frac{2u (1-u)^{\sum_{x\in \Sing} c_x} } {(1-u)^4 (1+u)} $. \end{proof}

\begin{lemma}\label{cycle-rank-two} The characteristic cycle of $  ( sym_* \mathcal F^{\boxtimes n})^{S_n}[n]$ on $C^{(n)}$ is \[ \sum_{\substack{ e_x :  \Sing \to \mathbb N \\  e_x \leq c_x \\ w_k : \{1,2 \} \to \mathbb N \\ \sum_{x \in |N|} e_x + w_1 + 2w_2 =n } }2^{w_1} \left (\prod_{x \in \Sing } \binom{c_x }{ e_x} \right)  [ B_{(e_x),(w_k)} ]  .\] \end{lemma}

\begin{proof} To prove this, we apply Theorem \ref{characteristic-cycle-sym}.

We have  \[( sym_* \mathcal F^{\boxtimes n})^{S_n}[n] = ( sym_* (\mathcal F[1])^{\boxtimes n} \otimes \sgn )^{S_n} \] so we may take $K= \mathcal F[1]$ and $\rho=\sgn$. Then Theorem \ref{characteristic-cycle-sym} guarantees that the characteristic cycle of this complex is \[ \sum_{\substack{ e_{x} :  \Sing \to \mathbb N \\ w_{k} : \mathbb N^+ \to \mathbb N \\  \sum_{ x\in \Sing} e_{x} + \sum_{k =1}^{\infty} k w_{k} = n   }}M_{ K, \rho} ( (e_x), (w_k))    [ B_{(e_x),(w_k)} ]     \] so we must check that $M_{K, \rho}$ vanishes if $e_x > c_x$ or $w_k >0$ for $k>2$ and it equal to $2^{w_1} \left (\prod_{x \in \Sing } \binom{c_x }{ e_x} \right) $ otherwise. By definition,

\[M_{ K, \rho} ( (e_x), (w_k)) = 
\dim \left (  \bigotimes_{x \in \Sing }  ( \mathbb C^{c_{x}})^{\otimes e_{x} }  \otimes  (\mathbb C^{\operatorname{rank}} ) ^{ \otimes \sum_{k = 1 }^{\infty}  k w_{k }  }    \otimes \rho \right)^{ \prod_{x \in \Sing}  S_{e_{x} } \times  \prod_{k =1}^{\infty}  S_{k  }^{w_{k } } } .   \] 

Tensoring with $\sgn$ and taking symmetric group invariants is the same as taking a wedge power, and $\rank=2$, so this is
\[ \dim \left (  \bigotimes_{x \in \Sing} \wedge^{e_x}  ( \mathbb C^{c_{x}}) \otimes \bigotimes_{k=1}^{\infty}  \left( \wedge^k ( \mathbb C^{2} )\right)^{\otimes w_k} \right)\]
\[ =\left( \prod_{x \in \Sing} \binom{c_x }{ e_x}  \right) 2^{w_1} 1^{w_2} \prod_{k=2}^{\infty} 0^{w_k} \]
as desired. 
\end{proof}

We can view the tangent space of $\mathbb P^\vee$ at a point corresponding to a linear form $l$ on polynomials as the space of linear forms on polynomials modulo $l$. Hence we can view the cotangent space at this point as the space of polynomials which $l$ vanishes on.

Let $B_{(e_x),(w_k)}^\vee$ be the closed subset of $T^* \mathbb P^\vee$ defined as the set of pairs of a linear form vanishing on all polynomial multiples of $\left( \prod_{x \in \Sing} l_x^{e_x} \right) \prod_{k=1}^{\infty} f_k^{k-1}$ with a cotangent vector that is a scalar multiple of $\left( \prod_{x \in \Sing} l_x^{e_x} \right) \prod_{k=1}^{\infty} f_k^{k}$, for $f_k$ a polynomial of degree $w_k$.

\begin{lemma}\label{characteristic-cycle-Whittaker} The characteristic cycle of $ K_n$ is \[ \left( 2\sum_{k=0}^{n-1}  \binom{ \deg N -4 }{ k} \right) [\mathbb P^\vee ] + \sum_{\substack{ e_x :  \Sing \to \mathbb N \\  e_x \leq c_x \\ w_k : \{1,2 \} \to \mathbb N \\ \sum_{x \in |N|} e_x + w_1 + 2w_2 =n } }2^{w_1} \left (\prod_{x \in \Sing } \binom{c_x }{ e_x} \right)  [ B_{(e_x),(w_k)}^\vee ] .\]

\end{lemma} 

\begin{proof} This follows from \cite[Corollary 6.12]{saito1}, which says that the characteristic cycle of the Radon transform of a complex is the Legendre transform of the characteristic cycle of that complex, where the Legendre transform of a cycle defined in \cite[(6.16)]{saito1}. By definition, $K_n$ is the Radon transform of $( sym_* \mathcal F^{\boxtimes n})^{S_n}[n]$, so by Lemma \ref{cycle-rank-two} it suffices to compute the Legendre transform of  \[ \sum_{\substack{ e_x :  \Sing \to \mathbb N \\  e_x \leq c_x \\ w_k : \{1,2 \} \to \mathbb N \\ \sum_{x \in |N|} e_x + w_1 + 2w_2 =n } }2^{w_1} \left (\prod_{x \in \Sing } \binom{c_x }{ e_x} \right)  [ B_{(e_x),(w_k)} ]  .\]

We will do this in two parts. We will first check that the Legendre transform of $[ B_{(e_x),(w_k)} ]$ is $[ B_{(e_x),(w_k)}^\vee ] $ plus some multiple of the zero-section, and we will then compute the multiple of the zero-section. The first part is \cite[Corollary 1.2.4]{saito-notes}.


For the second part, we observe that the multiplicity of the zero section is given by some intersection-theoretic formula involving the characteristic cycle. We can therefore replace $( sym_* \mathcal F^{\boxtimes n})^{S_n}[n]$ by any complex which has the same characteristic cycle. Letting $\mathcal F'  =\mathbb Q_\ell^2 + \sum_{x \in \Sing} \delta_x^{c_x}[-1]$, we observe that $\mathcal F'[1]$ is a perverse sheaf and has the same rank and conductors as $K[1]$, so by Theorem \ref{characteristic-cycle-sym}  \[ CC ( ( sym_* \mathcal F^{\boxtimes n})^{S_n}[n] ) = CC ( ( sym_* \mathcal F^{'\boxtimes n})^{S_n}[n] ).\]

Now the multiplicity of the zero-section in the characteristic cycle of the Radon transform of  $(sym_* \mathcal F^{'\boxtimes n})^{S_n}[n]$ is the $(-1)^n$ times the generic Euler characteristic of that Radon transform.   Hence by Lemma \ref{constant-sheaf-generic} the multiplicity of the zero-section is $(-1)^n$ times the coefficient of $u^n$ in \[   -2u  (1-u)^{ \deg N -4} / (1+u) \] which is the coefficient of $u^n$ in $2u (1+u)^{ \deg N- 4} / (1-u)$ which by the power series of $1/(1-u)$ and the binomial theorem is
\[ 2\sum_{k=0}^{n-1}  \binom{ \deg N -4 }{ k} .\]

\end{proof} 

\section{Calculating the polar multiplicities}\label{calculating-polar-multiplicities}

In this section, we calculate the polar multiplicities of $K_n$. First, we recall the definition of the polar multiplicities from \cite[\S3]{mypaper}.

\begin{defi} Let $Y$ be a smooth variety with a map $f$ to a variety $X$ (which may be the identity), and let $x$ be a point on $X$. Let $C_1, C_2$ be algebraic cycles on $Y$ such that $\dim C_1 + \dim C_2 = \dim Y$ and $C_1 \cap C_2 \cap f^{-1}(x)$ is proper. Assume that all connected components of $C_1 \cap C_2$ are either contained in $f^{-1}(x)$ and proper or disjoint from $f^{-1}(x) $. We define their intersection number locally at $x$  \[ (C_1,C_2)_{Y, x} \] to be the sum of the degrees of the refined intersection $C_1 \cdot C_2$ \cite[p. 131]{Fulton} on all connected components of $C_1 \cap C_2$ contained in $f^{-1}(x)$. \end{defi}

\begin{defi}\label{polar-multiplicity-2} Let $X$ be a smooth variety. Let $C$ be a $\mathbb G_m$-invariant cycle on the cotangent bundle $T^* X$ of $X$ of dimension $\dim X$ and let $x$ be a point on $X$.

For $0 \leq i< \dim X$, let $Y$ be a sufficiently general smooth subvariety of $X$ of codimension $i$ passing through $x$ and let $V$ be a sufficiently general sub-bundle of $T^* X$ over $Y$ with rank $i+1$. Define the $i$th polar multiplicity of $C$ at $x$ to be the intersection number \[ ( \mathbb P(C) , \mathbb P(V))_{ \mathbb P(T^* X), x} \] where $\mathbb P(T^* X)$ is the projectivization of the vector bundle $T^* X$.

Here ``sufficiently general" means that the strict transform of $\mathbb P(V)$ in the blowup of $\mathbb P(T^* X)$ at the fiber over $x$ does not intersect the strict transform of $\mathbb P(C)$ in that same blowup within the fiber over $x$.

For $i=\dim X$, define the $i$th polar multiplicity of $C$ at $x$ to be the multiplicity of the zero section in $C$. \end{defi}

\vspace{10pt}

To calculate the polar multiplicities, we will pass to a local model in which they are easier to compute. In fact the local model will be an affine space $\mathbb A^{d+1}$, and the cycle on $T^* \mathbb A^{d+1}$ which we compute the multiplicities of will be invariant under the scaling map of $\mathbb A^{d+1}$. Using this scale-invariance, we reduce the local intersection theory problem from Definition \ref{polar-multiplicity-2} to a global intersection theory problem on projective space, which reduces in Lemma \ref{local-polar-multiplicity} to a straightforward calculation with the Chern classes of vector bundles. In Lemmas \ref{local-model-schematic-smooth} through \ref{local-model-level} we set up this local model and explain its relationship to the polar multiplicities of $CC(K_n)$. In Lemmas \ref{local-polar-generating-formula} through Lemma \ref{final-Radon-bound} we use this relationship to turn Lemma \ref{local-polar-multiplicity} into a formula for the original polar multiplicities, and finally a bound for the trace function of $K_n$ (in Lemma \ref{final-Radon-bound}) and thus the function $f$ (in Lemma \ref{first-newform-bound}). This requires transforming the bound to a combinatorially more convenient form, which we do by defining an appropriate generating function.

\vspace{5pt}
Now let us explain the local model. View $\mathbb A^{d+1}$ as the space of linear forms on $H^0(\mathbb P^1, \mathcal O(d))$, and view the cotangent space at $\mathbb A^{d+1}$ as $H^0(\mathbb P^1, \mathcal O(d))$.

We will define a cycle $B_{d,r}'$ in the cotangent bundle of $\mathbb A^{d+1}$ for each $0 \leq r \leq d/2$. To define this cycle, consider the space $\mathbb P^{r} \times \mathbb P^{d-2r}$ parameterizing pairs of $ f_1 \in H^0(\mathbb P^1, \mathcal O(d-2r))$ and $ f_2 \in H^0( \mathbb P^1, \mathcal O(r)) $, both nonzero and up to scaling. On this vector space, consider the vector bundle \[ \left( H^0 ( \mathbb P^1 , \mathcal O(d) )  / \left( f_2 \cdot H^0( \mathbb P^1, \mathcal O(d-r)) \right) \right) ^\vee \oplus  \mathcal O_{\mathbb P^r \times \mathbb P^{d-2r}} (1,2)\] where $f_2 \cdot H^0( \mathbb P^1, \mathcal O(d-r)) $ is the subspace of $H^0 ( \mathbb P^1 , \mathcal O(d) )$ consisting of products of $f_2$ with an element of $H^0( \mathbb P^1, \mathcal O(d-r))$ and $ \mathcal O_{\mathbb P^r \times \mathbb P^{d-2r}} (1,2)$ is the subspace of $H^0(\mathbb P^1, \mathcal O(d))$ consisting of scalar multiples of $f_1 f_2^2$. We define $B_{d,r}'$ as the cycle on \[T^* \mathbb A^{d+1} = H^0(\mathbb P^1, \mathcal O(d))^\vee \times H^0(\mathbb P^1, \mathcal O(d)) \] given by the pushforward of this vector bundle to $H^0(\mathbb P^1, \mathcal O(d))^\vee \times H^0(\mathbb P^1, \mathcal O(d))$ under the product of the natural maps \[ \left( H^0 ( \mathbb P^1 , \mathcal O(d) )  / \left( f_2 \cdot H^0( \mathbb P^1, \mathcal O(d-r) ) \right) \right) ^\vee  \to  H^0 ( \mathbb P^1 , \mathcal O(d) )^\vee \]  and \[  \mathcal O_{\mathbb P^r \times \mathbb P^{d-2r}} (1,2) \to H^0 ( \mathbb P^1, \mathcal O(d)).\]

Because $B_{d,r}'$ is the pushforward of a vector bundle on an irreducible variety, $B_{d,r}'$ is irreducible.

Consider the map $loc_d$ from $\mathbb A^{d+1}$ to the moduli stack $\Bun_2$ of vector bundles on $\mathbb P^1$ that sends a linear form on $H^0 ( \mathbb P^1, \mathcal O(d))$ to the extension $0 \to \mathcal O \to V \to \mathcal O( d+2) \to 0$ arising from the corresponding class in \[ H^0 ( \mathbb P^1, \mathcal O(d))^\vee = H^1(\mathbb P^1, \mathcal O(-d-2)) = \operatorname{Ext}^1 ( \mathcal O(d+2), \mathcal O).\]

\begin{lemma}\label{local-model-schematic-smooth} The map $loc_d$ is schematic, locally of finite type, and smooth. \end{lemma}

\begin{proof} The map $loc_d$ is schematic and locally of finite type because it is a map from a scheme of finite type to an Artin stack. 

Let $L_1 = \mathcal O$ and let $L_2 = \mathcal O(d+1)$. For a vector bundle $V$, the tangent space to $V$ in $\Bun_2$ is given by $H^1( \mathbb P^1, V \otimes V^\vee )$. If we write $V$ as an extension $0 \to L_1 \to V \to L_2 \to 0$, the tangent space to the space of extensions of $L_2$ by $L_1$  is $\operatorname{Ext}^1 ( L_2, L_1) = H^1 ( \mathbb P^1, L_1 \otimes L_2^{-1} )$. Furthermore, the derivative at $V$ of the map from the space of the space of extensions to $\Bun_2$ is the map $ H^1 ( \mathbb P^1, L_1 \otimes L_2^{-1} ) \to H^1( \mathbb P^1, V \otimes V^\vee )$ induced by the map $L_1 \otimes L_2^{-1} \to V \otimes V^{\vee}$ given by embedding $L_1$ into $V$ and $L_2^{-1}$ into $V^\vee$. (This can be checked by working with vector bundles over $k[\epsilon]/\epsilon^2$, say.)

Hence the derivative map is surjective as long as \[(H^1 (\mathbb P^1, (V \otimes V^{\vee} ) / L_1 \otimes L_2^{-1} ))=0.\] The quotient \[(V \otimes V^{\vee} ) /( L_1 \otimes L_2^{-1} )\] is the extension of $L_2 \otimes L_1^{-1}$ by $L_1 \otimes L_1^{-1} + L_2 \otimes L_2^{-1}$, and so this cohomology group vanishes as soon as $\deg L_2 - \deg L_1 > -2$, which is automatic in our case as $\deg L_2 = d+2$ and $\deg L_1 =0$. \end{proof}

\begin{defi}\label{Wr-notation} Let $W_r$ be the space of linear forms on $H^0(\mathbb P^1, \mathcal O(d))$ such that there exists a nonzero $f_2$ in $H^0 (\mathbb P^1, \mathcal O(r))$ where the linear form vanishes on all multiples of $f_2$.\end{defi}

 Because $W_r$ is  the projection from $ \mathbb P ( H^0 ( \mathbb P^1, \mathcal O(r)) )\times \mathbb A^{d+1}$ to $\mathbb A^{d+1}$ of a closed set, $W_r$ is closed.


\begin{lemma}\label{local-model-conormal} The closed set $B_{d,r}'$ is the conormal bundle to $W_r$. \end{lemma}

\begin{proof} 
The set $W_r$ is the projection from $ \mathbb P ( H^0 ( \mathbb P^1, \mathcal O(r)) \times \mathbb A^{d+1}$ to $\mathbb A^{d+1}$ of the set $Z_r$ of pairs of a polynomial $f_2$ and a linear form vanishing on multiples of $f_2$. For a point $(f_2, \alpha)$ in this closed set, the projection of the tangent space of $Z_r$ to $\mathbb A^{d+1}$ is the set of coefficients of $\epsilon$ in linear forms  $H^0 (\mathbb P^1, \mathcal O(n)) \otimes k[\epsilon]/(\epsilon^2) \to k[\epsilon]/(\epsilon^2)$ that vanish on multiples of $f_2+ \epsilon f_3$ for some $f_3$ and that are congruent mod $\epsilon$ to $\alpha$. In particular, because these linear forms vanish on all multiples of $f_2 + \epsilon f_3$, they vanish on all multiples of $(f_2+ \epsilon f_3) (f_2 - \epsilon f_3) = f_2^2$ and thus are contained in the space of linear forms vanishing on all multiples of $f_2^2$.

For $\alpha$ a generic point of $W_r$, the map $Z_r \to W_r$ is \'{e}tale over $\alpha$, so the tangent space of $W_r$ is contained in the space of linear forms vanishing on multiples of $f_2^2$. Since $W_r$ is $2r$-dimensional, so the tangent space of $W_r$ is $2r$-dimensional, and the dimension of the space of linear forms vanishing on all multiples of $f_2^2$ is $2r$, the tangent space of $W_r$ at generic point is equal to the space of linear forms vanishing on all multiples of $f_2^2$. Hence the tangent space of $W_r$ at a generic point is the perpendicular space to the fiber of $B_{d,r}'$ over that point. 

%
%

The conormal bundle of a singular variety is defined as the closure of the conormal bundle of its smooth locus. Because $B_{d,r}'$ is an irreducible closed variety, it is the closure of any open subset of itself. So because it is equal to the conormal over an open set, it is equal to the conormal bundle everywhere. 
%
%
%

\end{proof}

\begin{lemma}\label{pullback-bun-stratification} For $0 \leq r <d/2$, $W_r$ is the inverse image under $loc_d$ of the locus in $\Bun_2$ consisting of line bundles $\mathcal O(a) + \mathcal O(b)$ with \[ a \leq r  \leq 2+d-r \leq b \] \end{lemma}

\begin{proof}  Let $V = loc_d(\alpha)$. Then $V$ lies in an exact sequence $0 \to \mathcal O \to V \to \mathcal O(d+2) \to 0 $, so $V$ has degree $d+2$.  We can write $V$ as $\mathcal O(a) + \mathcal O(b)$ with $ a \leq r  \leq 2+d-r \leq b $ if and only if $V$ admits a nonzero map from $\mathcal O(d-r+2)$. (The if direction because if there is a map from $\mathcal O(d-r+2)$ to $V$, the saturation of its image is a subbundle with degree $\geq d-r+2$, and then the quotient has degree $\leq r < d-r+2$ so the extension splits. The only if direction is because we can map $\mathcal O(d-r+2)$ to $\mathcal O(b)$.)

Composing a map  $\mathcal O(d-r+2)\to V$ with the map $V \to \mathcal O(d+2) $ from the short exact sequence, we get a map $\mathcal O(d-r+2) \to \mathcal O(d+2)$. If this map is zero, we get a nonzero map $\mathcal O(d-r+2) \to \mathcal O$, which is impossible as $d-r+2>0$.  

We can view the set of nonzero maps $\mathcal O(d-r+2) \to \mathcal O(d+2)$ as the set of nonzero sections $f_2$ of $\mathcal O(r)$. Given any such map, we can pull back the class $\alpha$ in $\Ext^1 (\mathcal O(d+2), \mathcal O)$ to obtain a class in $\Ext^1( \mathcal O(d-r+2), \mathcal O)$, which concretely corresponds to the fiber product of $V$ and $\mathcal O(d-r+2)$ over $\mathcal O(d+2)$. This new extension splits if and only if $f_2$ lifts to a map $\mathcal O(d-r+2)\to V$ (because a lift of a map $A\to C$ along a map $B\to C$ is equivalent to a section of the natural map $A \times_C B \to A$ by the universal property.)

By Serre duality, the pulled-back class in $\Ext^1 (\mathcal O(d+2), \mathcal O)$ vanishes if and only if the linear form on $\mathcal O(d-r)$ induced by composing $\alpha$ with multiplication by $f_2$ vanishes, which happens only if $\alpha$ vanishes on all multiples of $f_2$.
\end{proof}


We introduce a quantity $d_{\alpha, (e_x)} $ associated to a linear form $\alpha$ that will determine which $d$ is appropriate to use to define the local model.

\begin{defi}\label{mdalpha-notation} Let $\alpha$ be a linear form on polynomials of degree $n$. Fix $(e_x): \Sing \to \mathbb N$

 Let $m_{\alpha, ( e_x)} $ be the minimum $m$ such that $\alpha$ vanishes on all multiples of $f \prod_{x\in \Sing } l_x^{e_x} $ by an element of $H^0 (\mathbb P^1, \mathcal O (n- m - \sum_x e_x))$, for some $f$ of degree $m$.
 
 Let $d_{\alpha, (e_x)} = n -2m_{ \alpha, (e_x)} - \sum_x e_x$.
\end{defi} 

\begin{lemma}\label{local-model-level}  The $i$th polar multiplicity of $B_{(e_x),(w_k)}^\vee$ at a linear form $\alpha\in \mathbb P^\vee$ is equal to the $i+1-2m_{\alpha, (e_x)} - \sum_x e_x$th polar multiplicity of $B_{d,r}'$ at $0 \in \mathbb A^{d+1}$, where $d = d_{\alpha, (e_x)} $ and $r = w_2 - m$.  In particular it vanishes if $r<0$. \end{lemma}

\begin{proof} In this proof, we simplify notation by writing $m$ for $m_{\alpha, (e_x)}$. 

 The vanishing if $r<0$ is clear because if $m>w_2$ then $B_{(e_x), w_k}^{\vee}$ does not intersect the fiber over $\alpha$, because $\alpha$ does not vanish on all multiples of any polynomial of degree $w_2$.  Hence we may assume $m \leq w_2 \leq (n- \sum_x e_x) /2$.

We have a map  \[ \mathbb P^\vee  = ( H^0 ( \mathbb P^1, \mathcal O(n) )^\vee - \{0\} )/ \mathbb G_m \to (H^0 ( \mathbb P^1, \mathcal O(n- \sum_x e_x) ))^\vee / \mathbb G_m\] where we map linear forms on $\mathcal O(n)$ to linear forms on $\mathcal O(n- \sum_x e_x) $ by composing with multiplication by $\prod_x l_x^{e_x}$.

The map $loc_{n- \sum_x e_x}  : H^0 (\mathbb P^1, \mathcal O(n- \sum_x e_x ))^\vee \to \Bun_2$ is invariant under scaling $H^0 (\mathbb P^1, \mathcal O(n))^\vee$, because scaling an Ext class gives an isomorphic extension. Hence we can compose the map\[ \mathbb P^\vee  \to H^0 ( \mathbb P^1, \mathcal O(n- \sum_x e_x) )^\vee / \mathbb G_m\]  with the descended form of $loc_{n- \sum_x e_x} $ to obtain a map $loc_{n ,(e_x)}: \mathbb P^\vee \to \Bun_2$.

 The linear form $\alpha$ is sent to a vector bundle by $loc_{n, (e_x)}$ . We can write this vector bundle as a sum of line bundles $\mathcal O(a) + \mathcal O(b)$ with $a\leq b$. By the definition of $m$ and Lemma \ref{pullback-bun-stratification}, we have $a \leq m$ but $a > m-1$, so we must have $a=m$. Then we have $b= n+2-m$. 

Let $loc_{n,(e_x), m}$ be $loc_{n, (e_x)}$ but with the resulting vector bundle twisted by $\mathcal O(-m)$, so $loc_{n,(e_x),m}(\alpha) = \mathcal O + \mathcal O(n+2-2m-\sum_x e_x )$. Letting $d=n-2m-\sum_x e_x$, we see that $d\geq 0$ by our earlier assumption on $m$, and that $loc_d(0) = \mathcal O + \mathcal O(n+2-2m -\sum_x e_x)$.  Let \[ Y_{n,(e_x), m} = \left( H^0 (\mathbb P^1, \mathcal O(d))^{\vee} \right)  \times_{\Bun_2}  \mathbb P^\vee  ,\] using $loc_d$ and $loc_{n,(e_x),m}$ to define the fiber product. Let $\mu_1$ and $\mu_2$ be the induced maps $Y_{n,(e_x), m} \to H^0 (\mathbb P^1, \mathcal O(d))^{\vee} $ and $\mathbb P^\vee$, respectively. 
 Let $y \in Y_{n,m}$ be a point sent to $0$ by $\mu_1$ and to $\alpha$ by $\mu_2$.
 
 By Lemma \ref{local-model-schematic-smooth},$Y_{n,(e_x), m} $ is a scheme and $\mu_1$ and $\mu_2$ are smooth.

Let us check that \[ \mu_1^! B_{d,r}' = \mu_2^! B_{(e_x), (w_k)}^\vee.\]

By Lemma \ref{local-model-conormal}, both $\mu_1^! B_{d,r}' $ and $mu_2^! B_{(e_x), (w_k)}^\vee$  are conormal bundles to their supports, so their pullbacks under a smooth map are the conormal bundles of the pullbacks of their support. By Lemma \ref{pullback-bun-stratification} the support of $\mu_1^! B_{d,r}'$ is the pullback of $W_r$ under $loc_d\circ \mu_1 $ and the support of $\mu_2^! B_{(e_x), (w_k)}$ is the pullback of $W_r$ under $loc_{n, (e_x),m } \circ \mu_2$. By the commutative diagram
\[ \begin{tikzcd} H^0 (\mathbb P^1, \mathcal O(d))^{\vee} \arrow[ r, "loc_d" ] & \Bun_2 \\ Y_{n,m} \arrow[u, "\mu_1"] \arrow[r, "\mu_2"] & \mathbb P^\vee \arrow[u, swap,"loc_{n, (e_x), m}"] \end{tikzcd} \] these maps are equal, so the supports are equal, and thus the cycles are equal.

 Because  $\mu_2$ is smooth, the $i$th polar multiplicity $B_{(e_x), (w_k)}^{\vee}$ at $\alpha$ is the $i -n + \dim Y_{n,m}$th polar multiplicity of  $\mu_2^! B_{(e_x),(w_k)}^\vee $ at $y$. By the identity of cycles we just proved, this is also equal to the $i- n + \dim Y_{n,m}$th polar multiplicity of of $\mu_1^! B_{d,r}'$ at $y$. Now by the smoothness of $\mu_1$, this is the $i- n +  \dim Y_{n,m} - \dim Y_{n,m}+(n-2m+1- \sum_x e_x)$th polar multiplicity of $B_{d,r}'$ at zero, which, because \[ i- n +  \dim Y_{n,m} - \dim Y_{n,m}+(n-2m+1- \sum_x e_x )=i+1-2m - \sum_x e_x, \] gives the desired formula.\end{proof}

\begin{lemma}\label{local-polar-multiplicity} The $i$th polar multiplicity of $B_{d,r}'$ at $0$ is \[ 2^{2r-i}\binom{ d-i}{ d-2r} \binom{d+1-r }{ d+1-i}.\] \end{lemma}

Here we take binomial coefficients to vanish if evaluated outside the range where they are normally defined.

\begin{proof} We can express $B_{d,r}'$ as the product of two vector bundles on $\mathbb P^r$, where $\mathbb P^r$ paramaterizes $f_2 \in H^0(\mathcal O(r))$ (up to scaling). The first vector bundle $V_1$ has rank $r$ and consists of linear forms on $H^0(\mathcal O(d))$ that vanish on multiples of $f_2$, while the second $V_2$ has rank $d+1-2r$ and consists of multiples of $f_2^2$ by a polynomial of degree $d$.

By definition, the polar multiplicity is the local intersection number of $V_1 \times_{\mathbb P^r} \mathbb P(V_2)$ with a general codimension $i$ subspace of $H^0(\mathcal O(d))^\vee$ and a general codimension $d-i$ subspace of $H^0( \mathcal O(d))$.  It is equivalent to intersect $\mathbb P(V_1) \times_{\mathbb P^r} \mathbb P(V_2)$ with a general codimension $i-1$ subspace of $H^0(\mathcal O(d))^\vee$ and a general codimension $d-i$ subspace of $H^0( \mathcal O(d))$. In other words, this is the degree on $\mathbb P(V_1) \times_{\mathbb P^r} \mathbb P(V_2)$ of the $i-1$st power of the hyperplane class of $\mathbb P(V_1)$ times the $d-i$th power of the hyperplane class of $\mathbb P(V_2)$. 

For a vector bundle $V$ of rank $w$  on $\mathbb P^r$ with total Chern class $c(V)  = 1 + c_1(V) + \dots +  c_w(V)$, the Segre class in $A^* (\mathbb P^r)$ is equal to $ c(V)^{-1}$ \cite[Proposition 4.1(a)]{Fulton}. Furthermore, the Segre class is equal to the sum over $j$ of the pushforward of the $j$th power of $c_1 (\mathcal O(1))$ from the projectivization $\mathbb P(V)$ to $\mathbb P^r$ \cite[Example 4.1.2]{Fulton}. In particular, the pushforward of the $j$th power of the hyperplane class is the codimension $j+1-w$ part of the Segre class, and therefore is the codimension $j+1-w$ part of $ c(V)^{-1}$.

Observe that $V_2$ is the sum of $d+1-2r$ copies of the line bundle of scalar multiples of $f_2$, which is $\mathcal O(-2)$, so $c(v_2) = (1 - 2H)^{d+1-2r}$, with $H$ the hyperplane class of $\mathbb P^r$, so the pushforward of the $d-i$th power of the hyperplane class is the degree \[ (d-i) +1 - (d+1-2r) = (2r-i)\] part of  $1/(1- 2H)^{d+1-2r}$ and thus is  $ (2H)^{2r-i}  \binom{ d-i}{ d-2r}$.

On the other hand, $V_1$ is dual to the complement of the sum of $d+1-r$ copies of the line bundle of scalar multiples of $f_2$, which is $\mathcal O(-1)$, so \[ c(V_1) = c (  \mathcal O(1)^{ d+1-r} )^{-1} =  (1+H)^{- (d+1-r)},\] so the pushforward of the $i-1$st power of the hyperplane class from $\mathbb P (V_1) $ is the degree \[ i-1+ 1 - r = i-r \] part of   $(1+H)^{d+1-r}$ , which is ${d+1-r \choose d+1-i} H^{i-r}$

Hence their product is $ 2^{2r-i}{ d-i\choose d-2r} {d+1-r \choose d+1-i} H^{r}$, whose degree is  $2^{2r-i}{ d-i\choose d-2r} {d+1-r \choose d+1-i}$.

This handles the case $i \neq d+1$. If $i=d+1$, the polar multiplicity is defined as the multiplicity of the zero section, which vanishes because $B_{d,r}'$ is not the zero section, and the stated formula for the polar multiplicity vanishes also, so the stated formula remains valid in this case.
\end{proof} 

\begin{lemma}\label{local-polar-generating-formula} We have \[\sum_{d=0}^{\infty} \sum_{i=0}^{d+1} \sum_{r=0}^{\lfloor d/2 \rfloor}  2^{d-i}{ d-i\choose d-2r} {d+1-r \choose d+1-i} u^d v^r w^i = \frac{1}{(1-u^2 v w^2) (1 - du - 2u^2 vw - u^2v w^2 )}. \] \end{lemma}

\begin{proof} The summand vanishes unless $r \leq i \leq 2r \leq d$. We can reparameterize so that $a=i-r, b=2r-i, c= d-2r$ so $i = 2a+b$, $r=a+b$, $d=c+2a+2b$. Then the sum is

\[\sum_{a,b,c=0}^{\infty}  2^{b+c}{ b+c \choose c} {a+b+c+1 \choose b+c+1 } u^{c+2a+2b} v^{a+b} w^{2a+b} .\] 

We have \[ \sum_{a=0}^{\infty}{a+b+c+1 \choose b+c+1 } (u^2 v w^2)^a = 1/ (1- u^2 v w^2)^{b+c+2}\]

so this sum is (taking $n=b+c$) 
\[ \sum_{b,c=0}^{\infty} {b+c \choose c} 2^{b+c}  u^{c+2b} v^{b} w^{b} / (1- u^2 v w^2)^{b+c+2} =  \sum_{ n=0}^{\infty}  \frac{  2^n ( u+  u^2 vw)^n }{(1 - u^2 vw^2 )^{n+2}}\]

\[ = \frac{1}{1- (2u + 2u^2 vw) /(1- u^2 v w^2)} \frac{1}{ (1- u^2 v w^2)^2} = \frac{1}{(1-u^2 v w^2) (1 -2 u - 2u^2 vw - u^2v w^2 )}. \]\end{proof}

\begin{lemma}\label{level-polar-multiplicity} The $i$th polar multiplicity of  \[\sum_{ \substack{ w_1, w_2 \in \mathbb N  \\ \sum_{x } e_x + w_1+2w_2 = n }} 2^{w_1}   [ B_{(e_x),(w_k)}^\vee ] \] at $\alpha$ is equal to the coefficient of \[u^{d_{\alpha, (e_x)} } w^{i +1+d_{\alpha, (e_x)} -n  } \]  in \[ \frac{1 }{(1-u^2  w^2) (1 -2 u - 2u^2 w - u^2 w^2 )}.\] 

\end{lemma}

\begin{proof} By Lemma \ref{local-model-level}, this is the same as the $i+1-2m-\sum_x e_x$th polar multiplicity of \[\sum_{ \substack{ w_1, w_2 \in \mathbb N  \\ \sum_{x } e_x + w_1+2w_2 = n }} 2^{w_1}   [ B_{ d_{\alpha, (e_x)} , w_2-m}' ] = \sum_{w_2 =m_{\alpha, (e_x)}  }^{ \lfloor \frac{n - \sum_x e_x}{2} \rfloor } 2^{n - 2w_2 - \sum_x  e_x }   [ B_{ d_{\alpha, (e_x)} , w_2-m_{\alpha,e_x}}' ] = \sum_{r=0 }^{ \lfloor \frac{d_{\alpha, (e_x)} }{2} \rfloor } 2^{d_{\alpha, (e_x)} -2r }   [ B_{ d_{\alpha, (e_x)} , r}' ] .\]

By Lemma \ref{local-polar-multiplicity}, the $j$th polar multiplicity of this cycle at $0$ is the same as 
\[   \sum_{r =0 }^{\lfloor \frac{d_{\alpha, (e_x)} }{2} \rfloor}  2^{d_{\alpha, (e_x)} -2r }  2^{2r-j}{ d_{\alpha, (e_x)} -j\choose d_{\alpha, (e_x)} -2r} {d_{\alpha, (e_x)} +1-r \choose d_{\alpha, (e_x)} +1-j}.\]

Taking the sum over $r$ in Lemma \ref{local-polar-generating-formula} and plugging in $1$ for $v$, we see that the $j$th polar multiplicity is the coefficient of $u^{d_{\alpha, (e_x)} } w^j$ in \[ \frac{1   }{(1-u^2  w^2) (1 -2 u - 2u^2 w - u^2 w^2 )}.\]   Plugging in \[j= i-2m-\sum_x e_x+1 = i+d_{\alpha, (e_x)} +1-n,\] we get the stated formula. \end{proof}

\begin{defi} For $d$ an integer, we define $\mathcal B(d)$ be the coefficient of $u^d$ in the formal power series \[  \frac{1}{ (1-u) (1+u)^2 (1 - (2\sqrt{q}+1)u ) }  = 1 + (2\sqrt{q} ) u + (4 q +2 \sqrt{q} + 2) u^2 +  \dots \] 

By convention, if $d<0$ then $\mathcal B(d)=0$. 
\end{defi}

\begin{lemma}\label{final-Radon-bound} The trace of $\Frob_q$ on the stalk of $K_n$ at a point $\alpha \in \mathbb P^\vee$ is at most

\[   \left( 2 q^{\frac{n-1}{2} }   \sum_{k=0}^{n-1}  { \deg N -4 \choose k} \right)  +  \sum_{\substack{ (e_x) :  \Sing \to \mathbb N \\  e_x \leq c_x \\ \sum_x e_x \leq n }} \left (\prod_{x \in \Sing } {c_x \choose e_x} \right)  q^{\frac{n}{2}}\mathcal B( d_{\alpha,(e_x)} ) . \]

\end{lemma}

\begin{proof} By definition, the trace of $\Frob_q$ on $\mathcal H^* ( K_n)_\alpha$ is $\sum_i (-1)^i \tr (\Frob_q, \mathcal H^{-i} (K_n)_\alpha)$. Because $K_n$ is pure of weight $2n-1$ by Lemma \ref{Radon-perverse-pure}, the eigenvalues of Frobenius on $\mathcal H^{-i}$ have size at most $q^{ \frac{2n-1-i}{2}}$, so this sum is at most 
\[ \sum_i q^{  \frac{2n-1-i}{2}}  \dim \mathcal H^{-i} (K_n)_\alpha. \]

Now by \cite[Theorem 1.4]{mypaper}, $\dim \mathcal H^{-i} (K_n)_\alpha$ is at most the $i$th polar multiplicity of $CC(K_n)$ at $\alpha$. By Lemma \ref{characteristic-cycle-Whittaker}, the characteristic cycle of $K_n$ is \[ \left( 2\sum_{k=0}^{n-1}  { \deg N -4 \choose k} \right) [\mathbb P^\vee ] + \sum_{\substack{ e_x :  \Sing \to \mathbb N \\  e_x \leq c_x \\ w_k : \{1,2 \} \to \mathbb N \\ \sum_{x \in |N|} e_x + w_1 + 2w_2 =n } }2^{w_1} \left (\prod_{x \in \Sing } {c_x \choose e_x} \right)  [ B_{(e_x),(w_k)}^\vee ] .\]

For $i=n$, the polar multiplicity is simply the multiplicity of the zero-section, which gives the first term. For all other $i$, we observe that the polar multiplicity is the sum over $e_x$ of $\left (\prod_{x \in \Sing } {c_x \choose e_x} \right) $ times the $i$th polar multiplicity of \[\sum_{ \substack{ w_1, w_2 \in \mathbb N  \\ \sum_{x } e_x + w_1+2w_2 = n }} 2^{w_1}   [ B_{(e_x),(w_k)}^\vee ]. \] By Lemma \ref{level-polar-multiplicity}, the term corresponding to $(e_x)$ in the $i$th polar multiplicity is the coefficient of \[u^{d_{\alpha, (e_x)} } w^{i +d_{\alpha, (e_x)} +1 -n  } \]  in \[ \frac{1   }{(1-u^2  w^2) (1 -2 u - 2u^2 w - u^2 w^2 )}. \]

Hence the term corresponding to $(e_x),i$ in our bound for the trace is this coefficient times $q^{ \frac{2n-1-i}{2}} = q^{ - \frac{i + d_{\alpha, (e_x)} +1-n}{2} } q^{ \frac{ d_{\alpha, (e_x)} }{2}} q^{n/2}$. Thus the term corresponding to $(e_x),i$ is $q^{n/2 }$ times the coefficient of  \[u^{d_{\alpha, (e_x)} } w^{i +d_{\alpha, (e_x)} +1-n  } \] in \[ \frac{1   }{(1-u^2  w^2) (1 -2 \sqrt{q} u - 2\sqrt{q} u^2 w - u^2 w^2 )}, \] having multiplied $u$ by $\sqrt{q}$ and divided $w$ by $\sqrt{q}$. Summing over all values of $i$ is equivalent to summing over all powers of $w$, which is the same as substituting $w=1$. So the term corresponding to $(e_x)$ in our bound for the trace is $q^{n/2}$ times the coefficient of $u^{d_{\alpha,j}}$ in  \[ \frac{1}{(1-u^2) (1 - 2 \sqrt{q} u - (2\sqrt{q}+1) u^2)}  = \frac{1}{ (1-u) (1+u)^2 (1 - (2\sqrt{q}+1)u ) }.\]  By definition, this is exactly $q^{n/2}$ times $\mathcal B(d_{\alpha,(e_x)})$. Summing over all values of $e_x$, we get the stated formula.

\end{proof}

For $a,z,b \in \mathbb A_F$, let $n$ be $\deg (a/b)  - 2$, so that $\operatorname{div}( \omega_0a/b)$ has degree $n$ and thus there exists an isomorphism $\mathcal O( \operatorname{div}( \omega_0a/b)) \cong \mathcal O(n)$. Fix one such isomorphism.

By definition, $\mathcal O ( \div(\omega_0 a/b))$ is the sheaf whose nonzero sections on an open set $U$ consist of those $w \in F^\times$ where  $  \operatorname{div} (w) + \operatorname{div} (\omega_0a/b) \geq 0$, restricted to $U$, is effective.  Hence the nonzero global sections of $\mathcal O ( \div(\omega_0 a/b))  \cong \mathcal O(n)$ are in natural bijection with the $w \in F^\times $ with $\operatorname{div} (w) + \operatorname{div} (\omega_0a/b) \geq 0$, or, equivalently, $ \operatorname{div} (w\omega_0a/b)  \geq 0$. 

Let $\alpha_{a,b,z}$ be the linear form on $H^0( \mathcal O(n))$ whose value on the section corresponding to $w$ is $\langle z, \omega_0 w \rangle$.

\begin{lemma}\label{first-newform-bound} Let $f$ be a cuspidal newform of level $N$ whose central character has finite order.

For $a,z,b \in \mathbb A_F$, let $n$ be $\deg (a/b)  - 2$, and let $\alpha_{a,b,z}$ be as above.

\[ \left| f \left( \begin{pmatrix} a & az \\ 0 & b \end{pmatrix} \right) \right| \leq  |C_f|  \Biggl(q^{1/2} 2 ^{ \deg N-3}  + q \sum_{\substack{ (e_x ):  \Sing \to \mathbb N \\  e_x \leq c_x \\ \sum_x e_x \leq n }} \left (\prod_{x \in \Sing } {c_x \choose e_x} \right)   \mathcal B( d_{\alpha_{a,b,z},(e_x)}) \Biggr) .\]

\end{lemma}

\begin{proof} By Lemma \ref{Drinfeld-formula}
\[ f \left( \begin{pmatrix} a & bz \\ 0 & b \end{pmatrix} \right) =C_f   q^{ - \deg (\omega_0a/b)/2} \eta(b)^{-1} \sum_{\substack{ w \in F^\times \\ \operatorname{div} (w\omega_0a/b) \geq 0}}  \psi(wz ) r_{\mathcal F} (\operatorname{div} (wa\omega_0/b) ) . \]

The bijection between $w \in F^\times$ such that $ \operatorname{div} (w\omega_0a/b) \geq 0$ and nonzero sections $s \in H^0( \mathcal O(n) )$ equates $\div ( w \omega_0 a/b) $ and $\div(s)$, because these two divisors have the same order of vanishing at each point. Thus, using this bijection, we have  \[ \sum_{\substack{ w \in F^\times \\ \operatorname{div} (w\omega_0a/b) \geq 0}}  \psi(wz ) r_{\mathcal F} (\operatorname{div} (wa\omega_0/b) )  = \sum_ { \substack { s \in H^0 (\mathbb P^1, \mathcal O(n)) \\ s\neq 0 }} \psi_0 (  \alpha_{a,b,z} (s) )  r_{ \mathcal F} ( \div (s)) \]\[
= \sum_{ D \textrm{ effective, degree }n} r_{\mathcal F}(D)  \sum_{\substack{ s \in H^0 (\mathbb P^1, \mathcal O(n) ) \\ \operatorname{div} (s) =D }} \psi_0 (  \alpha_{a,b,z} (s) ) .\]  

The set of $s \in H^0 (\mathbb P^1 , \mathcal O(n) )$  with $\div(s)=D$ is a line through the origin in the vector space $H^0 (\mathbb P^1 , \mathcal O(n) )$, minus the origin. In particular, it has $q-1$ elements. 

If $ \alpha_{a,b,z} (s) $ is identically zero on this line, then \[ \sum_{\substack{ s \in H^0 (\mathbb P^1, \mathcal O(n) ) \\ \operatorname{div} (s) =D }} \psi_0 (  \alpha_{a,b,z} (s) )= q-1 .\] If $\alpha_{a,b,z}$ is not identically zero on this line, then it is a nonconstant linear form, and thus $\psi_0 (  \alpha_{a,b,z} (s) )$ is a nontrivial character. Because the sum of a nontrivial character over an abelian group vanishes, the sum over the line vanishes in this case, and so the sum over the line minus the origin is $-1$, i.e.
\[ \sum_{\substack{ s \in H^0 (\mathbb P^1, \mathcal O(n))  \\ \operatorname{div} (s) =D }} \psi_0 (  \alpha_{a,b,z} (s) ) = -1 .\] 

By definition, $ \alpha_{a,b,z} (s) $ is identically zero on this line if and only if $D \in P(\alpha_{a,b,z})$. Thus
\[ \sum_{ D \textrm{ effective, degree }n}  r_{\mathcal F}(D)  \sum_{\substack{ s \in H^0 (\mathbb P^1, \mathcal O(n))  \\ \operatorname{div} (s) =D }} \psi_0 (  \alpha_{a,b,z} (s) )  \]
is 
 \[q \sum_{\ D  \in P(\alpha_{a,b,z} ) }   r_{\mathcal F}(D) - \sum_{D \textrm{ effective, degree }n } r_{\mathcal F}(D)\]
\[ = -q  \tr(\Frob_q, (K_n)_{\alpha_{a,b,z} })- \sum_{ D \textrm{ effective, degree }n} r_{\mathcal F} (D) \] by Lemma \ref{Radon-trace-function}.

We have \[ \sum_{ D \textrm{ effective, degree }n} r_{\mathcal F} (D)  \leq { \deg N -4 \choose n} q^{n/2}\] by the Riemann hypothesis for the $L$-function of $\mathcal F$. By Lemma \ref{final-Radon-bound}, we have

\[ q\tr(\Frob_q, (K_n)_{\alpha_{a,b,z}} )) \leq  \left( 2 q^{\frac{n+1}{2} }   \sum_{k=0}^{n-1}  { \deg N -4 \choose k} \right)  +  \sum_{\substack{( e_x) :  \Sing \to \mathbb N \\  e_x \leq c_x \\ \sum_x e_x \leq n }} \left (\prod_{x \in \Sing } {c_x \choose e_x} \right)  q^{\frac{n+1}{2}} \mathcal B( d_{\alpha_{a,b,z},(e_x)} ) . \]

Hence 

\[  \sum_{\substack{ w \in F^\times \\ \operatorname{div} (w a \omega_0/b) \geq 0}}  \psi(wz ) r_{\mathcal F} (w a \omega_0/b)  \] \[ \leq \left( 2 q^{\frac{n+1}{2} } \sum_{k=0}^{n}  { \deg N -4 \choose k} \right)  +  \sum_{\substack{ (e_x ):  \Sing \to \mathbb N \\  e_x \leq c_x \\ \sum_x e_x \leq n }} \left (\prod_{x \in \Sing } {c_x \choose e_x} \right)  q^{\frac{n+2}{2}} \mathcal B( d_{\alpha_{a,b,z},(e_x)} ) . \]

Now because $n = \deg \div (a \omega_0/b)$ and $\eta$ has finite order, we have $  q^{ - \deg (\omega_0a/b)/2}=q^{-n/2}$ and $| \eta(b)^{-1} |=1$ so 

\[ \left|  f \left( \begin{pmatrix} a & bz \\ 0 & b \end{pmatrix} \right) \right| \leq|C_f|\Biggl(  2q^{1/2} \sum_{k=0}^{n}  { \deg N -4 \choose k}  +  q\sum_{\substack{ (e_x ):  \Sing \to \mathbb N \\  e_x \leq c_x \\ \sum_x e_x \leq n }} \left (\prod_{x \in \Sing } {c_x \choose e_x} \right) \mathcal B( d_{\alpha_{a,b,z},(e_x)} )\Biggr)\]
 
\[ \leq |C_f|  \Biggl(q^{1/2} 2 ^{ \deg N-3}  + q \sum_{\substack{ (e_x ):  \Sing \to \mathbb N \\  e_x \leq c_x \\ \sum_x e_x \leq n }} \left (\prod_{x \in \Sing } {c_x \choose e_x} \right) \mathcal   B( d_{\alpha_{a,b,z},(e_x)} )\Biggr).\]

On the other hand, if $\alpha_{a,b,z}$ is zero, by the Riemann hypothesis for $\mathcal F$, \[  \sum_{\substack{ w \in F^\times \\ \operatorname{div} (wa \omega_0/b) \geq 0}}  \psi(wz ) r_{\mathcal F} (\operatorname{div} (w a \omega_0 /b)) =  (q-1) \sum_{ D \textrm{ effective, degree }n} r_{\mathcal F} (D)  \leq { \deg N -4 \choose n} (q-1) q^{n/2},\] and multiplying by $|C_f|q^{-n/2}$, we obtain a bound of $|C_f|{  \deg N -4 \choose n} (q-1)$. On the other hand, the contribution to our stated upper bound from the terms where $\sum_x e_x = n $ is $|C_f|q  {\deg N \choose n}\mathcal  B(0) = |C_f| q {\deg N \choose n}$, which is at least as large, and so the stated upper bound is sufficient.

\end{proof}

We can now also prove Proposition \ref{automorphic-cohomology-vanishing} from the introduction. This is not used in the remainder of the paper, but might provide a point of comparison.

\begin{proposition}[Proposition \ref{automorphic-cohomology-vanishing}]\label{a-c-v}Let $\mathcal F$ be a perverse sheaf on $\Bun_G$ whose characteristic cycle is contained in the nilpotent cone of the moduli space of Higgs bundles. Then for a $G$-bundle $\alpha$ on $C$, $\mathcal H^i(\mathcal F)_{\alpha} $ vanishes for \[i > \dim \{ v \in H^0 ( C, \operatorname{ad} (\alpha) \otimes K_C) | v \textrm { nilpotent} \} - (g-1) \dim G. \] \end{proposition}

\begin{proof} This follows immediately from \cite[Corollary 1.5]{mypaper}, which says that the stalk cohomology of a perverse sheaf $\mathcal F$ on $X$  at a point $x$ vanishes in degree greater than the dimension of the fiber of the characteristic cycle of $\mathcal F$ over $x$ minus the dimension of $X$. We take $f: X \to \Bun_G$ a smooth map of relative dimension $r$ from a smooth scheme $X$ and $x \in X$ with $f(x) = \alpha$, and apply \cite[Corollary 1.5]{mypaper} to $(f^* \mathcal F[r], X, x)$. Because $\dim X= \dim  \Bun_G+r  = (g-1) \dim G+r$ and because the fiber over $\alpha$ in the nilpotent cone of the space of Higgs bundles is $ \{ v \in H^0 ( C, \operatorname{ad} (\alpha) \otimes K_C) | v \textrm { nilpotent} \} $, the stalk cohomology of $f^*\mathcal F[r]$ vanishes in degrees $> \dim \{ v \in H^0 ( C, \operatorname{ad} (\alpha) \otimes K_C) | v \textrm { nilpotent} \} - (g-1) \dim G$, so the stalk cohomology of $\mathcal F$ vanishes in degrees $> \dim \{ v \in H^0 ( C, \operatorname{ad} (\alpha) \otimes K_C) | v \textrm { nilpotent} \} - (g-1) \dim G$.  \end{proof}

\section{Heights of virtual cusps}\label{heights-virtual-cusps}

For $  \begin{pmatrix}  a & b\\ c & d\end{pmatrix} \in GL_2(\mathbb A_F)$, $(f_1:f_2) \in \mathbb P^1 ( \overline{ \mathbb F}_q(C))$, and $e_x$ a function from $\Sing$ to $\mathbb N$ with $0\leq e_x \leq c_x$, let

\[ h \left(  \begin{pmatrix}  a & b\\ c & d \end{pmatrix},  (f_1:f_2), (e_x) \right)\] \[=2\deg ( \min( \operatorname{div} (af_1+cf_2), \operatorname{div}(bf_1+ df_2)+ \sum_x e_x [x] )) - \deg ( \operatorname{div}(ad-bc) + \sum_x e_x[x]) .\]

Here $af_1+cf_2$ and $bf_1+df_2$ are elements of the adeles of $\overline{\mathbb F}_q (C)$, so their divisors are divisors on $C$ defined over $\overline{\mathbb F}_q$, and the min of two divisors involves taking the min of their multiplicities at a given point.

Let us describe the analogue of this in the classical theory of modular forms. Associated to each point of $\mathbb P^1(\mathbb Q)$ is a cusp on the upper half plane. Associated to the cusp at $\infty$ is the function $\log y$ on the upper half plane which measures the height (in some sense, measures how close a point is to that cusp). We can associate to any other cusp a corresponding function, by choosing  any element of $SL_2(\mathbb Z)$ that sends that cusp to $\infty$ and then composing $\log y$ with that element of $SL_2(\mathbb Z)$. One can define this function adelically, and the definition will look similar to our definition of $h$. However, our definition has the additional complexity that $(f_1:f_2)$ need not be defined over $\mathbb F_q(C)$, but in fact over $\overline{\mathbb F}_q(C)$. Thus we are studying ``virtual" cusps that appear as actual cusps only over an extension of the constant field.

The height function has an alternate definition in the geometry of vector bundles on $\mathbb P^1$, which we explain next. After proving the equivalence of the two definitions, we will use whichever is more convenient for proving a particular lemma. However, everything can be proved using only the adelic definition or only the geometric one, and readers might want to try to work out the analogue in the other setting of one of the arguments we use. 

Let $V_{   (e_x)} \left( \scriptscriptstyle{  \begin{pmatrix}  a & b\\ c & d \end{pmatrix} }\right) $ be the vector bundle on $\mathbb P^1_{\overline{\mathbb F}_q} $ whose sections over an open set $U$ consist of all those $(f_1, f_2) \in \overline{\mathbb F}_q(C)^2$ such that the restriction of the divisor $\min ( \div ( af_1+cf_2) , \div(bf_1 +df_2) + \sum_x e_x [x] ) $ to $U$ is effective.

Let $L_{ (f_1:f_2), (e_x) } \left( \scriptscriptstyle{  \begin{pmatrix}  a & b\\ c & d \end{pmatrix} }\right) \subset V_{   (e_x)} \left( \scriptscriptstyle{  \begin{pmatrix}  a & b\\ c & d \end{pmatrix} }\right) $ be the sub-line-bundle of $ V_{   (e_x)} \left( \scriptscriptstyle{  \begin{pmatrix}  a & b\\ c & d \end{pmatrix} }\right) $ generated by the meromorphic section $(f_1,f_2)$.


\begin{lemma}\label{height-geometric} We have \[ h \left(  \begin{pmatrix}  a & b\\ c & d \end{pmatrix},  (f_1:f_2), (e_x) \right) = 2 \deg L_{ (f_1:f_2), (e_x) } \left( \scriptscriptstyle{  \begin{pmatrix}  a & b\\ c & d \end{pmatrix} }\right)  - \deg \det  V_{   (e_x)} \left( \scriptscriptstyle{  \begin{pmatrix}  a & b\\ c & d \end{pmatrix} }\right) .\]\end{lemma} 

\begin{proof} We first can see that  \[ \deg \min ( \div ( af_1+cf_2) , \div(bf_1 +df_2) + \sum_x e_x [x] )  = \deg L_{ (f_1:f_2), (e_x) } \left( \scriptscriptstyle{  \begin{pmatrix}  a & b\\ c & d \end{pmatrix} }\right)    \]  because this divisor, by construction, is the divisor of the line bundle $L_{ (f_1:f_2), (e_x) } \left( \scriptscriptstyle{  \begin{pmatrix}  a & b\\ c & d \end{pmatrix} }\right)  $.

Furthermore, we can see that \[  \deg ( \div (ad-bc) + \sum_x e_x [x] )= \deg \det  V_{   (e_x)} \left( \scriptscriptstyle{  \begin{pmatrix}  a & b\\ c & d \end{pmatrix} }\right)  .\] In the case where $e_x=0$ for all $x$, this is because the matrix $\begin{pmatrix} a &b\\ c & d \end{pmatrix}$ gives gluing data for $V$, so its determinant gives gluing data for $\det V$. In the case where $(e_x) \neq 0$, this simply corresponds to a modification of this vector bundle where we extend it by a length $e_x$ module over $x$, and so the divisor class of its determinant increases by $\sum_x e_x[x]$. 

The equality now follows from the definition of  $h$.\end{proof}

\vspace{10pt}

It turns out (in Lemma \ref{height-volume-comparison}) that this $h$ function, or more precisely its maximum over $f_1,f_2$, determines the $d_{\alpha,(e_x)}$ defined before Lemma \ref{local-model-level}. In addition, $h$ has suitable invariance properties under the action of $GL_2(F)$ and $\Gamma_1(N)$, proved in Lemmas \ref{height-F-invariance} and \ref{height-compact-invariance}. Because of this, expressing the bound of Lemma \ref{first-newform-bound} in terms of $h$ will provide a more invariant formulation that  is easier to optimize.  In this optimization process it will also be crucial to understand how $h$ depends on $(f_1:f_2)$ and $(e_x)$. This is accomplished, in different ways, in Lemmas \ref{height-mountain-shape} and \ref{height-orbit-size}. Finally Lemma \ref{height-triangularizable-characterization} is the tool we will use to put our elements of $GL_2(\mathbb A_F)$ in upper-triangular form. Combining all these elements, we will deduce Lemma \ref{cusp-newform-bound}.

\vspace{5pt}

%
%
%
%
%
%

\begin{lemma}\label{height-F-invariance}  $h \left(  \begin{pmatrix} a & b\\ c & d \end{pmatrix},  (f_1:f_2), (e_x) \right)$ is invariant under the right action of $\Gamma_1(N)$ on $ \begin{pmatrix} a & b\\ c &d \end{pmatrix}$.\end{lemma}

\begin{proof} In view of Lemma \ref{height-geometric} it suffices to show that the vector bundle $V_{   (e_x)} \left( \scriptscriptstyle{  \begin{pmatrix}  a & b\\ c & d \end{pmatrix} }\right) $  is invariant under the right action of $\Gamma_1(N)$ on $ \begin{pmatrix} a & b\\ c &d \end{pmatrix}$. In other words, we must show, for $\begin{pmatrix}  a' & b'\\ c' & d' \end{pmatrix} \in \Gamma_1(N)$, that if  $\div ( af_1+cf_2) $ and $\div(bf_1 +df_2) + \sum_x e_x [x]$ are effective over an open set $U$, then  $\operatorname{div}( a'(af_1+cf_2)+ c' (bf_1+ df_2) )$ and $\operatorname{div}( b'(af_1+cf_2)+ d' (bf_1+ df_2) )+ \sum_x e_x [x] $ are as well.  (We only need show the ``if" direction because the ``only if" direction follows upon taking the inverse matrix).

At each point $x \in U$, because $ (af_1+cf_2)$ is integral, and $a'$ is integral, $a'(af_1+cf_2)$  is integral, and because the order of pole of $(bf_1+df_2)$ is at most $e_x$, and the order of zero of $c'$ is at least $c_x \geq e_x$, $ c' (bf_1+ df_2)$ is integral, so their sum is integral as well.

Similarly, because $b'$ is integral, and $(af_1+cf_2)$ is integral, $b'(af_1+cf_2)$ is integral, and because $d'$ is integral, and the order of pole of $(bf_1+ df_2)$ is at most $e_x$, the order of pole of $d' (bf_1+ df_2) $ is at most $e_x$, so the order of pole of $ b'(af_1+cf_2)+ d' (bf_1+ df_2) )$ is at most $e_x$. \end{proof}

\begin{lemma}\label{height-compact-invariance} $ h \left(  \begin{pmatrix}  a & b\\ c & d\end{pmatrix},  (f_1:f_2), (e_x) \right) $ is invariant under the action of $GL_2(F)$ by left multiplication on  $\begin{pmatrix} a & b\\ c& d \end{pmatrix}$ and, simultaneously, multiplication of the inverse transpose on $\begin{pmatrix} f_1 \\ f_2 \end{pmatrix}$.\end{lemma}

\begin{proof} Note that \[ \begin{pmatrix} f_1 \\ f_2 \end{pmatrix}^T \begin{pmatrix}   a &  b\\ c& d \end{pmatrix} = \begin{pmatrix} af_1 + cf_2  \\ bf_1+df_2 \end{pmatrix}^T \] so multiplying $\begin{pmatrix}   a &  b\\ c& d \end{pmatrix}$ on the left by a matrix and $\begin{pmatrix} f_1 \\ f_2 \end{pmatrix}^T$ on the right by the inverse matrix does not affect those terms, and it also does not affect the degree of the determinant $ad-bc$ as the determinant of the matrix being multiplied by lies in $\overline{\mathbb F}_q(C)$ and thus has zero degree. \end{proof}

\begin{lemma}\label{height-unique-cusp} For each $ \begin{pmatrix}  a & b\\ c & d \end{pmatrix},( e_x)$, if $(f_1:f_2) \neq (f_3:f_4)$ as points of the projective line, then

\[ h \left(  \begin{pmatrix}  a & b\\ c &d \end{pmatrix},  (f_1:f_2), (e_x) \right) + h \left(  \begin{pmatrix} a & b\\ c &d \end{pmatrix},  (f_3:f_4), (e_x) \right)\leq 0 \]
 
In particular, there is at most one $(f_1:f_2)$ such that $h \left(  \begin{pmatrix} a & b\\ c &d \end{pmatrix},  (f_1:f_2), (e_x) \right)$ is positive.\end{lemma}

\begin{proof} In view of Lemma \ref{height-geometric}, it suffices to show that for $L_{12}, L_{34}$ distinct line sub-bundles of a vector bundle $V$, we have \[ (2\deg L_{12} - \deg \det V) + (2 \deg L_{34} - \deg \det V) \leq 0.\] This follows from the fact that the natural map $L_{12} + L_{34} \to V$ is injective, so $V$ is the extension of $L_{12} +L_{34}$ by a finite length module and thus  $\deg \det V \geq \deg L_{12} + \deg L_{34}$. 

\end{proof}


\begin{lemma}\label{height-mountain-shape} For each $\begin{pmatrix} a & b\\ c & d \end{pmatrix}, (f_1:f_2)$, there is a unique $(e_x(f_1:f_2))$ and $h^*(f_1:f_2)$ such that \[ h \left(  \begin{pmatrix}  a & b\\ c & d \end{pmatrix},  (f_1:f_2), e_x' \right) = h^*(f_1:f_2) - \sum_{x \in \Sing} | e_x(f_1:f_2) - e_x' | \] for all tuples $(e_x')$ with $0\leq e_x' \leq c_x$ for all $x$. \end{lemma}

\begin{proof}First we check that such a representation is unique. This follows because $h^*(f_1:f_2)$ is identifiable as the maximum value of $h$ with varying $(e_x)$, and $(e_x(f_1:f_2))$ is identifiable as the location of the maximum value.

Then we check existence. Note that $h$ may be written as a sum over the closed points of $C_{\overline{\mathbb F}_q}$ of the contribution of that point to the degrees of the relevant divisors. Only the contribution from the point $x$ depends on $e_x$, so $h$, viewed as a function of the tuple $(e_x)$, is a sum of functions depending on the individual $e_x$, plus a constant. It suffices to check that these individual functions are piecewise linear with slopes $1$ and $-1$ (in order). The multiplicity of $x$ in the divisor $ \min( \operatorname{div} (af_1+cf_2), \operatorname{div}(bf_1+ df_2)+ \sum_x e_x [x] )$ is a piecewise linear function with slopes $1$ and $0$. Doubling it gives slopes $2$ and $0$ and then subtracting the $e_x$ appearing in the \[ \deg ( \operatorname{div}(ad-bc) + \sum_x e_x[x])= \deg (ad-bc) + \sum_x e_x\] term gives slopes $1$ and $-1$. \end{proof}

\begin{lemma}\label{height-orbit-size}Fix a matrix $\begin{pmatrix} a & b\\ c & d \end{pmatrix}$. If $h^*(f_1:f_2)>0$ then the size of the orbit of $\Gal(\mathbb F_q)$ on $(f_1:f_2)$ is equal to the size of the orbit of $\Gal(\mathbb F_q)$ on $e_x(f_1:f_2)$.\end{lemma}

\begin{proof} This is immediate because, by Lemmas \ref{height-unique-cusp} and \ref{height-mountain-shape}, $e_x(f_1:f_2)$ is uniquely determined by $(f_1:f_2)$ and $(f_1:f_2)$ is uniquely determined by $e_x(f_1:f_2)$.\end{proof}

\begin{lemma}\label{height-triangularizable-characterization} Let $f_1, f_2$ be elements of $\mathbb F_q(C)$.

There exists $\bfg \in \Gamma_1(N)$, $\gamma \in GL_2(F)$ such that $\gamma \begin{pmatrix} a & b\\ c & d \end{pmatrix} \bfg$ is upper triangular and $\left(\gamma^{-1}\right)^T \begin{pmatrix} f_1 \\ f_2 \end{pmatrix} = \begin{pmatrix} 0 \\ 1 \end{pmatrix} $  if and only if $e_x(f_1:f_2) =c_x$.  \end{lemma}

\begin{proof}
 We check the case ``only if" first. Using Lemmas \ref{height-F-invariance} and \ref{height-compact-invariance}, we can multiply on the left by $\gamma$ and on the right by $\bfg$, preserving the height function in doing so, and thus preserving $e_x(f_1:f_2)$. Thus we can assume, without loss of generality, that $\bfg $ and $\gamma
$ are both the identity matrix. In other words, we may assume that $f_1=0, f_2=1$ and $c=0$.
 
  In this case \[ h \left(  \begin{pmatrix} a & b\\ c & d \end{pmatrix},  (f_1:f_2), (e_x) \right)= h \left(  \begin{pmatrix} a & b\\ 0 & d \end{pmatrix},  (0:1), (e_x) \right)\] 
\[ =2\deg ( \min( \operatorname{div} (0), \operatorname{div}(d)+ \sum_x e_x [x] )) - \deg ( \operatorname{div}(ad) + \sum_x e_x[x])\] \[ =2 \deg ( \operatorname{div}(d)+ \sum_x e_x [x] ) - \deg ( \operatorname{div}(ad) + \sum_x e_x[x]) = \deg( \div{d/a}) + \sum_x e_x,\] because $\div(0)=\infty$. The sum $\deg( \div{d/a}) + \sum_x e_x$ is maximized when $e_x$ is maximized for all $x$. So indeed $e_x(f_1:f_2) = c_x$.

For the ``if" direction, first observe that by assumption, $f_1,f_2 \in \mathbb F_q(C)$. Hence by the action of a suitable element of $GL_2(F)$, we may assume $f_1=0$. Now by assumption \[2\deg ( \min( \operatorname{div} (cf_2), \operatorname{div}(df_2)+ \sum_x e_x [x] )) - \deg ( \operatorname{div}(ad) + \sum_x e_x[x]) \] is an increasing function of $e_x$ as $e_x \leq c_x$, which implies that for all $x \in \Sing$, the order of vanishing of $c$ at $x$ is at least the order of vanishing of $d$ plus $c_x$. This means we can multiply on the right by $\begin{pmatrix} 1 & 0 \\ - c/d & 1 \end{pmatrix} \in \Gamma_1(N)$ to make $c$ vanish. 

At all points not in $\Sing$, we use the fact that $GL_2( \mathcal O_{F_v})$ acts transitively on $ \mathbb P^1( \mathcal O_{F_v}) = \mathbb P^1(F_v)$ to multiply on the right by something that makes $c$ vanish. 
\end{proof}

\begin{lemma}\label{height-volume-comparison} Let $a,b,z$ be adeles and let $n= \deg a -\deg b-2$.

\begin{enumerate}

\item We have $\sum_x e_x>n$ if and only if \[ h \left(  \begin{pmatrix} a & bz\\ 0 & b \end{pmatrix},  (0:1), (e_x) \right) >-2.\]

\item For all other $(e_x)$, we have 

 \[ d_{\alpha_{a,b,z}, (e_x)} = \max_{(f_1,f_2) \in \overline{\mathbb F}_q (C)}  h \left(  \begin{pmatrix} a & bz\\ 0 & b \end{pmatrix},  (f_1:f_2), (e_x) \right)-2. \] 
 
 \end{enumerate}
 \end{lemma} 
 
 \begin{proof} 
 For part (1), note that
 \[ h \left(  \begin{pmatrix} a & bz\\ 0 & b \end{pmatrix},  (0:1), (e_x) \right)  = 2\deg ( \min( \operatorname{div} (0), \operatorname{div}(b)+ \sum_x e_x [x] )) - \deg ( \operatorname{div}(ab) + \sum_x e_x[x])\] \[ =2 \deg ( \operatorname{div}(b)+ \sum_x e_x [x] ) - \deg ( \operatorname{div}(ab) + \sum_x e_x[x])= \deg(b/a)+ \sum_x e_x\] \[ = \deg(b/a) + \sum_x e_x= \sum_x e_x - n-2  \] so it takes a value $> -2$ if and only if $\sum_x e_x>n$.

 For part (2), recall from Definition \ref{mdalpha-notation} that $d_{\alpha_{a,b,z},(e_x)}$ is $n-2m- \sum_x e_x$ where $m$ is the minimum $m$ such that  $\alpha_{a,b,z}$ vanishes on all polynomial multiples of $f \prod_{x\in \Sing } l_x^{e_x} $ for some $f$ of degree $m$.
 
We can equivalently say that $m$ is the minimum $r$ such that the linear form $g \mapsto \alpha_{a,b,z} ( g \prod_{x\in \Sing } l_x^{e_x} ) $ on $H^0( \mathcal O(n- \sum_x e_x))$ lies in the subset $W_r$ of Definition \ref{Wr-notation}.  By Lemma \ref{pullback-bun-stratification}, $W_r$ is the set where  the extension $0 \to \mathcal O \to V \to \mathcal O(n+2- \sum_x e_x ) \to 0 $ defined, by Serre duality, from this linear form, has a line bundle summand of degree at most $r$. Because the determinant of this vector bundle $V$ has degree $n+2-\sum_x e_x$, we can see that the maximum of $2 \deg L - \deg \det V$ over line sub-bundles $L$ of $V$ is \[ 2 (n+2 - m- \sum_x e_x ) - (n+2-\sum_x e_x) = n+ 2 - 2m - \sum_x e_x = d_{\alpha_{a,b,z},(e_x)}+2.\] So to show that  \[ d_{\alpha_{a,b,z}, (e_x)} = \max_{(f_1,f_2) \in \overline{\mathbb F}_q (C)}  h \left(  \begin{pmatrix} a & bz\\ 0 & b \end{pmatrix},  (f_1:f_2), (e_x) \right)-2  \]  it suffices to show that \[V \cong V_{(e_x)} \left( {\scriptscriptstyle \begin{pmatrix} a & bz \\ 0 & b \end{pmatrix} } \right)\] up to twisting by line bundles.

To do this, we show that they are each constructed from the same $\Ext$ class, using the adelic construction of Serre duality. For a line bundle $L$, we can express $H^1(L)$ as $(\mathbb A_F \otimes L) / (F \otimes L + ( \prod_d \mathcal O_{F_v}) \otimes L) $ by taking a torsor, trivializing it over the generic point and over a formal neighborhood of each point, and viewing the discrepancy between those two trivializations in the punctured formal neighborhood of each point as an element of $\mathbb A_F \otimes L $. Then the Serre duality pairing between $H^1(L)$ and $H^0( K L^{-1})$ is the residue pairing between these adeles and global sections of $KL^{-1}$. In our case, the relevant $H^1$-torsor is the extension class of $V_{   (e_x)} \left( \scriptscriptstyle{  \begin{pmatrix}  a & bz\\ 0 & b \end{pmatrix} }\right) / L_{ (0:1), (e_x) } \left( \scriptscriptstyle{  \begin{pmatrix}  a & bz\\0  & b \end{pmatrix} }\right) $ by $L_{ (0:1), (e_x) } \left( \scriptscriptstyle{  \begin{pmatrix}  a & bz\\ 0 & b \end{pmatrix} }\right) $. We can trivialize this torsor by splitting the extension. Over the generic point, we choose the splitting $F^2 = F + F$. Over a formal neighborhood of each point, we choose the inverse image under $   \begin{pmatrix}  a & bz\\ 0 & b \end{pmatrix}^T $ of the splitting of $\mathcal O_{F_x} + \pi_x^{-e_x} \mathcal O_{F_x}$ into $\mathcal O_{F_x} $ and $ \pi_x^{-e_x} \mathcal O_{F_x}$. Because this is an extension of $\mathcal O ( \div a) $ by $\mathcal O( \div b+ \sum_x e_x [x])$, the relevant line bundle is $\mathcal O (  \div{b/a} + \sum_x e_x[x])$.  The torsor is represented by the adele $z$, because the image of the vector $(1,0)$ generates the splitting over $F$ and $(1, -z) = \begin{pmatrix}  a & bz\\ 0 & b \end{pmatrix}^{-T} (a,0)$ generates the local splitting, so the discrepancy is $z$. When we view this torsor as a linear form on $\mathcal O ( \div{a \omega_0/b} - \sum_x e_x[x]) = \mathcal O(n- \sum_x e_x[x])$ using Serre duality / the residue pairing, we will obtain exactly the definition of $\alpha_{a,b,z}$.

 \end{proof}

 We now express our bound as a sum over $(f_1:f_2) \in \mathbb P^1( \overline{\mathbb F}_q(C)) $ with $h^*( f_1:f_2) \geq 2$, instead of a sum over tuples $e_x$. To do this, we first define a set $S_{(f_1:f_2)}$ of tuples $(e_x')$. We will then combine the terms for all these $(e_x')$ into a single term for $(f_1:f_2)$. To do this, we will use Lemma \ref{f1-f2-finder}, which guarantees there exists a suitable $(f_1:f_2)$ for every relevant tuple $(e_x')$.

 \begin{defi} For $(f_1:f_2) \in \mathbb P^1( \overline{\mathbb F}_q(C))$ define $S_{ (f_1:f_2)}$ as \[\left\{ e_x' : \Sing \to \mathbb N \mid e_x' \leq c_x , \sum_x e_x' \leq n ,  \sum_x |e_x' - e_x(f_1:f_2) | \leq h^* (f_1:f_2) - 2\right\}.\] \end{defi}

 We note that the final condition makes $S_{(f_1:f_2)}$ empty unless $h^*(f_1:f_2) \geq 2$.

 \begin{lemma}\label{f1-f2-finder} Let $ \begin{pmatrix} a & bz \\ 0 & b \end{pmatrix} \in GL_2(\mathbb A_F)$.  Let $n =\deg a-\deg b-2$. 
 
 For every $(e_x'): \Sing \to \mathbb N$ with $e_x' \leq c_x$, $\sum_x e_x' \leq n$, and $ d_{\alpha_{a,b,z },(e_x')} \geq 0$, there is $(f_1,f_2) \in \mathbb P^1(\overline{\mathbb F}_q(C))$ with $ (e_x') \in S_{ (f_1,f_2) },$  $ f_1 \neq 0 $, and \[ h^*(f_1:f_2) - 2- \sum_x |e_x' - e_x (f_1:f_2) | = d_{\alpha_{a,b,z },(e_x')}.\]
   
   \end{lemma}
   
   \begin{proof}  By Lemma \ref{height-volume-comparison}(2) and Lemma \ref{height-mountain-shape}, we have  \[ d_{\alpha_{a,b,z},(e_x')}= \max_{(f_1:f_2) \in \overline{\mathbb F}_q (C)}  h \left(  \begin{pmatrix} a & bz\\ 0 & b \end{pmatrix},  (f_1:f_2), (e_x') \right)-2 \] \[= \max_{(f_1:f_2) \in \overline{\mathbb F}_q (C)} ( h^* ( f_1:f_2) -2-  \sum_{x\in\Sing} |e_x (f_1:f_2) -e_x'| ). \] 

Take $(f_1:f_2)$ maximizing this function. Because we have assumed that $d_{\alpha_{a,b,z},(e_x') }\geq 0$, we have \[  h^* ( f_1:f_2) -2-  \sum_{x\in\Sing} |e_x (f_1:f_2) -e_x'| \geq 0\] so $(e_x') \in S_{(f_1:f_2)}$. 

By Lemma \ref{height-volume-comparison}(1), because $\sum_x e_x' \leq n$ and $h \left(  \begin{pmatrix} a & bz\\ 0 & b \end{pmatrix},  (f_1:f_2), (e_x') \right)\geq 2$, we have $(f_1:f_2) \neq (0:1)$ and thus $f_1 \neq 0$.\end{proof}



\begin{lemma}\label{cusp-newform-bound} Let $f$ be a cuspidal newform of level $N$ whose central character has finite order. We have

\[ \frac{ \left| f \left( \begin{pmatrix} a & b \\ c & d \end{pmatrix} \right) \right|}{ |C_f| }  \leq q^{1/2} 2 ^{ \deg N-3}  +  \] \[   q \sum_{\substack{ (f_1:f_2) \in \mathbb P^1(\overline{\mathbb F}_q(C)) \\ h^* (f_1:f_2) \geq 2 \\ (e_x(f_1:f_2)) \neq (c_x) }} \sum_{\substack{ (e_x ') \in S_{(f_1:f_2)}  }} \left (\prod_{x \in \Sing } {c_x \choose e_x'} \right) \mathcal   B\left( h^*(f_1:f_2) -2 - \sum_x |e_x' -e_x(f_1:f_2)| \right) .\] \end{lemma}

\begin{proof} We will show how this follows from Lemma \ref{first-newform-bound}. First note that the function $f$ and our putative bound for it are left invariant under $GL_2(\mathbb F_q(C))$ and right invariant under $\Gamma_1(N)$.

Let us first check that we may assume that $c=0$ and that, for any $(f_1:f_2) \in \mathbb P^1 ( \overline{\mathbb F}_q(C))$ with $h^*(f_1:f_2) \geq 2$ and $e_x(f_1:f_2)=c_x$, we have $f_1=0$.  First suppose that there is such a $(f_1:f_2)$. By Lemma \ref{height-orbit-size}, because $h^*(f_1:f_2)\geq 2>0$, $(f_1:f_2)$ is fixed by $\Gal(\mathbb F_q)$, so up to scaling we have $f_1,f_2 \in \mathbb F_q(C)$. Then by Lemma \ref{height-triangularizable-characterization}, after multiplying by a suitable element of $GL_2(\mathbb F_q(C))$ on the left and a suitable element of $\Gamma_1(N)$ on the right, we may assume that $c$ and $f_1$ are $0$. Then by Lemma \ref{height-unique-cusp}, there is no other $(f_1:f_2)$ with $h^*(f_1:f_2) \geq 2$ and $e_x(f_1:f_2)=c_x$. Next suppose that there is no $(f_1:f_2)$ satisfying the hypotheses. Then we can still take  $f_1, f_2 \in \mathbb F_q(C) $ with $e_x(f_1:f_2)=c_x$ and use Lemma \ref{height-triangularizable-characterization} to force $c=0$, and our other claim is vacuously true.

By Lemmas \ref{first-newform-bound} and \ref{f1-f2-finder}, we have

\[ \frac{ \left| f \left( \begin{pmatrix} a & b \\ c & d \end{pmatrix} \right) \right|}{ |C_f| }  - q^{1/2} 2^{ \deg N -3 } \leq  q \sum_{\substack{ (e_x' ):  \Sing \to \mathbb N \\  e_x' \leq c_x \\ \sum_x e_x' \leq n }} \left (\prod_{x \in \Sing } {c_x \choose e_x'} \right)   \mathcal B( d_{\alpha_{a,d,b/d},(e_x')})    \] 
\[ \leq  q \sum_{\substack{ (e_x' ):  \Sing \to \mathbb N \\  e_x' \leq c_x \\ \sum_x e_x' \leq n }} \sum_{ \substack{ (f_1:f_2) \in \mathbb P^1 (\overline{\mathbb F}_q (C)) \\ (e_x') \in S_{(f_1:f_2)} \\ f_1 \neq 0 } }   \left (\prod_{x \in \Sing } {c_x \choose e_x'} \right)   \mathcal B( h^*(f_1:f_2) - 2- \sum_x |e_x' - e_x (f_1:f_2) | ) \]
\[ \leq   q  \sum_{ \substack{ (f_1:f_2) \in \mathbb P^1 (\overline{\mathbb F}_q (C)) \\ h^* (f_1:f_2) \geq 2 \\ f_1 \neq 0  } }  \sum_{ (e_x' ) \in S_{ (f_1:f_2) } }  \left (\prod_{x \in \Sing } {c_x \choose e_x'} \right)   \mathcal B( h^*(f_1:f_2) - 2- \sum_x |e_x' - e_x (f_1:f_2) | )  \]
\[ \leq   q  \sum_{ \substack{ (f_1:f_2) \in \mathbb P^1 (\overline{\mathbb F}_q (C)) \\ h^* (f_1:f_2) \geq 2 \\ (e_x (f_1:f_2)) \neq (c_x)   } }  \sum_{(e_x' ) \in S_{ (f_1:f_2) } }  \left (\prod_{x \in \Sing } {c_x \choose e_x'} \right)   \mathcal B( h^*(f_1:f_2) - 2- \sum_x |e_x' - e_x (f_1:f_2) | ) , \]
giving the desired bound, where on the third line we use the fact that if $S_{(f_1:f_2)}$ is nonempty then $h^*(f_1:f_2)\geq 2$ and in the third line we use our earlier assumption that if $h^*(f_1:f_2) \geq 0$ and $f_1 \neq 0$ then $(e_x(f_1:f_2)) \neq (c_x)$.  

\end{proof} 

\section{Atkin-Lehner operators and conclusion}\label{Atkin-Lehner}

We now use Atkin-Lehner operators to further optimize Lemma \ref{cusp-newform-bound}. We first construct the Atkin-Lehner operators in our setting, in Lemma \ref{Atkin-Lehner-Swap}, which shows that for our cuspidal newform $f$ we can associate a modified form $f'$ such that the size of $f$ at a point $\begin{pmatrix} a & b \\ c & d \end{pmatrix}  $ equals the size of $f'$ at the translation of $\begin{pmatrix} a & b \\ c & d \end{pmatrix}  $ by the special matrix $\begin{pmatrix} 0 &  \pi_v^{-c_v}  \\  1& 0 \end{pmatrix} $. This relies on multiplying $f$ by an auxiliary character constructed in Lemma \ref{character-existence-lemma}. Because $f'$ also satisfies the bound of Lemma \ref{cusp-newform-bound}, our goal is to apply this operation until $\begin{pmatrix} a & b \\ c & d \end{pmatrix}  $ is translated to a point that minimizes this bound.

To do this, we need to understand how the $h$ function changes when we translate by this special matrix. In Lemma \ref{e-switching-lemma}, we will see that, after this translation, $h$ has the same dependence on $(f_1:f_2)$ but a new dependence on $(e_x)$. In terms of $h^*(f_1:f_2)$ and $e_x(f_1:f_2)$, this preserves $h^*(f_1:f_2)$ but changes $e_x(f_1:f_2)$.

Our goal is then to take the value of $(f_1:f_2)$ that maximizes $h^*(f_1:f_2)$ and multiply by a suitable sequences of these matrices to make $e_x(f_1:f_2) = c_x$ for all $x$. Doing this will make the term associated to $(f_1:f_2)$ disappear from the sum in Lemma \ref{cusp-newform-bound}, leading to a better bound.

In the remaining lemmas, we will combinatorially manipulate this bound until we obtain one independent of $\begin{pmatrix} a & b \\ c & d \end{pmatrix}  $, from which we deduce our main theorem.

\begin{lemma}\label{character-existence-lemma} Let $\eta$ be a continuous character of $F^\times \backslash \mathbb A_F^\times $. Assume that the restriction of $\eta$ to $F_v^\times$ is is trivial on $\mathbb F_q^\times \subseteq F_v^\times$. Then there exists a finite order character $\theta$ of $F^\times \backslash \mathbb A_F^\times$, unramified away from $v$, that agrees with $\eta$ on $\mathcal O_{F_v}^\times$. \end{lemma}

\begin{proof} Consider the map \[\mathcal O_{F_v}^\times \to  F^{\times} \backslash \mathbb A_F^\times / \prod_{ \substack{ w \in |C| \\ w\neq v}} \mathcal O_{F_w}^\times.\] Its kernel is  $F^\times \cap \prod_{ \substack{ w \in |C| }} \mathcal O_{F_w}^\times = \mathbb F_q^\times$. Its cokernel is $  F^{\times} \backslash \mathbb A_F^\times / \prod_{  w \in |C| } \mathcal O_{F_w}^\times =\mathbb Z$. So any continuous character of $\mathcal O_{F_v}^\times$, trivial on the kernel, can be extended to a continuous character of $\mathbb A_F^\times / \prod_{ \substack{ w \in |C| \\ w\neq v}} \mathcal O_{F_w}^\times$ that is trivial on some fixed inverse image of the generator of $\mathbb Z$. Let $\theta$ be the extension in this way of the restriction of $\eta$ to $\mathcal O_{F_v}^\times$. Because the restriction of $\eta$ to $\mathcal O_{F_v}^\times $ is a continuous character of a compact group, it has finite order. Because $\theta$ is trivial on some element whose image generates the quotient group, the order of $\theta$ equals the order of $\eta$ and is finite. \end{proof}

 \begin{lemma}\label{Atkin-Lehner-Swap} Let $f$ be a cuspidal newform of level $N$ whose central character $\eta$ has finite order. Let $v$ be a closed point in the support of $N$. Assume that for restriction of $\eta$ to the global units $\mathbb F_q^\times \subset F_v^\times$ is trivial.
 
 Then there exists a cuspidal newform $f'$ of level $N$ such that  for all $\begin{pmatrix} a & b \\ c & d \end{pmatrix} \in GL_2(\mathbb A_F)$,
\begin{equation}\label{f-f'-relation} \left| f \left( \begin{pmatrix} a & b \\ c & d \end{pmatrix} \right) \right| = \left|  f' \left( \begin{pmatrix} a & b \\ c & d \end{pmatrix}  \begin{pmatrix} 0 &  \pi_v^{-c_v}  \\  1& 0 \end{pmatrix} \right)\right|\end{equation} 
and $|C_f|=|C_{f'}|$. Furthermore, the central character of $f'$ has finite order,  order and agrees with $\eta$ when restricted to the units of any place except $v$, and agrees with $\eta^{-1}$ when restricted to the units of $v$.

\end{lemma}

\begin{proof} By Lemma \ref{character-existence-lemma} there exists a finite order character $\theta$ of $\mathbb A_F^\times / F^\times$, unramified away from $v$, that agrees with $\eta$ on $\mathcal O_{F_v}^\times$. Then \[ f \left( \begin{pmatrix} a & b \\ c & d \end{pmatrix} \right) \theta( ad-bc) \] is left nvariant under $GL_2(F)$, right invariant under $\Gamma_1(N)$ away from $v$ and invariant under the subgroup of   $\begin{pmatrix} a & b \\ c & d \end{pmatrix} \in GL_2( \mathcal O_{ F_v})$ where $c \equiv 0 \mod \pi_v^{c_v}$, $a\equiv 1 \mod \pi_v^{c_v}$. 

Let \[ f' \left( \begin{pmatrix} a & b \\ c & d \end{pmatrix} \right) =  f \left( \begin{pmatrix} a & b \\ c & d \end{pmatrix}  \begin{pmatrix} 0 & 1 \\ \pi_v^{c_v} & 0 \end{pmatrix} \right) \theta( ad-bc) .\]

Then $f'$ is left invariant under $GL_2(F)$, right invariant under $\Gamma_1(N)$, cuspidal, and a Hecke eigenform. Its Langlands parameter is simply the parameter $V$ of $f$ twisted by the character $\theta^{-1}$, which as an inertia representation at $v$ is the dual representation. Thus, because the restriction to inertia of the Langlands parameter of $f$ agrees with either the Langlands parameter of $f$ or its dual at every place, the conductor of the Langlands parameter of $f'$ is $N$, and so the level of $f'$ is equal to the conductor of the Langlands parameter of $f'$, and thus $f'$ is a newform.

Since $|\theta(ad-bc)|=1$, we have \eqref{f-f'-relation}.

The $L^2$-norm of $f$ equals the $L^2$ norm of $f'$, and the adjoint $L$-function of $f$ equals the adjoint $L$-function of $f'$. By Lemma \ref{rankin-selberg-normalized}, it follows that $|C_f|=|C_{f'}|$.

The central character of $f'$ is $\eta \theta^{-2}$, which agrees with $\eta$ on the units of every place other than $v$ since $\theta$ is unramified away from $v$, and agrees with $\eta^{-1}$ when restricted to the units of $v$ by assumption on $\theta$.

\end{proof}

\begin{lemma}\label{e-switching-lemma} Fix a closed point $v \in |N|$ and $(e_x): \Sing \to \mathbb N$ with $e_x\leq c_x$.

Let $\tilde{e}_x $ equal $ e_x$ for $x$ not lying over $v$ and equal $c_x -e_x$ for $x$ lying over $v$. Then

\[ h \left(  \begin{pmatrix} a & b \\ c & d \end{pmatrix}  \begin{pmatrix} 0 &  \pi_v^{-c_v}  \\  1& 0 \end{pmatrix} , (f_1:f_2), (e_x)\right) = h \left( \begin{pmatrix} a & b \\ c & d \end{pmatrix} , (f_1:f_2), (\tilde{e}_x)\right) . \]
 \end{lemma}
 
 \begin{proof} We must show that
 \[ 2\deg ( \min( \operatorname{div} (af_1+cf_2), \operatorname{div}(bf_1+ df_2)+ \sum_x e_x [x] )) - \deg ( \operatorname{div}(ad-bc) + \sum_x e_x[x])\]
  \[= 2\deg ( \min(\operatorname{div}(bf_1+ df_2),  \operatorname{div} (\pi_v^{-c_v} af_1+\pi_v^{-c_v} cf_2)+ \sum_x \tilde{e}_x [x] )) - \deg ( \operatorname{div}( \pi_v^{-c_v} (ad-bc) + \sum_x \tilde{e}_x[x]).\]
  
  It suffices to show that the contributions to the degrees from the valuations at each place are equal. At places not over $v$, this is immediate, so fix $x$ over $v$. We must show
  \[ 2  \min ( v_x ( af_1+cf_2), v_x(bf_1+df_2) +e_x)  - (v_x(ad-bc) + e_x)\]\[  = 2 \min( v_x(bf_1+df_2), v_x(a f_1 + c f_2) - c_v + c_v-e_x) - ( v_x(ad-bc)  -  c_v +c_v-e_x)  .\]
  
 This is straightforward because
 \[  2 \min( v_x(bf_1+df_2), v_x(a f_1 + c f_2) - c_v + c_v-e_x) - ( v_x(ad-bc)  -  c_v +c_v-e_x) \] \[=2 \min( v_x(bf_1+df_2), v_x(a f_1 + c f_2) -e_x) - ( v_x(ad-bc) -e_x) \] \[  =2 \min( v_x(bf_1+df_2) +e_x, v_x(a f_1 + c f_2) ) - ( v_x(ad-bc) -e_x)-2e_x \] \[=  2  \min ( v_x ( af_1+cf_2), v_x(bf_1+df_2) +e_x)  -( v_x(ad-bc) + e_x) .\]

 \end{proof}

\begin{lemma}\label{Atkin-Lehner-bound} Let $f$ be a cuspidal newform of level $N$ whose central character $\eta$ has finite order.

Assume that, for all closed points $v$ in the support of $N$, the restriction of $\eta$ to the global units $\mathbb F_q^\times \subset F_v^\times$ is trivial. Then

\[ \frac{  \left| f \left( \begin{pmatrix} a & b \\ c & d \end{pmatrix} \right) \right| }{   |C_f| } \leq  2 ^{ \deg N-3} q^{1/2}  +\] \[\hspace*{-1cm} q  \sum_{\substack{ (f_1:f_2) \in \mathbb P^1(\overline{\mathbb F}_q(C)) \\ h^* (f_1:f_2) \geq 2 }} \sum_{\substack{ (e_x ')\in S_{(f_1:f_2)} } }\left (\prod_{x \in \Sing } {c_x \choose e_x'} \right)   \mathcal B\left( h^*(f_1:f_2) -2 - \sum_x |e_x' -e_x(f_1:f_2)|\right)  \] \[\hspace*{-1cm} -q   \max_{\substack{ (f_1:f_2) \in \mathbb P^1(\mathbb F_q(C)) \\ h^* (f_1:f_2) \geq 2 \\ e_x(f_1:f_2) \in \{0,c_x\} \textrm{ for all }x}} \sum_{\substack{ (e_x ')\in S_{(f_1:f_2)}}} \left (\prod_{x \in \Sing } {c_x \choose e_x'} \right)   \mathcal B\left( h^*(f_1:f_2) -2 - \sum_x |e_x' -e_x(f_1:f_2)| \right)  .\] \end{lemma}

\begin{proof}   Fix $(f_1:f_2)$ maximizing \[\sum_{\substack{ (e_x ')\in S_{(f_1:f_2)} }} \left (\prod_{x \in \Sing } {c_x \choose e_x'} \right)   \mathcal B( h^*(f_1:f_2) -2 - \sum_x |e_x' -e_x(f_1:f_2)|)  .\] We prove the stated bound by induction on the number of geometric points $x$ with $e_x(f_1:f_2)=0$.

First assume that this number is zero. Because there is a unique $(f_1:f_2)$ with $h^*(f_1:f_2) \geq 2$ and $e_x(f_1:f_2)=c_x$ by Lemma \ref{height-unique-cusp}, the stated bound is exactly the bound of Lemma \ref{cusp-newform-bound}, where the $\max$ term cancels the unique term in the sum with $e_x(f_1:f_2)= c_x$.

Next assume that it is positive. By Lemma \ref{height-orbit-size}, because $(f_1:f_2)$ is defined over $\mathbb F_q$, $e_x$ is stable under the Galois action, and so there is some place $v$ with $e_x=0$ for $x \in v$.  By Lemma \ref{Atkin-Lehner-Swap}, there exists a cuspidal newform $f'$ of level $N$ satisfying \eqref{f-f'-relation} and such that $|C_{f'}|= |C_f|$. Furthermore, the central character of $f'$ is trivial on the global units at each place.

 We will prove our bound for $f$ by inductively using the corresponding bound for $f'$.
 
 Here it is necessary to keep track of the dependence, suppressed everywhere else, of $h^*(f_1:f_2)$ and $e_x(f_1:f_2)$ on $ \begin{pmatrix} a & b \\ c & d \end{pmatrix}$, which we will do by writing them as $h^*\left( (f_1:f_2),  \begin{pmatrix} a & b \\ c & d \end{pmatrix}\right)$ and $e_x\left( (f_1:f_2),  \begin{pmatrix} a & b \\ c & d \end{pmatrix}\right)$.
 
 It follows from Lemma \ref{e-switching-lemma} that
 \[ h^*\left( (f_1:f_2),  \begin{pmatrix} a & b \\ c & d \end{pmatrix}\begin{pmatrix} 0 & 1 \\ \pi_v^{c_v} & 0 \end{pmatrix}\right) = h^*\left( (f_1:f_2),  \begin{pmatrix} a & b \\ c & d \end{pmatrix}\right) \]
 and
 \[ e_x\left( (f_1:f_2),  \begin{pmatrix} a & b \\ c & d \end{pmatrix}\begin{pmatrix} 0 & 1 \\ \pi_v^{c_v} & 0 \end{pmatrix} \right) = \tilde{e}_x\left( (f_1:f_2),  \begin{pmatrix} a & b \\ c & d \end{pmatrix}\right) .\]
 Hence the sum over $S_{(f_1:f_2)}$, for a given $(f_1:f_2)$, is the same whether we work with $ \begin{pmatrix} a & b \\ c & d \end{pmatrix}$ or $ \begin{pmatrix} a & b \\ c & d \end{pmatrix}\begin{pmatrix} 0 & 1 \\ \pi_v^{c_v} & 0 \end{pmatrix} $. In particular, the same $(f_1:f_2)$ maximizes this sum in both cases. 
 
 Thus, we may apply the induction hypothesis to get the desired bound for $ f' \left( \begin{pmatrix} a & b \\ c & d \end{pmatrix}  \begin{pmatrix} 0 &  \pi_v^{-c_v}  \\  1& 0 \end{pmatrix} \right) ,$ because the number of $x$ with $e_x \left((f_1:f_2),  \begin{pmatrix} a & b \\ c & d \end{pmatrix}\begin{pmatrix} 0 & 1 \\ \pi_v^{c_v} & 0 \end{pmatrix}\right )=0$ is equal to the number of $x$ with $e_x \left((f_1:f_2),  \begin{pmatrix} a & b \\ c & d \end{pmatrix}\right)=0$ minus the number of $x$ lying over $v$. Then using Equation \eqref{f-f'-relation}, we get the desired bound for $f \left(  \begin{pmatrix} a & b \\ c & d \end{pmatrix}   \right)$, concluding the induction step.\end{proof}

\begin{definition} Let $\mathcal S(a,b)$ be the coefficient of $u^a$ in $\frac{(1+u)^b}{  (1-u)(1+u)^2 (1- (2\sqrt{q}+1) u)}$. \end{definition}

Because this differs from the definition of $\mathcal B(a)$ only in the $(1+u)^b$ term in the numerator, it follows from the binomial theorem that \begin{equation}\label{S-B-relation} \mathcal S(a,b) = \sum_{k=0}^b {b\choose k} \mathcal B(a-k).\end{equation}

\begin{lemma}\label{local-squarefree-bound} Let $f$ be a cuspidal newform of level $N$ whose central character has finite order.

Assume that $N$ is squarefree (i.e. $c_x=1$ for all $x \in \Sing$) and for all closed points $v$ in the support of $N$, the restriction of $\eta$ to the global units $\mathbb F_q^\times \subset F_v^\times$ is trivial. 

Then \begin{equation}\label{eq-local-squarefree} \left| f \left( \begin{pmatrix} a & b \\ c & d \end{pmatrix} \right) \right| \leq  |C_f|  \Biggl( 2 ^{ \deg N-3} q^{1/2} + q \sum_{\substack{ (f_1:f_2) \in \mathbb P^1(\overline{\mathbb F}_q(C)) \\ \deg N /2 \geq h^* (f_1:f_2) \geq 2 }}  \mathcal S( h^*(f_1:f_2) -2 , \deg N ) \Biggr).
\end{equation}
\end{lemma}

\begin{proof} Because $c_x=1$ for all $x$, $e_x=0$ or $1$ for all $x$, thus $e_x=0$ or $c_x$ for all $x$. Because of this, Lemma \ref{Atkin-Lehner-bound} reduces to

\begin{equation}\label{squarefree-simplified-bound} \begin{aligned}  \left| f \left( \begin{pmatrix} a & b \\ c & d \end{pmatrix} \right) \right|/ |C_f|  \hspace{300pt}\\ 
\leq    2 ^{ \deg N-3} q^{1/2}  + q  \sum_{\substack{ (f_1:f_2) \in \mathbb P^1(\overline{\mathbb F}_q(C)) \\  h^* (f_1:f_2) \geq 2 }}  \mathcal S( h^*(f_1:f_2) -2 , \deg N )  - q   \max_{\substack{ (f_1:f_2) \in \mathbb P^1(\overline{\mathbb F}_q(C)) \\ h^* (f_1:f_2) \geq 2} } \mathcal S( h^*(f_1:f_2) -2 , \deg N ) .\end{aligned}\end{equation}

This uses \eqref{S-B-relation}, which implies that $\mathcal S(h^*(f_1:f_2)-2, \deg N)$ matches exactly the sum over $(e_x') \in S_{(f_1:f_2)} $, because there are ${\deg N \choose k}$ tuples $(e_x')$ with $\sum_x |e_x(f_1:f_2) -e_x'| = k$. 

Next observe that there is at most one $(f_1:f_2)$ with $\deg N/2 < h^*(f_1:f_2)$: Given two such $(f_1:f_2)$ and $(f_3:f_4)$, we have \[\sum_{x \in \Sing} | e_x(f_1:f_2) - e_x(f_3:f_4) | \leq \sum_{x\in \Sing} 1 = \deg N\] so by  Lemma \ref{height-unique-cusp}. and  Lemma \ref{height-mountain-shape}   \[0 \geq   h \left(  \begin{pmatrix} a, b\\ c,d \end{pmatrix},  (f_1:f_2), (e_x(f_1:f_2)) \right) + h \left(  \begin{pmatrix} a, b\\ c,d \end{pmatrix},  (f_3:f_4), (e_x(f_1:f_2)) \right)\] \[= h^*(f_1:f_2) + h^*(f_3:f_4) -  \sum_{x \in \Sing} | e_x(f_1:f_2) - e_x(f_3:f_4) | > \frac{\deg N}{2} + \frac{\deg N}{2} - \deg N ,\] a contradiction.

Because there is at most one such $(f_1:f_2)$, if there is any, it is Galois-invariant by Lemma \ref{height-orbit-size}. Thus it lies in $\mathbb P^1(\mathbb F_q(T))$ and it is canceled (at least) by the $\max$ term in Equation \eqref{squarefree-simplified-bound}. So we may remove it from the summation, obtaining Equation \eqref{eq-local-squarefree}.\end{proof}

\begin{remark} If we do not make assumptions on the level or central character but assume that  $ h^* (f_1:f_2) \leq \deg N/2$ for all $(f_1:f_2) \in \mathbb P^1 (\mathbb F_q(T))$, a similar bound as Lemma \ref{local-squarefree-bound} holds, for the same logic that there is at most one $(f_1:f_2)$ with $\deg N/2 < h^*(f_1:f_2)$ and therefore it must be defined over $\mathbb F_q(T)$. This is equivalent to assuming that the point $\begin{pmatrix} a & b \\ c & d \end{pmatrix}$ is not too close to any cusp. However, the definition of $\mathcal S(a,b)$ must be modified to depend on $(e_x ( f_1:f_2) )$.\end{remark}

\begin{lemma}\label{height-packing-bound} Assume that $N$ is squarefree. For any $\begin{pmatrix} a& b \\ c& d \end{pmatrix}$, we have 
\[  \sum_{\substack{ (f_1:f_2) \in \mathbb P^1(\overline{\mathbb F}_q(C)) \\  h^* (f_1:f_2) \geq 1 }}  \sum_{k=0}^{ h^* (f_1:f_2)-1} {\deg N \choose k} \leq 2^{\deg N}.\]\end{lemma}

\begin{proof} By Lemma \ref{height-mountain-shape}, $\sum_{k=0}^{ h^* (f_1:f_2)-1} {\deg N \choose k} $ is the size of the set of $e_x$ such that \[h \left( \begin{pmatrix} a& b \\ c& d \end{pmatrix}, (f_1:f_2), (e_x)\right)>0.\] By Lemma \ref{height-unique-cusp}, these sets of $(e_x)$ do not overlap, so their total size is at most the total number of $(e_x)$, which is $2^{\deg N}$. \end{proof} 

\begin{lemma}\label{B-increasing} \[ \frac{ \mathcal S(a, \deg N) }{  \sum_{k=0}^{a+1} {\deg N \choose k} }\] is increasing as a function of $a$ when $a$ ranges over integers at least $-1$. \end{lemma} 

\begin{proof} Let \[F(a, n) =  \frac{ \mathcal S(a, n) }{  \sum_{k=0}^{a+1} {n \choose k} }.\]

We must prove that $F(a,n)$ is increasing as a function of $a$, which we will do by induction on $n$, starting at $n=2$.

By definition we have $\sum_a \mathcal S(a,2) u^a = \frac{1}{ (1-u) (1-(2\sqrt{q}+1)u)}$ so we have \[ \sum_a (\mathcal S(a,2) - \mathcal S(a-1,2)) u^a  = \frac{1}{ 1-(2\sqrt{q} +1)u}\] and thus  \[\mathcal S(a,2) = \mathcal S(a-1,2)+ (2\sqrt{q} +1)^a > \mathcal S(a-1,2).\] Now for $a$ at least $1$, the denominator $\sum_{k=0}^{a+1} {2 \choose k} $ is constant, so the ratio $F(a,2) $ is increasing for $a$ at least $1$. For $a=-1, 0,1$ the ratio $F(a,2)$ is $0, 1/3, (2\sqrt{q}+2)/4$ respectively so in fact the sequence is increasing for all $a$.

For the induction step, we use the identities \[ \mathcal S(a, n ) = \mathcal S(a-1, n-1)+ \mathcal S(a, n-1)\] and \[\sum_{k=0}^{a+1} {n \choose k} = \sum_{k=0}^{a} {n-1 \choose k} + \sum_{k=0}^{a +1} { n-1 \choose k} .\] These identities make $F(a,n)$ a convex combination of $F(a-1,n-1)$ and $F (a, n)$, so assuming $F (a,n-1)$ is increasing in $a$, we have $F ( a-1, n-1) < F (a, n) < F (a, n-1)$. This shows that $F(a,n)$ is increasing for $a$ at least zero, and because $\mathcal S(-1,n)=0$ while $\mathcal S(0,n)=1$, $F(a,n)$ is increasing for $a$ at least $-1$, giving the induction step.\end{proof}

\begin{lemma}\label{B-binomial-estimate} We have \[ \mathcal S(a,b) \leq \frac{1}{ 2\sqrt{q}} \frac{ (2\sqrt{q}+2)^{b-2} }{ (2\sqrt{q}+1)^{b-3-a}}  .\]\end{lemma}

\begin{proof}  $\mathcal S(a, b)$ is the coefficient of $u^a$ in \[\frac{ (1+u)^{b-2} } { (1-u) (1 - (2\sqrt{q}+1) u)} .\] We have \[\frac{1}{ (1-u) (1 - (2\sqrt{q}+1) u)} = \frac{1}{2 \sqrt{q} u } \left( \frac{1}{ 1- (2\sqrt{q}+1) u} - \frac{1}{ 1-u} \right),\] so this is $1/(2\sqrt{q})$ times the coefficient of $u^{a+1} $ in  $\frac{ (1+u)^{b-2} }{ 1- (2\sqrt{q}+1) u}$ minus the coefficient of $u^{a+1}$ in $\frac{ (1+u)^{b-2} }{ 1-u}$. The coefficient of $u^{a+1}$ in  $\frac{ (1+u)^{b-2} }{ 1-u}$ is nonnegative so subtracting it can only lower our bound and thus we can ignore it.

The coefficient of $u^{a+1}$ in  $\frac{ (1+u)^{b-2} }{ 1- du} $ is \[  \sum_{k=0}^{\min( b-2, a+1)} {b-2 \choose k} d^{a+1-k} \leq \sum_{k=0}^{b-2}  {b-2 \choose k} d^{a+ 1-k}= \frac{(d+1)^{b-2}}{d^{b-3-a}}\] so this is at most

\[ \frac{1}{ 2\sqrt{q}} \frac{ (2\sqrt{q}+2)^{b-2} }{ (2\sqrt{q}+1)^{b-3-a}} .\]

\end{proof}

\begin{lemma}\label{final-Whittaker-bound} Let $f$ be a cuspidal newform of level $N$ whose central character has finite order.

Assume that $N$ is squarefree (i.e. $c_x=1$ for all $x \in \Sing$) and for all closed points $v$ in the support of $N$, the restriction of $\eta$ to the global units $\mathbb F_q^\times \subset F_v^\times$ is trivial.

Then \[ \left| f \left( \begin{pmatrix} a & b \\ c & d \end{pmatrix} \right) \right|=  O \left( |C_f | \left( \frac{2 \sqrt{q}+2}{\sqrt{ 2 \sqrt{q} +1}}\right)^{\deg N} \right).\]where the constant in the big $O$ is completely uniform. \end{lemma}

\begin{proof}By Lemma \ref{local-squarefree-bound}, Lemma \ref{B-increasing}, Lemma \ref{height-packing-bound}, and Lemma \ref{B-binomial-estimate}, we have

 \[ \left| f \left( \begin{pmatrix} a & b \\ c & d \end{pmatrix} \right) \right| \leq  |C_f|  \Biggl( 2 ^{ \deg N-3}  q^{1/2} + q \sum_{\substack{ (f_1:f_2) \in \mathbb P^1(\overline{\mathbb F}_q(C)) \\ \deg N /2 \geq h^* (f_1:f_2) \geq 2 }}  \mathcal S( h^*(f_1:f_2) -2 , \deg N )  \Biggr)\]
 
 \[ \leq |C_f|  \Biggl( 2 ^{ \deg N-3}q^{1/2}   +q  \sum_{\substack{ (f_1:f_2) \in \mathbb P^1(\overline{\mathbb F}_q(C)) \\ \deg N /2 \geq h^* (f_1:f_2) \geq 2 }}  \frac{ \mathcal S(\lfloor \deg N/2\rfloor-2, \deg N) } {\sum_{k=0}^{\lfloor \deg N/2\rfloor-1} {\deg N \choose k} }  \sum_{k=0}^{ h^*(f_1,f_2)-1} {\deg N \choose k}  \Biggr).\] 
 
 \[ \leq |C_f|\left(q^{1/2}  2 ^{ \deg N-3}  +q 2^{\deg N}  \frac{ \mathcal S(\lfloor \deg N/2\rfloor-2, \deg N) } {\sum_{k=0}^{\lfloor \deg N/2\rfloor-1} {\deg N \choose k} } \right) .\]
 
  \[ \leq  |C_f|  \left(  2 ^{ \deg N-3}  q^{1/2}+\frac{q  2^{\deg N}}{ \sum_{k=0}^{\lfloor \deg N/2\rfloor-1} {\deg N \choose k} }\frac{1}{ 2 \sqrt{q}}   \frac{ (2\sqrt{q}+2)^{\deg N-2} }{ (2\sqrt{q}+1)^{\lceil \deg N/2 \rceil-1 }}  \right) \]

 \[ \leq  |C_f| q^{1/2} \left(  2 ^{ \deg N-3}  +\frac{ 2^{\deg N-1}}{ \sum_{k=0}^{\lfloor \deg N/2\rfloor-1} {\deg N \choose k} } \frac{ (2\sqrt{q}+2)^{\deg N-2} }{ (2\sqrt{q}+1)^{\lceil \deg N/2 \rceil-1 }} \right) .\] 
 
Because $f$ is a cusp form of level $N$, we must have $\deg N \geq 4$. It follows that \[ \frac{ 2^{\deg N-1}}{ \sum_{k=0}^{\lfloor \deg N/2\rfloor-1} {\deg N \choose k} } = O(1).\]
  
 This gives
 \[ \left| f \left( \begin{pmatrix} a & b \\ c & d \end{pmatrix} \right) \right| \leq  |C_f| q^{1/2}\left(  2 ^{ \deg N-3}  +O \left( \frac{ (2\sqrt{q}+2)^{\deg N-2} }{ (2\sqrt{q}+1)^{\lceil \deg N/2 \rceil-1 }}    \right) \right).\]

  We have \[ \frac{ (2\sqrt{q}+2)^{\deg N-2} }{ (2\sqrt{q}+1)^{\lceil \deg N/2 \rceil -1 }}  \leq \left( \frac{2 \sqrt{q}+2}{\sqrt{ 2 \sqrt{q} +1}}\right)^{\deg N-2} .\]
  
Note that \[ \frac{2 \sqrt{q}+2}{\sqrt{ 2 \sqrt{q} +1}} \geq 2\] because $4 q + 8 \sqrt{q} +4 \geq 8 \sqrt{q} + 4$ and thus 
  \[  2 ^{ \deg N-3}   \leq  2^{\deg N -2} \leq   \left( \frac{2 \sqrt{q}+2}{\sqrt{ 2 \sqrt{q} +1}}\right)^{\deg N-2} .\]
  so
  \[ \left| f \left( \begin{pmatrix} a & b \\ c & d \end{pmatrix} \right) \right|=  O \left( |C_f|  q^{1/2}  \left( \frac{2 \sqrt{q}+2}{\sqrt{ 2 \sqrt{q} +1}}\right)^{\deg N-2} \right).\]
  
  Note in addition that
 
\[ \left( \frac{2 \sqrt{q}+2}{\sqrt{ 2 \sqrt{q} +1}} \right)^2 \geq q^{1/2} \] because $4q + 8 \sqrt{q} +4 \geq 2 q + \sqrt{q}$. Thus
  \[ \left| f \left( \begin{pmatrix} a & b \\ c & d \end{pmatrix} \right) \right|=  O \left( |C_f | \left( \frac{2 \sqrt{q}+2}{\sqrt{ 2 \sqrt{q} +1}}\right)^{\deg N} \right).\]
as desired.

\end{proof}

We now recall the statement of Theorem \ref{sup-norm-intro}, and prove it.

\begin{theorem}\label{sup-norm-final}[Theorem \ref{sup-norm-intro}] Let $ F = \mathbb F_q(T)$, let $N$ be a squarefree effective divisor on $\mathbb P^1$, and let $f: GL_2(\mathbb A_F) \to \mathbb C$ be a cuspidal newform of level $N$ with unitary central character. Assume that for each place $v$ in the support of $N$, the restriction of the central character of $f$ to $\mathbb F_q^\times \subset F_v^\times$ is trivial. Then 
\[ ||f||_{\infty} = O \left(  \left(\frac{ 2 \sqrt{q} +2}{ \sqrt{ 2 \sqrt{q}+ 1}}\right)^{ \deg N} \right)\] 
if $f$ is  Whittaker normalized and
\[ ||f||_{\infty} =  O \left(  \left(\frac{ 2  (1+ q^{-1/2} ) }{\sqrt{ 2 \sqrt{q}+ 1}}\right)^{ \deg N} \log(\deg N)^{3/2} \right) \]
if $f$ is $L^2$-normalized. \end{theorem}

\begin{proof} We see that the assumptions of Theorem \ref{sup-norm-intro} match exactly the assumptions of Lemma \ref{final-Whittaker-bound}, except that we have a bound in the case of a central character of finite order and we wish to prove a bound in the case of a unitary central character.  To reduce to the finite order case we observe that if $f$ has unitary central character then
\[ f\left( \begin{pmatrix} a & b\\ c & d \end{pmatrix} \right) \alpha^{ \deg (ad-bc)} ,\] where $\alpha$ is the square root of the value of $\eta$ on some fixed adele of degree $1$, has finite order central character and the same maximum value as $f$. 

If $f$ is Whittaker normalized then $|C_f|=1$ and our desired bound is exactly Lemma \ref{final-Whittaker-bound}.

If $f$ is $L^2$-normalized then it follows from Lemma \ref{rankin-selberg-normalized} that
\[1= 2 |C_f|^2   q^{2g-2 +  \deg N }    L (1 , \operatorname{ad} \mathcal F )  \prod_{v | N}  \det( 1- q^{- \deg v} \Frob_{|\kappa_v|}, (\mathcal F \otimes \mathcal F^\vee)^{I_v} / (\mathcal F^{I_v} \otimes (\mathcal F^\vee)^{I_v} )  )  / (1 -q^{- \deg v} ) .\]

 Furthermore, because $N$ is squarefree, the local monodromy representation of $\mathcal F$ at any point in the support of $N$ is either a rank two unipotent representation or a trivial representation plus a one-dimensional character. In either case, one can check that
\[  (\mathcal F \otimes \mathcal F^\vee)^{I_v} / (\mathcal F^{I_v} \otimes (\mathcal F^\vee)^{I_v} )\] is one-dimension with trivial $\Frob_q$ action. This, and the fact that $g=0$, gives

 \[1= 2 |C_f|^2   q^{\deg N - 2}    L (1 , \operatorname{ad} \mathcal F )  .\]

\[ |C_f| =2^{-1/2}   q^{1 -  \deg N/2  }    L (1 , \operatorname{ad} \mathcal F )^{-1/2}   .\]

By Lemma \ref{good-L-value-bound}, \[ L (1 , \operatorname{ad} \mathcal F )^{-1/2} =O (  (\log \deg N)^{3/2} ).\]  This gives

\[ |C_f| \leq q^{1 - \deg N/2} (\log \deg N )^{O(1)}\] and plugging this into Lemma \ref{final-Whittaker-bound} gives the stated bound in the $L^2$-normalized case. \end{proof}

\appendix

\section{Standard analytic number theory in the function field setting}

This appendix contains some now-standard results in analytic number theory - the Fourier expansion of modular forms, a Rankin-Selberg formula for the $L^2$-norm of a modular form in terms of an $L$-function special value, and an estimate for that special value using the Riemann hypothesis - done in the level of generality needed for this paper, except that we do not assume that $C = \mathbb P^1$.

\begin{lemma}\label{Drinfeld-formula} For any newform $f$ of level $N$ whose central character has finite order, there exists $\mathcal F$ an irreducible middle extension sheaf of rank two on $C$, pure of weight $0$, of conductor $N$, and $C_f\in \mathbb C$ such that

\[ f \left( \begin{pmatrix} a & bz \\ 0 & b \end{pmatrix} \right) =C_f  q^{- \frac{ \deg (\omega_0 a/b) }{2} } \eta(b)^{-1} \sum_{\substack{ w \in F^{\times} \\ \operatorname{div} (w \omega_0 a/b ) \geq 0}}  \psi(wz ) r_{\mathcal F} (\operatorname{div}(  w \omega_0 a/b )) . \]

$\mathcal F$ and $C_f$ are unique with this property.

\end{lemma}

The notations $\omega_0, \eta, \psi, r_{\mathcal F}$ used in this statement are all explained in \S\ref{ss-fourier}.

This formula is a slight variant of one proved by \citet[(4)]{Drinfeld}, who handled the case where the level $N$ is trivial. Peter Humphries explained to me where to look in the literature for the tools to perform this calculation.

\begin{proof} This follows from the Langlands correspondence for $GL_2(F)$ proven by Drinfeld, though we will find it more convenient to use the version stated by Laurent Lafforgue.

Let $\pi $ be the representation of $GL_2(\mathbb A_F)$ generated under right translation by $f$. In other words, $\pi$ is a space of functions $f'$ on $GL_2(\mathbb A_F)$, with an action of $GL_2(\mathbb A_F)$ where ${\mathbf g} \in GL_2(\mathbb A_F)$ takes $h \mapsto f'({\mathbf h})$ to $h \mapsto f'({\mathbf h} \bfg)$. 

By the strong multiplicity one theorem, $\pi$ is an irreducible automorphic representation of $GL_2(\mathbb A_F)$, in particular a tensor product of irreducible local representations $\pi_v$. Let $c$ be an adele such that the order of pole of $c$ at each place matches the order of vanishing of the meromorphic form $\omega_0$ and let $\psi'(z) =\psi(cz)$, so for each $v$, the maximal $\mathcal O_{F_v}$-lattice in $F_v$ on which $\psi'(z)$ is trivial is $\mathcal O_{f_v}$.

Consider the map that takes a function $f' \in \pi$ to a function \[ W_{\psi', f'} (\mathbf g) =  \int_{z \in \mathbb A_F/F} f'\left( \begin{pmatrix} 1 & z \\ 0 & 1\end{pmatrix}  \begin{pmatrix} c & 0 \\ 0 & 1\end{pmatrix} \bfg \right) \psi(-z) dz.\]

The image of this map is a space of functions isomorphic to $\pi$ and on which left translation by $\begin{pmatrix} 1 & z \\ 0 & 1\end{pmatrix}$, for $z \in \mathbb A_F$, acts as multiplication by $\psi(cz)$. Thus it is the tensor product of the Whittaker models $\mathcal W(\pi_v, \psi_v')$ of $\pi_v$, which are the unique spaces of functions stable under right translation, isomorphic to $\pi_v$ as representations of $GL_2( F_v)$, and on which left translation by unipotents acts by the additive character $\psi_v'$ of $F_v$. 

Now $f$ is invariant under $\Gamma_1(N)$. Hence it is a linear combination of products over $v$ of vectors invariant under $\left\{  \begin{pmatrix} a & b \\ c & d \end{pmatrix}  \in GL_2 ( \mathcal O_{F_v} ) \mid  c \equiv 0\mod N, d \equiv 1 \mod N \right \}$.  Moreover, $N$ is minimal such that there exists a vector of this form. It follows from \citep*[Theorem in (5.1)]{JPSS} that for each place $v$, the space of such vectors is one-dimensional, generated by the local newform, called the ``vecteur essential" by \cite*{JPSS}. So $f$ is a scalar multiple of the product of these vectors at each place.

Thus \[    \int_{z \in \mathbb A_F/F} f\left( \begin{pmatrix} 1 & z \\ 0 & 1\end{pmatrix}  \begin{pmatrix} c & 0 \\ 0 & 1\end{pmatrix}  \bfg \right) \psi(-z) dz = W_{\psi', f} ({\mathbf g}) = C \prod_{v} W_v(\mathbf g) \] where $W_v$ is the Whittaker function of the local newform at $v$, normalized so that $W_v(1)=1$, and $C$ is some constant.

Now by \citep*[Theorem in (4.1)]{JPSS} (whose proof is corrected in \citep[Theorem 1]{JacquetCorrection}, the function $W_v$ has the following properties:

\begin{enumerate}

\item For $a \in \mathcal O_{F_v}^\times $, \[  W_v \left( \bfg \begin{pmatrix} a & 0 \\ 0 & 1 \end{pmatrix} \right) = W_v( \bfg). \] 

\item We have \[  \int_{ a \in F_v^\times} W_v \left(  \begin{pmatrix} a & 0 \\ 0 & 1 \end{pmatrix} \right) |a|^{s-1/2} da  = L(s, \pi_v).\]

\end{enumerate}

Here the integral is taken against an invariant measure on $F_v^\times$ where $\mathcal O_{F_v}^\times$ has measure one. To obtain the second statement, one has to observe that the function $W$ defined in \cite[(3.3)]{JPSS} takes $a$ to $X_1^{\deg a}$ and plug in $X_1=1$.

 By \citep[Theorem VI.9]{Lafforgue}, there is associated to $\pi$ a two-dimensional representation $V$ of the Galois group of $F$, unramified outside the support of $N$, whose $L$-factors and $\epsilon$-factors agree with those of $\pi$. Because $V$ is a representation of $\pi_1 ( C - N)$, it defines a lisse sheaf on $C-N$. Let $\mathcal F$ be the middle extension of this lisse sheaf to $C$. By definition, the local $L$-factor of $V$ at the place $v\in |C|$ is  \[\frac{1}{ \det ( 1 - |\kappa_v|^{-s}   \operatorname{Frob}_{ |\kappa_v|}, \mathcal F_v)} = \sum_{n=0}^{\infty}  r_{\mathcal F} (n [v]) |\kappa_v|^{-ns} .\]
 
 Because $|a| =| \kappa_v|^{-\deg a}$, we can observe that the coefficient of $|\kappa_v|^{-ns}$ in the integral of (2) is simply the restriction of the integral to $a$ of degree $n$, where by (1) it takes a constant value, which must therefore be $ r_{\mathcal F} (n [v]) |a|^{1/2}$.
 
 So we have \[ W_v \left(  \begin{pmatrix} a & 0 \\ 0 & 1 \end{pmatrix} \right)  = r_{\mathcal F} (\deg a [v] )  |a|^{1/2} \] if $\deg a\geq 0$ and $0$ otherwise.
 
 Multiplying, for $a \in \mathbb A_F^\times,$ \[W_{\psi',f} \left(  \begin{pmatrix} a & 0 \\ 0 & 1 \end{pmatrix} \right)= C r_{\mathcal F} (\operatorname{div} a) |a|^{1/2} =  C r_{\mathcal F} (\operatorname{div} a) q^{-\deg (a)/2} .\]
 
By definition
 
 \[W_{\psi',f} \left(  \begin{pmatrix} a & 0 \\ 0 & 1 \end{pmatrix} \right) =  \int_{z \in \mathbb A_F/F} f\left( \begin{pmatrix} 1 & z \\ 0 & 1\end{pmatrix}  \begin{pmatrix} ac & 0 \\ 0 & 1 \end{pmatrix}    \right) \psi(-z)dz .\]
 
 Now observe that by Fourier analysis on $\mathbb A_F/ F$,
\[ f\left( \begin{pmatrix} a & z \\ 0 & 1\end{pmatrix}      \right)  = f\left( \begin{pmatrix} 1 & z \\ 0 & 1\end{pmatrix}  \begin{pmatrix} a & 0 \\ 0 & 1 \end{pmatrix}    \right) = \sum_{ w \in F}  \psi( w z)   \int_{z' \in \mathbb A_F/F} f\left( \begin{pmatrix} 1 & z' \\ 0 & 1\end{pmatrix}  \begin{pmatrix} a & 0 \\ 0 & 1 \end{pmatrix}    \right) \psi(-w z') dz. \]

Now, if $w=0$, the integral vanishes by cuspidality. For all other $w$, by left invariance under $GL_2(F)$, we have

\[ f\left( \begin{pmatrix} 1 & z' \\ 0 & 1\end{pmatrix}  \begin{pmatrix} a & 0 \\ 0 & 1 \end{pmatrix}    \right) =   f\left( \begin{pmatrix} w & 0 \\ 0 & 1 \end{pmatrix}   \begin{pmatrix} 1 & z' \\ 0 & 1\end{pmatrix}  \begin{pmatrix} a & 0 \\ 0 & 1 \end{pmatrix}    \right)  =  f\left(    \begin{pmatrix} 1 & wz' \\ 0 & 1\end{pmatrix}  \begin{pmatrix} wa & 0 \\ 0 & 1 \end{pmatrix}    \right)\]  so we obtain

\[ f\left( \begin{pmatrix} 1 & z \\ 0 & 1\end{pmatrix}  \begin{pmatrix} a & 0 \\ 0 & 1 \end{pmatrix}    \right) = \sum_{ w \in F^{\times} }  \psi( w z)   W_{\psi', f} \left(  \begin{pmatrix} wa/c & 0 \\ 0 & 1 \end{pmatrix}    \right)\]

\[=C \sum_{\substack{ w \in F^{\times} \\ \operatorname{div} w \omega_0 a  \geq 0}}  \psi(wz) r_{\mathcal F} (\operatorname{div} w \omega_0 a ) q^{ - \deg( w \omega_0 a)/2}. \]

  To calculate  $f\left( \begin{pmatrix} a & bz \\ 0 & b\end{pmatrix}      \right)$, we divide all entries by $b$, and use the definition of the central character, getting  
 \[ f\left( \begin{pmatrix} a & bz \\ 0 & b\end{pmatrix}  \right) =C \eta(b)^{-1} \sum_{\substack{ w \in F^{\times} \\ \operatorname{div}    (w \omega_0 a/b ) \geq 0}}  \psi(wz ) r_{\mathcal F} (\operatorname{div} (w  \omega_0 a/b) )  q^{ -\frac{  \deg( w \omega_0 a/b) }{2} }   . \] 
 Because $w$ is meromorphic, $\deg w=0$, so $\deg(w\omega_0 a/b) = \deg(\omega_0 a/b)$. Thus we have
  \[ f\left( \begin{pmatrix} a & bz \\ 0 & b\end{pmatrix}  \right) =C   q^{ -\frac{  \deg(  \omega_0 a/b) }{2} }  \eta(b)^{-1}  \sum_{\substack{ w \in F^{\times} \\ \operatorname{div}   (w \omega_0 a/b ) \geq 0}}  \psi(wz ) r_{\mathcal F} (\operatorname{div} (w  \omega_0 a/b) ) . \]

 We now describe why $\mathcal F$ has the stated properties. It is irreducible because the Galois representation $V$ is irreducible. It is pure of weight zero because its central character matches $\eta$ and hence has finite order. Its conductor is $N$ because the multiplicity of a place in the Artin conductor of a Galois representation is the exponent of the local $\epsilon$-factor of the Galois representation, which by Lafforgue's theorem matches the exponent in the local $\epsilon$-factor of the automorphic form, which is the level.
 
 The uniqueness of $\mathcal F$ follows from the fact that its trace function is determined by the Hecke eigenvalues of $\mathcal F$, and the uniqueness of $C$ is clear once $\mathcal F$ is fixed, because $f\neq 0$ so $f$ is not preserved by multiplication by any nontrivial scalar.
 
\end{proof}

\begin{lemma}\label{rankin-selberg-method} Let $f$ be a cuspidal newform of level $N$ with central character of finite order.

Let $\mu$ be the measure on $GL_2(F)  \backslash GL_2(\mathbb A_F) / \Gamma_1(N)$ that assigns mass to any double coset in \[GL_2(F)  \backslash GL_2(\mathbb A_F) / \Gamma_1(N)\] equal to $1$ over the order of its automorphism group.

Then for any $d \in \mathbb Z$,
 \[ \int_{\substack{  \bfg \in GL_2(F)  \backslash GL_2(\mathbb A_F) / \Gamma_1(N)\\\deg \det \bfg =d}} |f(\bfg) |^2 d\mu(\bfg) \] \[= |C_f|^2 \frac{   |J_C(\mathbb F_q) |} {1-q^{-1}}  q^{2g-3 + 2 \deg N }    L (1 , \operatorname{ad} \mathcal F )  \prod_{v | N}  \det( 1- q^{- \deg v} \Frob_{|\kappa_v|}, (\mathcal F \otimes \mathcal F^\vee)^{I_v} / (\mathcal F^{I_v} \otimes (\mathcal F^\vee)^{I_v} )  ).\]

\end{lemma}

\begin{proof} First observe that \[\sum_{D} |r_\mathcal F(D)|^2 q^{-s \deg D} = \prod_v \sum_n |r_\mathcal F(n[v])|^2 q^{-s n \deg v}\] and that if $v$ is unramified \[ \sum_n |r_\mathcal F(n[v])|^2 q^{ -s \deg v} = \frac{1- q^{-2s\deg v} }{ \det(1 - q^{-s\deg v } \Frob_{|\kappa_v|} , \mathcal F_v \otimes \mathcal F^\vee_v) }\] while if $v$ is ramified, \[ \sum_n |r_\mathcal F(n[v])|^2 q^{-s n \deg v}  = \frac{1 }{ \det(1 - q^{-s\deg v } \Frob_{|\kappa_v|},  \mathcal F_v \otimes \mathcal F^\vee_v) }.\] This means that, for $v$ unramified \[ \sum_n |r_\mathcal F(n[v])|^2 q^{-s n \deg v} \] is the local factor of the $L$-function \[ \frac{   \zeta_C(s)  L (s , \operatorname{ad} \mathcal F ) }{\zeta_C(2s) } .\] For $v$ ramified, the local factor of this $L$-function is \[ \frac{ 1 - q^{ -2 s \deg v}} { \det(1 - q^{-s \deg v} \Frob_{|\kappa_v|}, (\mathcal F \otimes \mathcal F^\vee)^{I_v})}\]  so $\sum_n |r_\mathcal F(n[v])|^2 q^{-s n \deg v} $ is equal to the local factor of this $L$-function times \[\frac{ \det( 1- q^{-s \deg v} \Frob_{|\kappa_v|}, (\mathcal F \otimes \mathcal F^\vee)^{I_v} / (\mathcal F^{I_v} \otimes (\mathcal F^\vee)^{I_v} )  } {1 -q^{-2s \deg v} }.\] Hence, the product of this local factor over all $v$ is
\[\sum_{D} |r_\mathcal F(D)|^2 q^{- s \deg D} = \prod_v \sum_n |r_\mathcal F(n[v])|^2 q^{ -s \deg v} \] \[=\frac{   \zeta_C(s)  L (s , \operatorname{ad} \mathcal F ) }{\zeta_C(2s) }\prod_{v | N} \frac{ \det( 1- q^{-s \deg v} \Frob_{|\kappa_v|}, (\mathcal F \otimes \mathcal F^\vee)^{I_v} / (\mathcal F^{I_v} \otimes (\mathcal F^\vee)^{I_v} )  } {1 -q^{-2s \deg v} }.\]

Taking a residue at $s=1$ (i.e. dropping the $\frac{1}{ 1- q^{1-s}}$ factor in $\zeta_C(s)$ and then substituting $1$ for $s$), we obtain
 \begin{equation}\label{l-function-limit} \lim_{n \to \infty} q^{-n}  \sum_{\substack {D \\ \deg D=n}} |r_{\mathcal F}(D)|^2 \end{equation} \[= \frac{ |J_C(\mathbb F_q)|  L (1 , \operatorname{ad} \mathcal F )   }{(1-q^{-1} ) q^g  \zeta_C(2) }    \prod_{v | N} \frac{ \det( 1- q^{- \deg v} \Frob_{|\kappa_v|}, (\mathcal F \otimes \mathcal F^\vee)^{I_v} / (\mathcal F^{I_v} \otimes (\mathcal F^\vee)^{I_v} )  } {1 -q^{- 2\deg v} }.\]

On the other hand, by Lemma \ref{Drinfeld-formula}, we have
 \[ f\left( \begin{pmatrix} a & bz \\ 0 & b\end{pmatrix}  \right) =C_f   q^{ -(\deg (\omega_0 a/b)/2} \eta(b)^{-1} \sum_{\substack{ w \in F^{\times} \\ \operatorname{div} (w\omega_0 a/b) \geq 0}}  \psi(wz ) r_{\mathcal F} (\operatorname{div} (w\omega_0 a/b)) . \]
 
 For  $w$ such that $\operatorname{div} (w a \omega_0/b) \geq 0$, we have $\psi(wz)=1$ as long as $z \in F$ or $z \in (a/b) \prod_v \mathcal O_{F_v}$. Thus the dual vector space to the space of $w$ such that $\operatorname{div} (w a \omega_0/b) \geq 0$ is $ \mathbb A_F / (F + (a/b) \prod_v \mathcal O_{F_v})$.  Furthermore note that the value of $f\left( \begin{pmatrix} a & bz \\ 0 & b\end{pmatrix}  \right)$ depends only on the equivalence class of $z$ modulo $F + (a/b) \prod_v \mathcal O_{F_v}$, since we can add an element of $F$ to $z$ through left multiplication by an upper unipotent in $GL_2(F)$ and add an element of $(a/b) \prod_v \mathcal O_{F_v}$ to $z$ through right multiplication by an upper unipotent in $\Gamma_1(N)$.
 
By the Plancherel formula, we have \[\sum_{z \in \mathbb A_F/ (F + (a/b) \prod_v \mathcal O_{F_v} )} \left|f\left( \begin{pmatrix} a & bz \\ 0 & b\end{pmatrix}  \right)\right|^2\] \[ = |C_f|^2 q^{ - \deg (\omega_0 a/b) +  \dim H^0(C, \div(\omega_0 a/b))} \sum_{\substack{ w \in F^{\times} \\ \operatorname{div} (w\omega_0 a/b) \geq 0}}  | r_{\mathcal F} (\operatorname{div} (w\omega_0 a/b)|^2 \]

For each $n$ congruent to $d$ mod $2$, let us fix $b_n \in \mathbb A_F$ of degree $(d-n)/2+g-1$. We will then sum over all classes $a \in F^\times\backslash \mathbb A_F^\times  / \prod_v \mathcal O_{F_v}^\times$ of degree $(n+d)/2+1-g$

This implies that $\deg (a/b_n) = n +2-2g$ and $\deg(ab_n) =d$. Taking $n$ sufficiently large, \[\dim H^0(C, \div{a} + \div \omega_0 - \div b_n) = n +1-g\] by Riemann-Roch and so we can simplify \[q^{ - \deg (\omega_0 a/b_n) +  \dim H^0(C, \div(\omega_0 a/b_n))} = q^{  -n + n+1-g} = q^{ 1-g}.\]

When we sum over $a \in F^\times\backslash \mathbb A_F^\times  / \prod_v \mathcal O_{F_v}^\times$ of degree $(n+d)/2+1-g$ and $w \in F^\times$,  each divisor of degree $n$ will occur $q-1$ times as $\operatorname{div} (w\omega_0 a/b_n)$, so we obtain

\[\sum_{\substack { a\in  F^\times\backslash \mathbb A_F^\times / \prod_v \mathcal O_{F_v}^\times\\ \deg a =(n+d)/2+1-g} } \sum_{z \in \mathbb A_F/ (F + (a/b) \prod_v \mathcal O_{F_v} )} \left|f\left( \begin{pmatrix} a & bz \\ 0 & b\end{pmatrix}  \right)\right|^2 =  q^{1-g} (q-1)  |C_f|^2 \sum_{\substack {D \\ \deg D=n}} |r_{\mathcal F}(D)|^2 \]

So this gives
\begin{equation}\label{rankin-selberg-fourier-step} \sum_{ \substack{ GL_2(F) \bfg \Gamma_1(N) \\ \deg \bfg = d}} |f(\bfg)|^2 \lim_{n \to \infty} q^{-n} \sum_{  \substack{ a\in    F^\times\backslash \mathbb A_F^\times /  \prod_v \mathcal O_{F_v}^\times\\ z \in \mathbb A_F/ (F + (a/b) \prod_v \mathcal O_{F_v})\\  \begin{pmatrix} a & b_n z \\ 0 & b_n\end{pmatrix}  \in GL_2(F) \bfg \Gamma_1(N)  }} 1= q^{1-g} (q-1) |C_f|^2   \lim_{n \to \infty}  \sum_{\substack {D \\ \deg D=n}} |r_{\mathcal F}(D)|^2.  \end{equation}

Here we have replaced the condition on the degree of $a$ with $\deg \bfg=d$, which is equivalent because $\deg a + \deg b_n = \deg \begin{pmatrix} a & b_n z \\ 0 & b_n\end{pmatrix} = \deg \bfg $.  Our next goal will be to evaluate the limit on the left hand side of \eqref{rankin-selberg-fourier-step}, which we will in particular show is independent of $\bfg$. This will allow us to relate the right hand side of \eqref{rankin-selberg-fourier-step} to $ \int_{\substack{  \bfg \in GL_2(F)  \backslash GL_2(\mathbb A_F) / \Gamma_1(N)\\\deg \det \bfg =d}} |f(\bfg) |^2 d\mu(\bfg)$, up to an explicit scalar.

To do this, we first observe that \[  \sum_{  \substack{ a\in    F^\times\backslash \mathbb A_F^\times /  \prod_v \mathcal O_{F_v}^\times\\ z \in \mathbb A_F/ (F + (a/b) \prod_v \mathcal O_{F_v})\\ \begin{pmatrix} a & b_n z \\ 0 & b_n\end{pmatrix}  \in GL_2(F) \bfg \Gamma_1(N)  }} 1 =\sum_{ \substack{\gamma \in \begin{pmatrix} * & * \\ 0 & 1 \end{pmatrix} \backslash GL_2(F)\\  \bfh \in \Gamma_1(N) / \begin{pmatrix} * & * \\ 0 & 1 \end{pmatrix} \\ \gamma \bfg \bfh = \begin{pmatrix}* & * \\ 0 & b_n \end{pmatrix}}} \frac{ (q-1)  }{  | GL_2(F) \cap \bfg \Gamma_1(N) \bfg^{-1}|}\] because $(a,z)$ determine $\gamma, \bfh$ up to the actions of upper-triangular matrices with a one in the bottom right, and vice versa $\gamma, \bfh$ determine $(a,z)$ up to multiplying $z$ by an element of $\mathbb F_q^\times$.

Furthermore we can count such elements  $\gamma  \in \begin{pmatrix} * & * \\ 0 & 1 \end{pmatrix} \backslash GL_2(F)$ by their bottom rows $(\gamma_1, \gamma_2)$. Such a row uniquely determines $\bfh$ modulo right multiplication, and a suitable $\bfh$ exists if and only if we can solve \[\begin{pmatrix} * & * \\  \gamma_1 & \gamma_2 \end{pmatrix} \bfg= \begin{pmatrix} * & * \\ 0 & b_n \end{pmatrix} \bfh^{-1} .\] This happens if, for each $v$ not in the support of $N$, the minimal valuation of the entries of \[ \begin{pmatrix} \gamma_1 & \gamma_2 \end{pmatrix} \bfg\] is $v(b_n)$, and for  each $v$ in the support of $N$, the first entry has valuation at least $v(b)+ m_v$ and the second is congruent to $b_n$ mod $\pi_v^{ v(b_n) + m_v}$.

In other words we are counting $\gamma_1, \gamma_2\in F$ such that
\[    \begin{pmatrix} \prod_{v \in N} \pi_v^{-m_v} &0 \\ 0 & 1 \end{pmatrix} b_n^{-1} \bfg^T \begin{pmatrix} \gamma_1 \\ \gamma_2 \end{pmatrix}  \] 
 is integral at every place, nondegenerate at unramified places, and satisfies a congruence condition at ramified places. By Riemann-Roch, the dimension of the space of $\gamma_1, \gamma_2$ in $F$ where this is integral at every place, goes, as $n$ and thus $\deg b_n^{-1}$ go to $\infty$, to 
 
 \[ \deg \det \left(  \begin{pmatrix} \prod_{v \in N} \pi_v^{-m_v} &0 \\ 0 & 1 \end{pmatrix} b_n^{-1} \bfg^T\right) +2 - 2g = - \deg N - 2 \deg b_n +  d+ 2 - 2g  \] \[   = - \deg N  - 2\left((d-n)/2 +g-1\right) + d  + 2-2g\] \[= n + 4 -4g - \deg N .\]
 
 By a sieve using Riemann-Roch, the total number of $\gamma_1, \gamma_2$ satisfying the conditions is $q^{ n + 4-4g - \deg N}$ times the product of local densities. The product of the local densities at the ramified primes is $ q^{ - \deg N}$, and the product of the local densities at the unramified primes is $\frac{1}{ \zeta_C(2)\prod_{v|N} (1- q^{-2 \deg v})}.$ Plugging this into Equation \eqref{rankin-selberg-fourier-step}, and combining with Equation \eqref{l-function-limit}, we obtain 
 \[ \frac{ q^{  4 -4g - 2\deg N } (q-1)  }{ \zeta_C(2)\prod_{v|N} (1- q^{-2 \deg v})} \int_{\substack{ \bfg\in  GL_2(F)  \backslash GL_2(\mathbb A_F) / \Gamma_1(N)\\ \deg \det \bfg=d}}   |f(\bfg)|^2 d\mu(\bfg) \] \[ = |C_f|^2q^{1-g} (q-1) \frac{ |J_C(\mathbb F_q)|  L (1 , \operatorname{ad} \mathcal F )   }{(1-q^{-1} ) q^g  \zeta_C(2) }    \prod_{v | N} \frac{ \det( 1- q^{- \deg v} \Frob_{|\kappa_v|}, (\mathcal F \otimes \mathcal F^\vee)^{I_v} / (\mathcal F^{I_v} \otimes (\mathcal F^\vee)^{I_v} )  } {1 -q^{- 2\deg v} }.\]

 Canceling like terms, we obtain
 \[ \int_{\substack{\bfg\in  GL_2(F)  \backslash GL_2(\mathbb A_F) / \Gamma_1(N)\\\deg \det \bfg=d}} |f(\bfg)|^2 d\mu(\bfg)\] \[= |C_f|^2 \frac{   |J_C(\mathbb F_q) |} {1-q^{-1}}  q^{2g-3 + 2 \deg N }    L (1 , \operatorname{ad} \mathcal F )  \prod_{v | N}  \det( 1- q^{- \deg v} \Frob_{|\kappa_v|}, (\mathcal F \otimes \mathcal F^\vee)^{I_v} / (\mathcal F^{I_v} \otimes (\mathcal F^\vee)^{I_v} )  ).\]
 
 \end{proof}

\begin{lemma}\label{rankin-selberg-normalized} Let $f$ be a cuspidal newform of level $N$ with central character of finite order. 

Let $\mu_{PGL_2}$ be the uniform measure on $PGL_2(\mathbb A_F)$ that assigns measure $1$ to the image of $\Gamma_1(N)$ inside $PGL_2(\mathbb A_F)$. Then
 \[ \int_{ \bfg \in PGL_2(F)  \backslash PGL_2(\mathbb A_F) } |f(\bfg) |^2 d\mu_{PGL_2}(\bfg) \] \[= 2 |C_f|^2   q^{2g-2 +  \deg N }    L (1 , \operatorname{ad} \mathcal F )  \prod_{v | N}  \det( 1- q^{- \deg v} \Frob_{|\kappa_v|}, (\mathcal F \otimes \mathcal F^\vee)^{I_v} / (\mathcal F^{I_v} \otimes (\mathcal F^\vee)^{I_v} )  )  / (1 -q^{- \deg v} ) .\]

\end{lemma}

\begin{proof}We will deduce this from Lemma \ref{rankin-selberg-method} by comparing the two integrals. 

We can divide $PGL_2(\mathbb A_F)$ into two components, one consisting of matrices whose determinant has odd degree and one consisting of matrices whose determinant has even degree. It suffices to show that the integral over each of these is \[ |C_f|^2   q^{2g-2 +  \deg N }    L (1 , \operatorname{ad} \mathcal F )  \prod_{v | N}  \det( 1- q^{- \deg v} \Frob_{|\kappa_v|}, (\mathcal F \otimes \mathcal F^\vee)^{I_v} / (\mathcal F^{I_v} \otimes (\mathcal F^\vee)^{I_v} )  )  / (1 -q^{- \deg v} ) .\]

First note that  \[ \int_{\substack{  \bfg  \in GL_2(F)  \backslash GL_2(\mathbb A_F) / \Gamma_1(N)\\\deg \det \bfg=d}} |f(\bfg)|^2 d\mu (\bfg)  =   \int_{   \substack{ \bfg \in GL_2(F)  \backslash GL_2(\mathbb A_F)  \\\deg \det \bfg=d}} |f(\bfg)|^2 d\mu_{GL_2} (\bfg) \] where $\mu_{GL_2}$ is the uniform measure that assigns mass $1$ to a right coset of $\Gamma_1(N)$. 

Let $\rho_d$ be the map from the degree $d$ part of $GL_2(F) \backslash GL_2(\mathbb A_F)$ to $PGL_2(F) \backslash PGL_2(\mathbb A_F)$. Let $\rho_{d*} \mu_{GL_2}$ be the pushforward of $\mu_{GL_2}$ along $\rho_d$. We have
\[  \int_{   \substack{ \bfg \in GL_2(F)  \backslash GL_2(\mathbb A_F)  \\\deg \det \bfg=d}} |f(\bfg)|^2 d\mu_{GL_2} (\bfg)  =   \int_{  \bfg \in PGL_2(F)  \backslash PGL_2(\mathbb A_F) } |f(\bfg)|^2 d (\rho_{d*} \mu_{GL_2} )(\bfg).\]

The pushforward $\rho_{d*} \mu_{GL_2}$ is right invariant by the degree zero elements of $GL_2(\mathbb A_F)$, which have two orbits on $PGL_2(\mathbb A_F)$, the even and odd components. Thus, if $d$ is even, then $\rho_{d*} \mu_{GL_2}$ is a uniform measure on the even component and zero on the odd component, and if $d$ is odd, $\rho_{d*}\mu_{GL_2}$ is a uniform measure on the odd component, and zero on the even component.

Furthermore, for $d=0$, restricted to the even component, $\rho_{0* }\mu_{GL_2}$  is equal to $\mu_{PGL_2}$ times the number of $\Gamma_1(N)$-cosets in  $\rho_0^{-1} (\rho_0( \Gamma_1(N))$. So the integral over $\mu_{PGL_2}$ over the even component is the integral evaluated in Lemma \ref{rankin-selberg-method}, divided by the number of such cosets. 
 This count of cosets is the number of degree zero elements of $| F^\times \backslash \mathbb A_F^\times / \mathbb A_F^\times \cap \Gamma_1(N)| $ and thus is $ \frac{ |J_C(\mathbb F_q)|  q^{\deg N} \prod_{v | n} (1- q^{-\deg v} ) }{q-1} $.
 
 Dividing the formula of Lemma \ref{rankin-selberg-method} by this quantity, we see that the integral over the even component is  \[ |C_f|^2   q^{2g-2 +  \deg N }    L (1 , \operatorname{ad} \mathcal F )  \prod_{v | N}  \det( 1- q^{- \deg v} \Frob_{|\kappa_v|}, (\mathcal F \otimes \mathcal F^\vee)^{I_v} / (\mathcal F^{I_v} \otimes (\mathcal F^\vee)^{I_v} )  )  / (1 -q^{- \deg v} ) .\]
 
 The integral over the odd component is given by the same formula, for the same reason.

\end{proof}

\begin{lemma}\label{good-L-value-bound} Assume $C$ admits a degree $d$ map to $\mathbb P^1$ defined over $\mathbb F_q$.

We have
\[ \frac{1}{   ( O ( \log (3g-3 +  2\deg N )))^{3d} } \leq  \left|   L(1,\operatorname{ad}(\mathcal F))  \right| \leq  ( O ( \log (3g-3 +  2\deg N )))^{3d} \]
where the constant depends only on $q$ and $d$.

\end{lemma}

\begin{proof} We have \[ \log L(1,\operatorname{ad}(\mathcal F)) = \sum_{n=0}^{\infty} \frac{ \tr (\Frob_q^n , H^1 (C_{\overline{\mathbb F}_q}, \operatorname{ad}(\mathcal F) )) q^{-n} }{n }.\]

We have the upper bound \begin{equation}\label{L-value-1} \tr (\Frob_q^n , H^1 (C_{\overline{\mathbb F}_q}, \operatorname{ad}(\mathcal F) )) \leq q^{n/2} \dim  H^1 (C_{\overline{\mathbb F}_q}, \operatorname{ad}(\mathcal F) ) \leq q^{n/2} (3g-3 + 2 \deg N) \end{equation} and the upper bound \begin{equation}\label{L-value-2} \tr (\Frob_q^n , H^1 (C_{\overline{\mathbb F}_q}, \operatorname{ad}(\mathcal F) )) = \sum_{x \in C(\mathbb F_{q^n} )} \tr (\Frob_{q^n}, \ad (\mathcal F)_x) \leq 3 | C( \mathbb F_{q^n})| \leq 3 d (q^n+1) \end{equation}

Let $k = \lfloor  2\log_q \frac{ 3g -3 +2 \deg N}{3d} \rfloor$. We use \eqref{L-value-1} for $n> k$ and \eqref{L-value-2} for $n \leq k$. The total contribution from \eqref{L-value-1} is \[   (3g-3 + 2 \deg N)  \sum_{n=k+1}^\infty \frac{q^{ -n/2} }{n} \leq\frac{ ( 3g-3 + 2 \deg N)  q^{ - (k+1)/2 } }{ (k+1) (1-1/\sqrt{q} ) } \leq  \frac{3 d}{ (k+1) (1-1/\sqrt{q})}.\] The total contribution from \eqref{L-value-2} is \[ \sum_{n=1}^k\frac{  3 d (1 + q^{-n}) }{n } \leq 3d \log (1+ q^{-1}) + \sum_{n=1}^k\frac{ 3d }{n}\] so the total contribution from both is \[ 3d \log (1+ q^{-1}) + \sum_{n=1}^k \frac{3d}{n} +  \frac{3 d}{ (k+1) (1-1/\sqrt{q})}    \leq O_{d,q} (1) +  \sum_{n=1}^k \frac{3d}{n}  \leq O_{d,q}(1) + 3d \log k  .\] The exponential of this is at most $O _{d,q} ( k^{3d} )= O_{d,q} ( ( \log_q ( 3g-3+ 2\deg N) )^{3d})$, and at least the inverse of that same term.

 \bibliographystyle{plainnat}

\bibliography{references}

\end{proof}

\tableofcontents
\end{document}